\documentclass[12pt]{amsart}
\usepackage[margin=1in]{geometry}
\usepackage{amsmath,amsfonts,amsthm,amssymb,bbm}
\usepackage{graphicx,color,dsfont}
\usepackage{enumitem}
\usepackage{fourier}

\newtheorem{theorem}{Theorem}
\newtheorem{lemma}{Lemma}
\newtheorem{proposition}{Proposition}

\newtheorem{corollary}{Corollary}

\newtheorem{claim}{Claim}

 \theoremstyle{definition}
 
 \theoremstyle{remark}

 \numberwithin{equation}{section}

\newcommand{\vertiii}[1]{{\left\vert\kern-0.25ex\left\vert\kern-0.25ex\left\vert #1
    \right\vert\kern-0.25ex\right\vert\kern-0.25ex\right\vert}}

\newcommand{\f}[2]{\frac{#1}{#2}}

\newcommand{\cl}{{\mathcal L}}

\newcommand{\al}{\alpha}
\newcommand{\be}{\beta}

\newcommand{\ga}{\gamma}

\newcommand{\de}{\delta}

\newcommand{\De}{\Delta}

\newcommand{\ka}{\kappa}
\newcommand{\la}{\lambda}

\newcommand{\si}{\sigma}

\newcommand{\vp}{\varphi}

\newcommand{\om}{\omega}

\newcommand{\rn}{{\mathbf R}^n}

\newcommand{\rone}{\mathbf R}
\newcommand{\rtwo}{\mathbf R^2}

\newcommand{\dpr}[2]{\langle #1,#2 \rangle}

\newcommand{\eps}{\epsilon}

\newcommand{\cq}{\mathcal Q}
\newcommand{\cp}{\mathcal P}

\newcommand{\p}{\partial}

\newcommand{\beq}{\begin{equation}}
\newcommand{\eeq}{\end{equation}}
\newcommand{\beqna}{\begin{eqnarray*}}
\newcommand{\eeqna}{\end{eqnarray*}}
\newcommand{\beqn}{\begin{equation*}}
\newcommand{\eeqn}{\end{equation*}}
\newcommand{\bp}{\begin{proof}}
\newcommand{\ep}{\end{proof}}
\newcommand{\bprop}{\begin{proposition}}
\newcommand{\eprop}{\end{proposition}}
\newcommand{\bt}{\begin{theorem}}
\newcommand{\et}{\end{theorem}}
\newcommand{\bex}{\begin{Example}}
\newcommand{\eex}{\end{Example}}
\newcommand{\bc}{\begin{corollary}}
\newcommand{\ec}{\end{corollary}}
\newcommand{\bcl}{\begin{claim}}
\newcommand{\ecl}{\end{claim}}
\newcommand{\bl}{\begin{lemma}}
\newcommand{\el}{\end{lemma}}

\begin{document}

\title[Sharp time decay rates for SQG and Boussinesq]
 {On the sharp time decay rates for  the 2D generalized quasi-geostrophic  equation and the Boussinesq system}

\thanks{ Hadadifard is supported in part by a graduate fellowship under NSF-DMS,    \# 1614734. Stefanov  is partially  supported by  NSF-DMS   under  \# 1614734.}

\author[Atanas Stefanov]{\sc Atanas G. Stefanov}
\address{ Department of Mathematics,
	University of Kansas,
	1460 Jayhawk Boulevard,  Lawrence KS 66045--7523, USA}
\email{stefanov@ku.edu}

 \author[Fazel Hadadifrad]{\sc Fazel Hadadifard}
 \address{ Department of Mathematics,
 	University of Kansas,
 	1460 Jayhawk Boulevard,  Lawrence KS 66045--7523, USA}
 \email{f.hadadi@ku.edu}



\date{\today}

\begin{abstract}
We compute the sharp time decay rates of the solutions of the IVP for quasi-geostrophic equation and the Boussinesq model, subject to fractional dissipation. Moreover, we explicitly  identify   the asymptotic profiles, the kernel of the $\alpha$ stable processes, which are analogues of the Oseen vortices.

\end{abstract}

\maketitle

\section{Introduction}
The initial value problem for the Navier-Stokes equation  
 \begin{equation}
 \label{10} 
 \left\{ 
 \begin{array}{l}
u_t+u \cdot \nabla u - \De u=\nabla p,  \ \ x\in \rn, t>0 \\
u(0,x):=u_0(x),  \nabla \cdot u=0
 \end{array}  
 \right.
 \end{equation}
 where $u$ is the fluid velocity and $p$ is the pressure, is ubiquitous and much studied model in the modern 
 PDE theory. Basic issues like global well-posedness remain elusively unresolved in spatial dimensions $n\geq 3$.  In the case of two spatial dimensions though, the problem is globally well-posed. This is mostly due to the following representation, which eliminates the pressure term  and leads  to equivalent {\it vorticity formulation}, 
 \begin{equation}
 \label{20} 
 \left\{ 
 \begin{array}{l}
 \om_t+u \cdot \nabla \om - \De \om=0,  \ \ x\in \rtwo, t>0 \\
 \om(0,x):=\om_0(x),  
 \end{array}  
 \right.
 \end{equation}
 where the vorticity $\om$, a  scalar quantity,  is given by  $\om=\p_1 u_2 - \p_2 u_1$. We denote $\nabla^\perp=\left(\begin{array}{c}
 -\p_2 \\  \p_1 \end{array}\right)$ for future reference, so that $\om=\nabla^\perp \vec{u}$. Many generalizations of this model have been considered, in particular to respond to modeling situations where the actual physical 
 dissipation is different than the one provided by the   Laplacian, in particular in large scale atmospheric models and large scale ocean modeling, see  \cite{AH, C1, JMWZ}. In particular,  we consider the following ``umbrella'' model 
  	\begin{eqnarray}\label{qq}
  	\begin{cases}
  	\partial_t z + u\cdot \nabla z
  	+   |\nabla|^{\al} z =0,\qquad x\in \mathbb{R}^2, \,\, t>0, \\ 
  	u = (|\nabla|^{\perp})^{- \be} z, \nabla \cdot u=0.
  	\end{cases}
  	\end{eqnarray}
  where $\al>1$ and $\be \geq  0$,  $(|\nabla|^{\perp})^{- \be} = \nabla^\perp m_{-\be-1}(|\nabla|)={\mathbf m}_{-\be}(\xi)$, where   ${\mathbf m}_a$ is a  symbol of order $a$, see Section \ref{sec:2.1} for precise definition\footnote{Note that it is a requirement  that $ m_{-\be-1}(|\nabla|)$ is a {\it radial symbol} of order $-\be-1$.}. These  type of equations frequently arise in fluid dynamics and as such, they  have been widely studied, especially so in the last twenty  years. We refer the reader to the works \cite{AH, B, C, C1, CC, GK, JMWZ, NS, W4, YJW} and references therein. 
   
    A few  examples, that we would like to emphasize as model cases, are as follows.    The 2D Fractional Navier-Stokes equation arises, if we take $z= \omega$ and $\be= 1$,
  \begin{equation}
  \label{w1}
  \om_t + u \cdot \nabla \om+ |\nabla|^{\al} \om = 0.
  \end{equation}expl
  If we let $z= \theta$ be the temperature of a flow, $\al> 1$ and $\be= 0$ the resulting equation is the so-called active scalar equation,
  \begin{equation}\label{w2}
  \theta_t + u \cdot \nabla \theta+ |\nabla|^{\al} \theta= 0,
  \end{equation}
  where $u_1=-R_2 \theta, u_2=R_1 \theta$, and $R_j, j=1,2$ are the Riesz transforms, given by the symbols $m_j(\xi)=i \f{\xi_j}{|\xi|}$.  
  
 The Boussinesq system, with general dissipations, reads
 \begin{eqnarray}\label{BSQ10}
 \begin{cases}
 \partial_t u + u\cdot \nabla u
 +  |\nabla|^{\si_1} u = - \nabla p+ \theta \vec{e_2},\qquad x\in \mathbb{R}^2, \,\, t>0,\\
 \partial_t \theta + u\cdot \nabla \theta
 +  |\nabla|^{\si_2} \theta =0,\qquad x\in \mathbb{R}^2, \,\, t>0,\\
 \nabla \cdot u=0.
 \end{cases}
 \end{eqnarray}
 where $u$ is the velocity of the fluid,  $\theta$ is its temperature, $p$ is the pressure and $\sigma_1, \sigma_2 > 0$ are the dissipation rates for the velocity and the temperature respectively, see \cite{AH, BS, C1, SH, JMWZ, W4, YJW} for background and various well-posedness results.  
 
 We consider  the equivalent vorticity formulation, with the usual scalar vorticity variable is given by  $\om=\partial_1 u_2- \partial_2 u_1$.    For the purposes of this work, we will only consider the diagonal case $\si_1=\si_2=\al$.  That is in vorticity formulation,  the system  consists of  the following coupled equations 
 \begin{eqnarray}\label{BSQ1}
 \begin{cases}
 \partial_t \om + u\cdot \nabla \om
 +  |\nabla|^{\al} \om = \partial_1 \theta,\qquad x\in \mathbb{R}^2, \,\, t>0,\\
 \partial_t \theta + u\cdot \nabla \theta
 +   |\nabla|^{\al} \theta =0,\qquad x\in \mathbb{R}^2, \,\, t>0,\\
 u= (\nabla^{\perp})^{- 1} \om,\ \nabla \cdot u=0.
 \end{cases}
 \end{eqnarray}

\subsection{Previous results}
   
As we have mentioned earlier, a lot of work has been done on the question of well-posedness, regularity of the solutions to these systems. We do not even attempt to overview the results, as this is only tangentially relevant for the current work, but the previously mentioned references contain lots of information about these issues.  As the purpose of this paper is to study the long time behavior of the said models, we discuss some recent works on the topic.   Most of the research has been done in the important (and classical) Navier-Stokes case in two and three dimensional cases. As the global regularity for this model remains a challenging open problem in 3D, some authors restricted themselves to weak solutions\footnote{which may be non-unique} or they considered eventual\footnote{that is, past eventual singularity formation} behavior of strong solutions. In this regard, we would like to reference the following works, \cite{C, GW, GW1, GK, MS, S2,S3, S4, SS1, SW1}. 

  In \cite{S3}, the author has exhibited lower time-decay bounds for the solutions, which match the upper bounds and are therefore sharp. The approach in \cite{GW, GW1}, for the same question, uses the method of the so-called scaled variables.  This was pioneered in \cite{GK, C}, although the idea really took of after the work   \cite{GW}.  It showed not only the optimal decay rates for the Navier-Stokes equation 
  ( this was actually previously established in \cite{BS}), but it provided an explicit asymptotic expansion of the solution, which explains the specific conditions on the initial data in  \cite{BS}, under which there are better decay rates. Recently, Goh and Wayne, \cite{GoW} have considered the Boussinesq model, with rapid rotation in 3D. They have shown, using the method of scaled variables, convergence to the Oseen vortex and associated leading order asymptotics. 
  
  In this paper, we follow this idea, to provide an explicit asymptotic expansion for the two models under consideration - the generalized quasi-geostrophic equation \eqref{qq} and the Boussinesq system (in vorticity formulation), \eqref{BSQ1}. Note that we work exclusively in two spatial dimensions. There are several reasons for this - 2D is the natural playground for \eqref{qq}, while the IVP for the Boussinesq system, the three (and higher) dimensional case, faces the same difficulties as the Navier-Stokes problem, namely absence of a global regularity theory. Moreover, we explore relatively low levels of dissipation, which in some sense, brings the global regularity theory to its limits, and we are still able to analyze the asymptotic behavior.   
  Another interesting feature that we deal with is the fractional dissipation. These have been studied  in the recent literature, but there are certain technical (and conceptual!)  difficulties associated with them, that we deal with by applying  advanced Fourier analysis methods. 
   \subsection{The scaled  variables}	We now introduce the scaled variables, for the models under consideration.  Basically, the method consists of introducing a new exponential time variable $\tau: e^\tau\sim t$ and the corresponding variables in $x$ are rescaled to accommodate this scaling, by keeping the linear part of the equation autonomous.  In this way, an algebraic decay in $t$ will manifest itself as an exponential decay in $\tau$. As is well-known, algebraic decays in time (especially non-integrable ones) are notoriously hard to propagate along  non-linear evolution equations, while any  (however small)  exponential decay, due to its integrability,  is more amenable to this type of analysis. 
   Here are the details. 
  \subsubsection{The scaled variables: the SQG equation}
  Consider the equation \eqref{qq}, and use the scaled  variables    to rewrite the variables in terms of
   \begin{equation}
   \xi= \f{x}{(1+ t)^{\f{1}{\al}}},\qquad \tau= \ln (1+ t).
   \end{equation}
      We define new functions $Z(\xi, \tau)$ and $U(\xi, \tau)$ correspond to $z(x, t)$ and $u(x, t)$ as follows:
   \begin{eqnarray}
   z(x, t)= \f{1}{(1+ t)^{1+\f{\be- 1}{\al}}} \ Z\left(\f{x}{(1+ t)^{\f{1}{\al}}}, \ln (1+ t)\right),\\ \label{sv}
   u(x, t)= \f{1}{(1+ t)^{1- \f{1}{\al}}} \ U\left(\f{x}{(1+ t)^{\f{1}{\al}}}, \ln (1+ t)\right). 
   \end{eqnarray}
   The choices of the parameters is clearly dictated by the stricture of the corresponding equation - the goal is to ensure an autonomous PDE in the new variables. Indeed, a straightforward calculation shows 
   \begin{eqnarray*}
   	z_t&=& \f{Z_{\tau}}{(1+t)^{2+\f{\be- 1}{\al}}}- \f{1}{\al}\ \f{1}{(1+t)^{2+\f{\be- 1}{\al}}} \f{x}{(1+ t)^{\f{1}{\al}}} \cdot \nabla_{\xi} Z- \f{1+\f{\be-1}{\al}}{(1+t)^{2+\f{\be- 1}{\al}}} Z,\\
   	|\nabla|^{\al} z&=& \f{1}{(1+t)^{2+\f{\be- 1}{\al}}} |\nabla|^{\al} Z, \\
   	u \cdot \nabla z &=& \f{1}{(1+t)^{2+\f{\be- 1}{\al}}} U \cdot \nabla_\xi Z.
   \end{eqnarray*}
   Hence,    $Z(\xi, \tau)$ satisfies the equation
   \begin{equation}\label{8}
   Z_{\tau}= \mathcal{L} Z- U \cdot\nabla_{\xi} Z
   \end{equation}
   where 
   \begin{equation}\label{lamb}
   \mathcal{L}Z= - |\nabla|^{\al} Z+ \f{1}{\al} \xi \cdot \nabla_{\xi} Z+ \left(1+\f{\be-1}{\al}\right)Z.
   \end{equation}
   Note that the relation $u=(|\nabla|^{\perp})^{- \be} z$ transforms into 
   $   U=(|\nabla|^{\perp})^{- \be} Z$. In addition, the property $\nabla\cdot u=0$ is retained, i.e.  $\nabla\cdot U=0$. 
   
   Next, we introduce the scaled variables for the Boussinesq system. 
  \subsubsection{The scaled variables: the Boussinesq  system}
Similar to the SQG case, we use the scaled variables 
$$
  \xi= \f{x}{(1+ t)^{\f{1}{\al}}},\qquad \tau= \ln (1+ t).
  $$
  We define new functions $W(\xi, \tau)$, $U(\xi, \tau)$ and $\Theta(\xi, \tau)$,  corresponding  to $\om(x, t)$,  $u(x, t)$ and $\theta(x, t)$ as follows 
  \begin{eqnarray*}
  \om(x, t) &=& \f{1}{(1+ t)} \ W\left(\f{x}{(1+ t)^{\f{1}{\al}}}, \ln (1+ t)\right)\\  
  u(x, t)&=& \f{1}{(1+ t)^{1- \f{1}{\al}}} \ U\left(\f{x}{(1+ t)^{\f{1}{\al}}}, \ln (1+ t)\right)\\
  \theta(x, t) &=&  \f{1}{(1+ t)^{2- \f{1}{\al}}} \ \Theta\left(\f{x}{(1+ t)^{\f{1}{\al}}}, \ln (1+ t)\right)
  \end{eqnarray*}
 Then, we calculate 
  \begin{eqnarray*}
  	\om_t&=& \f{W_{\tau}}{(1+t)^2}- \f{1}{\al}\ \f{1}{(1+t)^2} \f{x}{(1+ t)^{\f{1}{\al}}} \cdot \nabla_{\xi}  W- \f{1}{(1+t)^2} W,\\
  	|\nabla|^{\al} \om &=& \f{1}{(1+t)^2} \cdot |\nabla|^{\al} W,\
  	u \cdot \nabla \om= \f{1}{(1+t)^2} U \cdot \nabla W,\
  	\partial_1 \theta = \f{1}{(1+t)^2} \partial_1 \Theta.
  \end{eqnarray*}
  For the $\theta$ equation similar computation shows that
  \begin{eqnarray*}
  	\theta_t&=& \f{\Theta_{\tau}}{(1+t)^{3- \f{1}{\al}}}- \f{1}{\al}\ \f{1}{(1+t)^{3- \f{1}{\al}}} \f{x}{(1+ t)^{\f{1}{\al}}} \cdot \nabla_{\xi}  
  	 \Theta- \f{2-\f{1}{\al}}{(1+t)^{3- \f{1}{\al}}} \Theta,\\
  	|\nabla|^{\al} \theta&=& \f{1}{(1+t)^{3- \f{1}{\al}}}  |\nabla|^{\al} \Theta,\
  	u \cdot \nabla \theta= \f{1}{(1+t)^{3- \f{1}{\al}}} U \cdot \nabla \Theta.
  \end{eqnarray*}
 Therefore $W(\xi, \tau)$ and $\Theta(\xi, \tau)$ satisfy (with the $\cl$ defined above in \eqref{lamb}, but with $\be=1$) 
 \begin{eqnarray}
 \label{81}
 \begin{cases}
 W_{\tau}= \cl W- U \cdot\nabla_{\xi} W+ \partial_1 \Theta\\
 \Theta_{\tau}= (\cl+1-\f{1}{\al}) \Theta- (U \cdot\nabla_{\xi} \Theta)\\
 \end{cases}
 \end{eqnarray} 
 Clearly, the relations $\nabla \cdot u=0$ and $  u=(|\nabla|^{\perp})^{- 1} \om$ continue to hold for the capital letter variables as well, that is $\nabla\cdot U=0$ and $  U=(|\nabla|^{\perp})^{- 1} W$. In addition to the above equations we can define $p(x, t)= \f{1}{(1+ t)^{2- \f{2}{\al}}} \ P \left(\f{x}{(1+ t)^{\f{1}{\al}}}, \log (1+ t)\right)$ and find the following equation for $U(\xi, \tau)$, 
 \begin{equation}\label{UU}
 U_{\tau}= (\cl -\f{1}{\al}) U- (U \cdot\nabla_{\xi} U)- \nabla P+ \Theta \cdot e_2
 \end{equation}

  \subsection{Main results} 
  The main goal  of this work is to establish the sharp time decay rates of (various norms of) the solutions to \eqref{qq} and \eqref{BSQ1}. Our results actually  provide explicit asymptotic profiles, of which the precise asymptotic rates are a mere corollary. 
  
   Since it is clear that the equation for $\theta$ in \eqref{BSQ1} is basically\footnote{albeit  the relation of $u$ with $\theta$ is not a direct one, but through the vorticity $\om$}   \eqref{qq}, it is essential that we start with \eqref{qq}. This is the content of our first result, but in order to state it, we shall need to introduce a function $G: \hat{G}(p)=e^{-|p|^\al}$, see Section \ref{sec:g} for proper definitions and properties. This is a variant of the function $e^{-\f{|x|^2}{2}}$, or the Oseen vortex in the case $\al=2$. For the statement below, please refer to \eqref{weighted} for the definition of the weighted spaces $L^2(2)$. 
   \begin{theorem}(Global decay estimates for SQG)
   	\label{theo:10} 
   	Let $1< \al < 2$, and $\al+\be\leq 3$. Then, assuming that the initial data $z_0$ is in $L^2(2)\cap L^\infty$, the Cauchy problem \eqref{qq} has a unique, global solution in  $L^2(2)\cap L^\infty$. Moreover,    for all $\eps>0$, there is a constant $C=C_{\al, \be, \eps}$  and for all $p\in [1,2]$ and $t\geq 0$,  
   	\begin{equation}
   	\label{14}
  \|z(t,\cdot)- \f{\int_{\rtwo} z_0(x) dx}{(1+t)^{\f{2}{\al}}} G\left(\f{\cdot}{(1+t)^{\f{1}{\al}}}\right)  \|_{L^p} \leq  
   	\f{C_{\al, \be, \eps} \|z_0\|_{L^2(2)\cap L^\infty}}{(1+t)^{\f{3}{\al}-\f{2}{\al p}-\eps}}.
   	\end{equation} 
   	Moreover, if $\be>1$, we have that \eqref{14} holds for the full range of indices $1\leq p<\infty$. 
   	
   	For generic initial data,  that is $\int_{\rtwo} z_0(x) dx\neq 0$, we have 
   	$$
   	\|z(t, \cdot)\|_{L^p}\sim (1+t)^{-\f{2(p-1)}{\al p}}, \ \ 1\leq p\leq 2.
   	$$
   	which extends to all $1\leq p<\infty$, provided $\be>1$. 
   \end{theorem}
  {\bf Remarks:} 
  \begin{itemize}
  	\item Our results extend those in \cite{CW}, as they provide an upper bound for the time decay, for weak solutions of the SQG. 
  	\item In \cite{GW, GW1}, the authors go one step further in deriving explicitly the next order asymptotic profiles. The analysis required for this step is performed in higher order weighted $L^2$ space. This cannot be done in this framework, since the function $G$ does not belong to the next order weighted space, namely $L^2(3)$, see Proposition \ref{lem1}. This is in sharp contrast with the case $\al=2$, considered in  \cite{GW, GW1}, where the function is in Schwartz class. 
  	\item Related to the previous point, we need to address a problem, where the function $G$ and the heat kernel of the semigroup $e^{\tau \cl}$ have limited decay at infinity. Thus, any attempt to use the dynamical system approach in \cite{GW} to construct stable manifolds faces serious obstacles. We take a  different  approach to the problem  in that we use {\it a priori} estimates and estimates on the evolution operator to establish the asymptotic decomposition. 
  \end{itemize}
  Our next result concerns \eqref{BSQ1}. 
  \begin{theorem}(Global decay estimates for Boussinesq)  
  	\label{theo:20} 
  	Let $\al \in (1, \f{3}{2})$. Consider the Cauchy problem for \eqref{BSQ1}, with initial data  $w_0, \theta_0 \in Y:=L^2(2)\cap L^\infty\cap H^1(\rtwo)$. Then, the Cauchy problem \eqref{BSQ1} is globally well-posed in $Y$ - that is for every $t>0$, the solution $(w(t), \theta(t))\in Y\times Y$. 
  	
  	 Moreover, for every $\de>0$, there exists $C=C(\al, \de, \|w_0\|_Y, \|\theta_0\|_Y)$, so that for all $p\in [1,2]$ and for all   $t>0$, 
  	\begin{eqnarray}
  	\label{state1}
  & & 	\|w(t, \cdot)- \f{\ga_2(0)}{(1+t)^{\f{3}{\al}-1}} \p_1 
  	G\left(\f{\cdot}{(1+t)^{\f{1}{\al}}}\right) - 
 	 \f{\ga_1(0)}{(1+t)^{\f{2}{\al}}}   G\left(\f{\cdot}{(1+t)^{\f{1}{\al}}} \right)  \|_{L^p} \leq  
 	\f{C_{\al,  \de} \|(w_0, \theta_0)\|_{Y}}{(1+t)^{\f{6}{\al}-3-\f{2}{\al p}-\de}}, \\
 & & 		\|\theta(t, \cdot)- \f{\ga_2(0)}{(1+t)^{\f{2}{\al}}}   
 G\left(\f{\cdot}{(1+t)^{\f{1}{\al}}}\right)    \|_{L^p} \leq  
 \f{C_{\al,  \de} \|(w_0, \theta_0)\|_{Y}}{(1+t)^{\f{5}{\al}-2-\f{2}{\al p}-\de}}, 
  	\end{eqnarray}
  	where $\ga_1(0)= \int_{\rtwo} w_0(x) dx,  \ga_2(0)= \int_{\rtwo} \theta_0(x) dx$.  In particular, if $\ga_2(0)\neq 0$, we have 
  	$$
  	\|w(t, \cdot)\|_{L^p}\sim \f{1}{(1+t)^{\f{3}{\al}-1-\f{2}{\al p}}}, 	\|\theta(t, \cdot)\|_{L^p}\sim \f{1}{(1+t)^{\f{2}{\al} - \f{2}{\al p}}}, 
  	$$
  \end{theorem}
{\bf Remarks:} 
\begin{itemize}
	\item As in Theorem \ref{theo:10}, the results can be extended to provide asymptotic expansions for $w, \theta$ in 
	the norms $L^p, p\in (2, \infty)$, with the exact same statement.   
	\item Note that the  decay rate $(1+t)^{1-\f{3}{\al}}$ in the expression for $w$ is dominant over 
	$(1+t)^{-\f{2}{\al}}$. 
	\item For $\al\in(\f{4}{3}, \f{3}{2})$, the correction term 
	$ \f{\ga_1(0)}{(1+t)^{\f{2}{\al}}}   G\left(\f{\cdot}{(1+t)^{\f{1}{\al}}}\right)$ is faster decaying than the error term and we can state the result as follows 
	$$
		\|w(t, \cdot)- \f{\ga_2(0)}{(1+t)^{\f{3}{\al}-1}} \p_1 
		G\left(\f{\cdot}{(1+t)^{\f{1}{\al}}}\right)  \|_{L^p} \leq  
		\f{C_{\al,  \de} \|(w_0, \theta_0)\|_{Y}}{(1+t)^{\f{6}{\al}-3-\f{2}{\al p}-\de}},
	$$
	\end{itemize}
	
  The paper is organized as follows. In Section \ref{sec:2.1}, we introduce some basic Sobolev spaces, weighted $L^2$ spaces  and some relevant  estimates that will be useful in the sequel. In Section \ref{sec:3}, we study the operator $\cl$ - we establish the basic structure of its spectrum, as well as an explicit form of the semigroup $e^{\tau \cl}$. The semigroup is shown to act boundedly on certain weighted $L^2$ spaces.  This is  helpful for the study of the non-linear evolutions problem, but it also helps us identify the spectrum, through the Hille-Yosida theorem, see Section \ref{sec:3.6}. In Section \ref{sec:4}, we develop the local and global well-posedness theory for the generalized quasi-geostrophic equation, both in the original variables and then in the scaled variables. This is done via  standard    energy estimates methods.  Even at this level, the optimal decay estimates start to emerge, in the scaled variables context\footnote{But at this point, we cannot yet conclude the optimality of these estimates, as we are missing an estimate from below.}. Our asymptotic results for the quasi-geostrophic model  are in Section \ref{sec:6}. In it, we use the {\it a priori} information from Section \ref{sec:4}, together with new estimates for the Duhamel's operator   to derive the precise asymptotic profiles for the solutions. 
   For the Boussinesq system, we provide the necessary local and global well-posedness theory in Section \ref{sec:5}. Some of these results are basic and could have been recovered from earlier publications.  Others provide new  {\it a piori} estimates for the scaled variables system, which are used in Section \ref{sec:7}. In Section \ref{sec:7}, we provide the proof of our main result about the precise asymptotic profiles for the Boussinesq evolution.  \\
   {\bf Acknowledgement:} The authors wish to thank Ryan Goh and Jiahong Wu for stimulating discussions regarding these topics.

  \section{Preliminaries}

  \subsection{Fourier Transform, function spaces and mulitpliers} 
  \label{sec:2.1} 
  The Fourier transform and its inverse are taken in the form 
  $$
  \hat{f}(p)= \int_{\rn} f(x) e^{-i x\cdot p} dx, \ \ f(x) = (2\pi)^{-n} \int_{\rn} \hat{f}(p) e^{i x\cdot p} dp
  $$
  Consequently, since $\widehat{-\De f}(p)= |p|^2 \hat{f}(p)$, we define the operators $|\nabla|^a:=(-\De)^{a/2}, a>0$, via its action on the Fourier side $\widehat{|\nabla|^a f}(p)= |p|^a \hat{f}(p)$. More generally, the operators 
  $f(|\nabla|)$, for reasonable functions $f$,  are  acting as multipliers by $f(|p|)$.  
  We also make use of the following notation - {\it we say that ${\mathbf m}$ is a symbol of order $a, a\in \rone$}, if it is a  smooth function on $\rn\setminus \{0\}$, satisfying for all multi-indices $ \al \in {\mathbf N}^n$, 
  $$
  |\p^\al {\mathbf m}(\xi)|\leq C_\al |\xi|^{a-|\al|}. 
  $$
  It is actually enough to assume this inequality for a finite set of indices, say $|\al|\leq n$.  The prototype will be something of the form $m(\xi)=|\xi|^a$, but note that $a$ will be often negative in our applications.  We  schematically denote a symbol of order $a$ by ${\mathbf m}_a$. 
  
  The $L^p$ spaces are defined by the norm   $  \|f\|_{L^p}= \bigg( \int |f(x)|^p\ dx\bigg)^{\f{1}{p}}$, while  the weak 
  $L^p$ spaces are 
  \begin{equation*}
  L^{p, \infty}= \left\{f: \|f\|_{L^{p, \infty}}=\sup_{\la>0} \bigg \{\la \ |\{x: |f(x)|> \la\} |^{\f{1}{p}} \bigg\} < \infty \right\}.
  \end{equation*}
  In this context, recall the  Hausdorff--Young inequality which reads as follows: For 
  $p,  q, r \in (1, \infty)$ and $1+ \f{1}{p}= \f{1}{q}+ \f{1}{r}$  
  \begin{equation*}
  \|f * g\|_{L^p} \leq C_{p,q,r} \|f\|_{L^{q, \infty}} \|g\|_{L^r}.
  \end{equation*} 
  
  For an integer $n$ and $p\in (1, \infty)$,  the Sobolev spaces are the closure of the Schwartz functions in the norm $\|f\|_{W^{k, p}}= \|f\|_{L^p}+ \sum_{|\al| \leq k} \|\partial^{\al} f\|_{L^p}$,  while for a non-integer $s$ one takes   
  \begin{equation*}
  \|f\|_{W^{s, p}}= \|(1- \De)^{s/2} f\|_{L^p}\sim \|f\|_{L^p}+ \||\nabla|^s f\|_{L^p}.
  \end{equation*}
  The Sobolev embedding theorem states $\|f\|_{L^p(\rn)} \leq C \||\nabla|^s f\|_{L^q(\rn)}$, where $1<p<q<\infty$ and 
  $n (\f{1}{p}- \f{1}{q})= s$, with the usual modification for $p=\infty$, namely 
  $\|f\|_{L^\infty(\rn)} \leq C_s \|  f\|_{W^{s,q}(\rn)}$, $s> \f{n}{p}$.  In particular, an estimate that will be useful for us, is 
    \begin{equation}
    \label{30}
    \|(|\nabla|^{\perp})^{- \be} f\|_{L^p}\leq C \|f\|_{L^q}, \ \ 1<p<q<\infty,  \be=n (\f{1}{q}- \f{1}{p})
    \end{equation}
  This follows from the Mikhlin's criteria for $L^p, 1<p<\infty$ boundedness. Sometimes, we use the following replacement of \eqref{30}, when $p=\infty$ and $\be<n$, 
   \begin{equation}
   \label{31} 
  \|(|\nabla|^{\perp})^{- \be} f\|_{L^\infty}\leq C_\eps (\|f\|_{L^{\f{n}{\be}+\eps}}+\|f\|_{L^{\f{n}{\be}-\eps}}).
  \end{equation}
  We  provide a proof for this inequality in Appendix \eqref{2.2.2}. Note that these estimates hold in a more general setting, when $|\nabla|^{\perp})^{- \be}$ is replaced by an arbitrary symbol of order $-\be$, that is 
  \begin{equation}
  \label{32} 
  \|m_{-\be}(\nabla)  f\|_{L^\infty}\leq C_\eps (\|f\|_{L^{\f{n}{\be}+\eps}}+\|f\|_{L^{\f{n}{\be}-\eps}}).
  \end{equation}
  Another useful ingredient will be  the Gagliardo - Nirenberg interpolation inequality,
  \begin{equation*}
  \| |\nabla|^s f\|_{L^p} \leq \||\nabla|^{s_1} f\|^{\theta}_{L^q} \||\nabla|^{s_2} f\|^{1- \theta}_{L^r},
  \end{equation*}
  where   $s= \theta s_1+ (1- \theta) s_2$  and $\f{1}{p}= \f{1}{q}+ \f{1}{r}$.    
  
  For the optimal decay rates, we will   need to argue in the weighted spaces. For any $m \geq 0$ we define the Hilbert space $L^2(m)$ as follow
  \begin{equation}
  \label{weighted}
  L^2(m)= \bigg\{ f \in L^2 : \  \|f\|_{L^2(m)}=  \bigg( \int_{\rtwo} (1+ |x|^2)^m |f(x)|^2 dx\bigg)^{\f{1}{2}} < \infty \bigg\}
  \end{equation}
  One can show by means of H\"older's,  $L^2(2) \hookrightarrow L^p(\rtwo)$, whenever 
  $1 \leq  p \leq 2$.
  
  \subsection{The fractional Laplacian}
 
  First, we record the following kernel representation formula for negative powers of Laplacian. This is nothing, but a fractional integral - for $a\in (0,2)$, 
  \begin{equation}
  \label{921}
    |\nabla|^{-a} f(x) = c_a \int_{\rtwo} \f{f(y)}{|x-y|^{2-a}} dy.
  \end{equation}
For positive powers, we have a similar formula -  for  $ a\in (0,2)$, 
  $$
  |\nabla|^a f (x) = C_a p.v. \int_{\rtwo} \f{f(x)-f(y)}{|x-y|^{2+a}} dy. 
  $$
   see Proposition 2.1, \cite{CC}). 
   Next, we have the following result, due to Cordoba-Cordoba. 
   \begin{lemma}(Lemma 2.4, 2.5, \cite{CC}, Theorem 2, \cite{CL})
   	\label{le:90}
   	For $p: 1\leq p<\infty$, $a\in [0,2]$ and $f\in W^{a,p}(\rtwo)$, 
 \begin{equation}
 \label{a:20} 
  	\int_{\rtwo} |f(x)|^{p-2} f(x) [|\nabla|^a f](x) dx\geq 0.
 \end{equation}
   	If in addition, $p=2^n, n=1, 2, \ldots$, there is the stronger coercivity estimate 
   	\begin{equation}
   	\label{a:30} 
   	\int_{\rtwo} |f(x)|^{p-2} f(x) [|\nabla|^a f](x) dx\geq  \f{1}{p} \||\nabla|^{\f{a}{2}}[f^{\f{p}{2}}]\|_{L^2(\rtwo)}^2.
   	\end{equation}
   	Finally, for $p\in [1, \infty)$, $a\in (0,2)$, 
   	\begin{equation}
   	\label{a:31} 
   	\int_{\rtwo} |f(x)|^{p-2} f(x) [|\nabla|^a f](x) dx\geq  \f{1}{p} \|f\|_{L^{\f{2p}{2-a}}(\rtwo)}^2.
   	\end{equation}
   \end{lemma}

  \subsection{The function $G$} 
  \label{sec:g}
  The function $G$ defined by $\hat{G}(p)=e^{-|p|^\al}, p\in \rtwo$ will be used frequently in the sequel. We list and prove some important properties. 
  
  	\begin{lemma}
  		\label{le:10} 
  		For any $p \in [2, \infty]$ and $\al\in (1,2)$, 
  		\begin{equation}\label{GG}
  		(1+ |\xi|^2)\ G(\xi), (1+ |\xi|^2)\nabla G(\xi) \in L_{\xi}^p
  		\end{equation} 
  		In particular, $G, \nabla G\in L^1(\rtwo)\cap L^\infty(\rtwo)$. 
  	\end{lemma}
  	{\bf Note:} For $\al\in (1,2)$, the function $G$ does not belong to $L^2(3)$, due to the lack of smoothness of $\hat{G}$ at zero (or what is equivalent to the lack of decay of $G$ at $\infty$). 
  	\begin{proof}
  		For the $L^2$ estimate, $\|G\|_{L^2}= \|\hat{G}\|_{L^2}<\infty$.  Since $\widehat{G}$ is a radial function
  		\begin{eqnarray*}
  			\| |\xi|^2 G(\xi) \|_{L^2}= \| \De_p  \widehat{G}(p) \|_{L^2}= \| \De_p e^{- |p|^{\al}} \|_{L^2}= \| (\p_{\rho\rho}+ \f{1}{\rho} \partial_{\rho}) (e^{- \rho^{\al}})\|_{L^2(\rho d\rho)}.
  		\end{eqnarray*}
  		But, 	$(\p_{\rho \rho}+ \f{1}{\rho} \partial_{\rho}) (e^{- \rho^{\al}})= - \al(\al- 1) \rho^{\al- 2} e^{- \rho^{\al}}+ 
  		\al^2 \rho^{2(\al- 1)} e^{- \rho^{\al}}. $
  		Therefore,
  		$  			\| |\xi|^2 G(\xi) \|_{L^2}^2 \leq   I_1+ I_2$, where $I_1=\|\rho^{\al- 2} e^{- \rho^{\al}}\|_{L^2(\rho d\rho)}^2$, 
  		$I_2=\|\rho^{2(\al- 1)} e^{- \rho^{\al}}\|_{L^2(\rho d\rho)}^2$. We have  
  		\begin{eqnarray*}
  			I_1 \leq  \int_0^{1} \f{1}{\rho^{2(2- \al)- 1}}\ d \rho+ \int_1^{\infty} \rho^{2 (\al- 2)+1} e^{- 2 \rho^{\al}} \rho\ d \rho.
  		\end{eqnarray*}
  		Since $2(2- \al)- 1 < 1$,  the first term is bounded. The second term is also bounded by the exponential decay, whence $I_1$ is bounded. The second term, $I_2=\|\rho^{2(\al- 1)} e^{- \rho^{\al}}\|_{L^2(\rho d\rho)}^2$ is also bounded - no singularity at zero and exponential decay at $\infty$. This proves the $L^2$ estimate.   	
  			
  		For the $L^{\infty}$ estimate we can use the Hausdorf-Young's to bound $\|G\|_{L^\infty}\leq \|\hat{G}\|_{L^1}<\infty$. Similarly, 
  		\begin{eqnarray*}
  			\||\xi|^2  G(\xi)\|_{L^{\infty}} &\leq& \|\De_p \widehat{G}(p)\|_{L^1} \leq \al(\al- 1) \int_0^{\infty} \rho^{\al- 2} e^{- \rho^{\al}} \rho d \rho+ \al^2 \int_0^{\infty} \rho^{2(\al- 1)} e^{- \rho^{\al}} \rho d \rho\\ 
  			&\leq& \al(\al- 1) \int_0^{\infty} \rho^{\al- 1} e^{- \rho^{\al}}  d \rho+ \al^2 \int_0^{\infty} \rho^{2\al- 1} e^{- \rho^{\al}}  d \rho<\infty. 
  		\end{eqnarray*}
  		Now the interpolation between $L^2$ and $L^{\infty}$  yields $	(1+ |\xi|^2)\ G(\xi) \in L_{\xi}^p, 1\leq p\leq \infty$.  
  		
  		Regarding the claims about $\nabla G$, it is easy to see that  $\||\xi|^2 \nabla G\|_{L^2}=\|\De_p[p e^{-|p|^\al}]\|_{L^2}<\infty$. Indeed, the last conclusion follows easily from an identical argument as the one above, as the central issue was the singularity at zero for $ \|\De_p e^{-|p|^\al}\|_{L^2}$. Now the situation is better as we multiply by $p$, which actually alleviates the singularity at zero. Similar is the argument about $\||\xi|^2 \nabla G\|_{L^\infty}$, we omit the details. 
  	\end{proof}
  	The following Lemma will be used frequently in the next sections - it is an easy consequence of the Hausdorff-Young's inequality.  
  	\begin{lemma}\label{aa}
  		Let $\al > 0$, then for any $t > 0$ and $1 \leq p \leq \infty$,
  		\begin{eqnarray}
  		\label{2.7}
  		\|e^{- t |\nabla|^{\al}} f\|_{L^p} &\leq& C \|f\|_{L^p}\\
  		\label{2.8} 
  		\|e^{- t |\nabla|^{\al}} \nabla f\|_{L^p} &\leq& C t^{-\frac{1}{\al}} \|f\|_{L^p}
  		\end{eqnarray}
  	\end{lemma}
  	
  	\begin{proof}
  Clearly, 
  		\begin{equation*}
  		e^{- t |\nabla|^{\al}} f= \int G_t(x- y) f(y) dy
  		\end{equation*}
  		where $\widehat{G_t}(p)= \widehat{G} (t^{\frac{1}{\al}} p)$. Then $\|e^{- t |\nabla|^{\al}} f\|_{L^p} \leq \|G_t\|_{L^1} \|f\|_{L^p}=C \|f\|_{L^p}$, where $C=\|G\|_{L^1(\rtwo)}$. 
  		\begin{equation*}
  		\|e^{- t |\nabla|^{\al}} \nabla f\|_{L^p} =t^{-\frac{1}{\al}}  \bigg\|\int \nabla G (t^{-\f{1}{\al}}(\cdot-y)) f(y)  dy \bigg\|_{L^p} \leq C t^{-\frac{1}{\al}}  \|f\|_{L^p}, 
  		\end{equation*}
  		where $C=\|\nabla G\|_{L^1(\rtwo)}$. 
  	\end{proof}
  \subsection{Kato-Ponce and commutator estimates }
  The classical by now product rule estimate, usually attributed to Kato-Ponce can be stated  as follows. 
  \begin{lemma}
  	\label{kp}
  	Let $a\in (0,1)$ and $1<p,q,r<\infty$, so that $\f{1}{p}=\f{1}{q}+\f{1}{r}$. Then, there exists $C=C_{p,q,r,a}$
  	$$
  	\||\nabla|^a[f g]\|_{L^p}\leq C_{p,q,r,a} (\| |\nabla|^a f\|_{L^q} \|g\|_{L^r}+\| |\nabla|^a g\|_{L^q} \|f\|_{L^r} )
  	$$
  \end{lemma}
  We also make use of the following Lemma from \cite{SH}.
  	\begin{lemma}
  		\label{L_-com1}
  		Let $s_1, s_2$ be two reals so that $0\leq s_1$ and  $0\leq s_2-s_1\leq 1$. Let  $p,q,r$ be related via the H\"older's  $\f{1}{p}=\f{1}{q}+\f{1}{r}$, where 
  		$2<q<\infty$, $1<p,r<\infty$. Finally, let $\nabla \cdot V=0$.   
  		Then for any $a\in [s_2-s_1, 1]$
  		\begin{eqnarray}
  		\label{201}
  		& & \||\nabla|^{-s_1}[|\nabla|^{s_2}, V \cdot \nabla] \vp\|_{L^p}\leq C \||\nabla|^a V\|_{L^q} \||\nabla|^{s_2-s_1+1-a}\vp\|_{L^r}
  		\end{eqnarray}
  		In addition, we have the following end-point estimate. For $s_1>0, s_2>0, s_3>0$ and $s_1<1, s_3<1, s_2<s_1+s_3$, there is
  		\begin{equation}
  		\label{25}
  		\||\nabla|^{-s_1}[|\nabla|^{s_2}, |\nabla|^{-s_3} V\cdot \nabla]\vp\|_{L^2}\leq C \|V\|_{L^\infty} \||\nabla|^{s_2-s_1+1-s_3}\vp\|_{L^2}. 
  		\end{equation}
  	\end{lemma}
  	\subsection{A variant of the Gronwall's inequality}
  	We shall need a version of the Gronwall's inequality as follows. 
  	\begin{lemma}
  		\label{gro}
  		Let $\si>0, \mu>0, \ka>0$ and $a\in [0,1)$. Let $A_1,A_2, A_3$ be three positive constants so that a function $I:[0, \infty)\to \rone_+$ satisfies  $I(\tau)\leq A_1 e^{-\ga \tau}$, for some real $\ga$  and 
  	\begin{equation}
  	\label{gron} 
	I(\tau)\leq A_2 e^{-\mu \tau} + A_3 \int_0^\tau \f{e^{-\si(\tau-s)}}{(\min(1, |\tau-s|)^a} e^{-\ka s} I(s) ds.
  	\end{equation}
  		Then, there exists $C=C_{a, \si, \mu, \ka, \ga}$, so that 
  		$$
  		I(\tau)\leq C_{a, \si, \mu, \ka, \ga}(1+|A_1|+|A_2|+|A_3|)  e^{- \mu \tau}.
  		$$
  	\end{lemma}
  	The proof of Lemma \ref{gro} is rather elementary, but we  provide it for completeness  in the Appendix. 
  
  \section{The operator $\cl$ in $L^2(2)$: spectral analysis and  semigroup estimates}

  \subsection{Spectral theory for $\cl$}
  \label{sec:3} 
  The following result discusses the spectrum of $\cl$. 
  \begin{proposition}
  	\label{prop:10} 
  	Let $\mathcal{L}$ be as defined in \eqref{lamb}, then 
  	\begin{enumerate}
  		\item  \emph{The discrete spectrum:} Let $k \in \mathbb{N} \cup \{0\}$ be fixed and $\sigma= (\sigma_1, \sigma_2)$ be such that $|\sigma|= \sigma_1 + \sigma_2= k$ .   Then the function $\phi_{\sigma} (\xi)$ defined by 
  		\begin{equation}
  		\phi_{\sigma}(\xi)=  \partial_1^{\sigma_1} \partial_2^{\sigma_2} G, 
  		\end{equation}
  		is an eigenfunction of the multiplicity greater or equal to ${k+ 1\choose k}= k+1$ related to the eigenvalue $\la_{k}= 1-\f{3-\be+k}{\al} $. In fact, 
  		$$
  		\sigma_d(\mathcal{L})\supseteq  \bigg\{\la_k \in \mathbb{C}: \la_k= 1-\f{3-\be+k}{\al};  k= 0, 1, 2, \cdots \bigg\}.
  		$$
  			\item \emph{The continuous spectrum:} Let $\mu \in \mathbb{C}$ be such that $\Re \mu  \leq  - \f{1}{\al}$ and define,
  			$\psi_{\mu} \in L^2$ such that 
  			\begin{equation}\label{con}
  			\widehat{\psi_{\mu}}(p)= |p|^{- \al \mu} e^{- |p|^{\al}}.
  			\end{equation}
  			Then $\psi_{\mu}$ is an  eigenfunction of the operator $\mathcal{L}$ with the corresponding eigenvalue\footnote{Note however that all this eigenvalues are not isolated, hence they are in the essential spectrum.} $\la=1+\mu - \f{3-\be}{\al}$.  In fact,
  			$$
  			\sigma_{ess}(\mathcal{L})\supseteq \bigg\{ \la \in \mathbb{C}:  \Re \la  \leq  1-  \f{4-\be}{\al}  \bigg \}.
  			$$

  	\end{enumerate}
  \end{proposition}
  
  \begin{proof}
  	{\it Regarding discrete spectrum}, we start with a calculation, which will allow us to identify some of the eigenvalues. 	Let $\phi_0(\xi)$ be a radial function, i.e.  $\widehat{\phi_0}(p)= g(|p|)$.  Then
  	\begin{eqnarray}
  	\nonumber
  		\widehat{\mathcal{L} \phi_0} (p)&=& \widehat{ - |\nabla|^{\al} \phi_0}+ \f{1}{\al} \widehat{\xi \cdot \nabla_{\xi} \phi_0} (p)+ \left(1+\f{\be-1}{\al}\right)\widehat{\phi_0}(p)=  \\
  			\nonumber
  		&=& - |p|^{\al} \widehat{\phi_0}(p)- \f{2}{\al} \widehat{\phi_0}(p) - \f{1}{\al} \sum_{j=1}^2 p_j \partial_j \widehat{ \phi_0}(p)+ \left(1+\f{\be-1}{\al}\right) \widehat{\phi_0}(p)=\\
  			\nonumber
  		&=&  - |p|^{\al} g(|p|)- \f{2}{\al} \widehat{\phi_0}(p) - \f{1}{\al} \sum_{j=1}^2 p_j\  g'(|p|)\ \f{p_j}{|p|}+\left(1+\f{\be-1}{\al}\right) \widehat{\phi_0}(p)= \\
  		\label{50}
  		&=&  \left(1+\f{\be-3}{\al}\right)\widehat{\phi_0}(p)+ \bigg( - |p|^{\al} g(|p|)- \f{1}{\al} |p| \  g'(|p|) \bigg)
  	\end{eqnarray}
  	Now if $g$ satisfies,
  	\begin{equation}\label{gg}
  	- |p|^{\al} g(|p|)- \f{1}{\al} |p| \  g'(|p|)= 0
  	\end{equation} 
  then clearly $\la=\left(1-\f{3-\be}{\al}\right)$ is an eigenvalue for $\cl$. The solution of \eqref{gg}, gives the   eigenfunction,   $  \widehat{\phi_0}(p)=   e^{-|p|^{\al}}$ or $\phi_0=G$. 
  
  Now, let $\phi_k$ be an eigenfunction corresponding to the eigenvalue $\la_k= \left(1-\f{3-\be+k}{\al}\right)$, that is 
  	\begin{equation}\label{eig1}
  	\mathcal{L} \phi_k(\xi)=  \left(1-\f{3-\be+k}{\al}\right) \phi_k
  	\end{equation}
  	Taking a derivative $\p_j$ in \eqref{eig1}, we obtain 
  \begin{eqnarray*}
\left(1-\f{3-\be+k}{\al}\right)  \partial_j \phi_k  &=& \p_j \mathcal{L} \phi_k(\xi) = - |\nabla|^{\al} \partial_j \phi_k  + \f{1}{\al} \partial_j (\xi \cdot \nabla \phi_k)+ \left(1-\f{3-\be+k}{\al}\right)  \partial_j \phi_k = \\
&=& - |\nabla|^{\al} \partial_j \phi_k + \f{1}{\al} \partial_j \phi_k+ \f{1}{\al} \xi \cdot \nabla  (\partial_j\phi_k)+ \left(1-\f{3-\be+k}{\al}\right)  \partial_j \phi_k(\xi)=\\&=& \cl[\p_j \phi_k]+ \f{1}{\al} \partial_j \phi_k. 
  \end{eqnarray*}
 It follows that 
 $$
 \cl[\p_j \phi_k]= \left(1-\f{3-\be+(k+1)}{\al}\right)  \partial_j \phi_k
 $$  
  	It follows that $\left(1-\f{3-\be+k+1}{\al}\right)$ is an eigenvalue, corresponding to   an eigenfunction $\p_j \phi_k$. Thus, we have identified a family of eigenvalues and eigenvectors as follows.   
  Fix $k \in \mathbb{N}$, and let $(\sigma_1, \sigma_2)$ be so that $\sigma_1+ \sigma_2= k$.  Then, by   induction, for the function $\phi_k:=\partial_1^{\si_1} \partial_2^{\si_2} \phi_0$, we have \eqref{eig1}. 
   Note that what we have proved so far does not guarantee that there is not any more discrete spectrum, but merely an inclusion, as stated. 
  
  {\it Regarding essential spectrum,} we compute $\widehat{\cl \psi_\mu}$. From the calculation \eqref{50}, we have 
  $$
  \widehat{\cl \psi_\mu}(p)= \left(\mu + 1+\f{\be-3}{\al}\right) \widehat{\psi}_\mu(p),
  $$ 
  	whence $\psi_\mu$ is an eigenfunction. Indeed, $\psi_\mu \in L^2(2)$, when $\Re\mu\leq -\f{1}{\al}$. This is easy to see with a computation similar to the ones performed in Lemma  \ref{le:10}. 
  	\begin{eqnarray*}
\||\xi|^2 \psi_\mu\|_{L^2}^2=\|\De_p \hat{\psi}_\mu\|_{L^2}^2=\int_0^\infty |(\p_{\rho \rho}+ \f{1}{\rho} \p_\rho) \rho^{-\al \mu} e^{-\rho^\al}|^2  \rho d\rho.
  	\end{eqnarray*}
  	The worst term (when $\al>1$)  is exactly $\int_0^1 \rho^{-(3+2\al \mu)} d\rho$, which converges  for $\Re\mu< -\f{1}{\al}$. 
  \end{proof}
  \subsection{The semigroup $e^{\tau \cl}$}
  The following proposition yields an explicit formula for the semigroup $e^{\tau \cl}$. This is an extension  of the formula established  in \cite{GW}.  
  \begin{proposition}
  	\label{prop:20} 
  	The operator $\cl$ defines a $C_0$ semigroup on $L^2(2)(\rtwo)$,  $e^{\tau \cl}$.  In fact, we have the following  formula for its action 	
  
 \begin{eqnarray}\label{semi1}
  	\widehat{(e^{\tau \cl} f)}(p)&=& e^{(1-\f{3-\be}{\al}) \tau} e^{- a(\tau) |p|^{\al}} \widehat{f}(e^{-\f{ \tau}{\al}} p),  \\ \label{semi2}
  (e^{\tau \cl}f) (\xi)&=& \f{e^{(1-\f{1-\be}{\al})\tau}}{a(\tau)^{\f{2}{\al}}} \int_{\rtwo} G \left(\f{\xi- \eta}{a(\tau)^{\f{1}{\al}}}\right) f(e^{\f{\tau}{\al}}\eta ) d \eta,
  	\end{eqnarray} 
  	where $a(\tau)= 1- e^{- \tau}$. In particular, for $1\leq p\leq \infty$, 
  	\begin{eqnarray}
  	\label{202} 
  	\|e^{\tau \cl} f\|_{L^p} &\leq & C e^{(1-\f{1-\be}{\al}-\f{2}{\al p})\tau} \|f\|_{L^p} \\
  	\label{203} 
  	\|e^{\tau \cl} \nabla f\|_{L^p} &\leq & C \f{e^{(1-\f{2-\be}{\al}-\f{2}{\al p})\tau}}{a(\tau)^{\f{1}{\al}}} \|f\|_{L^p}. 
  	\end{eqnarray}
  \end{proposition}
 \noindent {\bf Remark:} Note that $a(\tau)\sim \min(1, \tau)$. This will be used frequently in the sequel.  
  \begin{proof}
  	The generation of the semigroup would follow, once we prove that the function \\  $g: \hat{g}(\tau, p):=e^{(1-\f{3-\be}{\al}) \tau} e^{- a(\tau) |p|^{\al}} \widehat{f}(p \cdot e^{-\f{ \tau}{\al}})$ satisfies $\p_\tau \hat{g}(\tau,p)=\widehat{\cl g(\tau, \cdot)}$. Clearly, $\hat{g}(0, p)=\hat{f}(p)$, so $g(0, \xi)=f(\xi)$. Next, we compute $\p_\tau \hat{g}(\tau,p)$. We have 
  	\begin{eqnarray*}
  		\partial_{\tau} \widehat{g}(\tau, p) &=& \bigg[(1- \f{3-\be}{\al} - a'(\tau) |p|^{\al}) \widehat{f}(p \cdot e^{-\f{ \tau}{\al}})- \f{1}{\al}  e^{-\f{\tau}{\al}} p \cdot \nabla_p \widehat{f}(p \cdot e^{-\f{ \tau}{\al}})\bigg] e^{\tau(1- \f{3-\be}{\al})} e^{- a(\tau) |p|^{\al}}=\\
  		&=& \left(1+\f{\be-3}{\al}\right) \widehat{g}(p)  + (a(\tau)-1)   |p|^{\al} \widehat{g}(p) - \f{1}{\al} e^{-\f{\tau}{\al}} p \cdot \nabla_p \widehat{f}(p \cdot e^{-\f{\tau}{\al}})  e^{\tau(1-\f{3-\be}{\al})} e^{- a(\tau) |p|^{\al}},
  	\end{eqnarray*}
  	where we have used the relation $a'(\tau)=1-a(\tau)$.   	
  	Next, by \eqref{50}, we have 
  	\begin{eqnarray*}
\widehat{\cl g(\tau, \cdot)} &=& - |p|^{\al} \widehat{g}(p) - \f{1}{\al} \sum_{j=1}^2 p_j \partial_j \widehat{ g}(p)+ \left(1+\f{\be-3}{\al}\right) \widehat{g}(p).
  	\end{eqnarray*}
  	But, 
  	\begin{eqnarray*}
  	& & 	\f{1}{\al} \sum_{j=1}^2 p_j \partial_j \widehat{g}(p)=  \f{1}{\al} \sum_{j=1}^2 p_j \bigg( - \al a(\tau) p_j  |p|^{\al-2}  \widehat{f}(p \cdot e^{-\f{\tau}{\al}})+ e^{-\f{\tau}{\al}} \partial_j\widehat{f}(p \cdot e^{-\f{\tau}{\al}}) \bigg) 
  		e^{\tau(1-\f{3-\be}{\al})} e^{- a(\tau) |p|^{\al}}\\
  		&=&  - a(\tau)  |p|^{\al} \widehat{f}(p \cdot e^{-\f{\tau}{\al}}) e^{\tau(1-\f{3-\be}{\al})}   e^{- a(\tau) |p|^{\al}}+ \f{1}{\al} e^{-\f{\tau}{\al}} p \cdot \nabla_p \widehat{f}(p \cdot e^{-\f{\tau}{\al}})  e^{\tau(1-\f{3-\be}{\al})} e^{- a(\tau) |p|^{\al}}.
  	\end{eqnarray*}
  	Altogether, 
  	 	\begin{eqnarray*}
  	 		\widehat{\cl g(\tau, \cdot)} &=&  - |p|^{\al} \widehat{g}(p)+  \left(1+\f{\be-3}{\al}\right) \widehat{g}(p)+  a(\tau)  |p|^{\al}  \widehat{g}(p)  - \\
  	 		&-&  \f{1}{\al} e^{-\f{\tau}{\al}} p \cdot \nabla_p \widehat{f}(p \cdot e^{-\f{\tau}{\al}})  e^{\tau(1-\f{3-\be}{\al})} e^{- a(\tau) |p|^{\al}}.
  	 	\end{eqnarray*}
  	An immediate inspection reveals that $	\partial_{\tau} \widehat{g}(\tau, p)=	\widehat{\cl g(\tau, \cdot)} (p)$ and so the semigroup formula \eqref{semi1} is established. The formula \eqref{semi2} is just a Fourier inversion of \eqref{semi1}. 
  		Regarding the estimate \eqref{202},  we proceed as follows 
  		\begin{eqnarray*}
  			\|e^{\tau \cl} f\|_{L^p}  &\leq & e^{(1-\f{1-\be}{\al})\tau} \|G_{a(\tau)^{\f{1}{\al}}}\|_{L^1} \|f(e^{\f{\tau}{\al}}\cdot)\|_{L^p}= 
  			e^{(1-\f{1-\be}{\al}-\f{2}{\al p})\tau} \|G\|_{L^1} \|f\|_{L^p}. 
  		\end{eqnarray*}
  		For \eqref{203}, note  that integration by parts yields 
  		\begin{eqnarray*}
  	(e^{\tau \cl}\p_j f) (\xi)&=& \f{e^{(1-\f{1-\be}{\al})\tau}}{a(\tau)^{\f{2}{\al}}} \int_{\rtwo} G \left(\f{\xi- \eta}{a(\tau)^{\f{1}{\al}}}\right) (\p_j f)(e^{\f{\tau}{\al}}\eta ) d \eta=   \f{e^{(1-\f{2-\be}{\al})\tau}}{a(\tau)^{\f{3}{\al}}} \int_{\rtwo} \p_j G \left(\f{\xi- \eta}{a(\tau)^{\f{1}{\al}}}\right) f(e^{\f{\tau}{\al}}\eta ) d \eta,
  		\end{eqnarray*}
  		whence 
  		$$
  		\|(e^{\tau \cl}\nabla f) (\xi)\|_{L^p}\leq  \f{e^{(1-\f{2-\be}{\al}-\f{2}{\al p})\tau}}{a(\tau)^{\f{1}{\al}}}  \|\nabla G\|_{L^1} \|f\|_{L^p}.
  		$$
  \end{proof}
  	We need a variant of Proposition A.2 in \cite{GW}, which discusses the commutation of the semigroup with differential operators. 
  		\begin{lemma}\label{lem0}
  			We have the following commutation relation for $e^{\tau \cl}$ 
  			\begin{equation}
  			\label{70} 
  			\nabla e^{\tau \mathcal{L}}= e^{\f{\tau}{\al}} e^{\tau \mathcal{L}} \nabla
  			\end{equation}
  		\end{lemma}	
  		\begin{proof}
  			Let $u(x, \tau)= e^{\tau \mathcal{L}} f(x)$, then $u$ satisfies the following equation
  			\begin{eqnarray*}
  				\begin{cases}
  					u_{\tau}= \mathcal{L} u, \\ 
  					u(0,x)= f(x).
  				\end{cases}
  			\end{eqnarray*}
  			Clearly, taking a derivative $\p_j$ in \eqref{lamb} yields, for $j= 1, 2$
  			\begin{eqnarray*}
  				\begin{cases}
  					(\partial_j u)_{\tau}= \p_j(\cl u)= \mathcal{L} \partial_j u+ \f{1}{\al} \partial_j u, \\ 
  					\partial_j u(x, 0)= \partial_j f(x),
  				\end{cases}
  			\end{eqnarray*}
  			which has the solution
  		$
  			\partial_j u(x, \tau)= e^{\tau[\mathcal{L}+ \f{1}{\al}]} \partial_j f(x).
  			$
  			In other words 
  		$
  			\nabla e^{\tau \mathcal{L}}  = e^{\f{\tau}{\al}} e^{\tau \mathcal{L}} \nabla.$
  		\end{proof}
  			\subsection{Semigroup estimates} 
  	We need to address an important question, namely the behavior of the bounded 
  	operators $e^{\tau \cl}$ on $L^2(2)$.  The next Proposition does that. More precisely, we are interested in the decay of the operator norms $\|e^{\tau \cl}\|_{L^2(2)\to L^2(2)}$. Importantly,  good decay estimates only happen, when the functions have mean value zero. 
  
  		\begin{proposition}
  			\label{lem1} 
  			Let $f \in L^2(2)$, $\hat{f}(0)= 0$ and $\ga=(\ga_1, \ga_2) \in {\mathbf N}^2, |\ga|=0,1$ and $0<\eps <<1$. Then there exists $C=C_\eps>0$,  such that for any $\tau> 0$, 
  			\begin{equation}  
  			\label{18}
  	\|\nabla^{\ga} ( e^{\tau \cl} f)\|_{L^2(2)} \leq C \f{e^{\left(1-\f{4-\be}{\al}+\eps \right) \tau}}{a(\tau)^{\f{|\ga|}{\al}}} \|f\|_{L^2(2)},
  			\end{equation}
or 
	\begin{equation}  
	\label{1888}
	\|\nabla^{\ga} ( e^{\tau \cl} f)\|_{L^2(2)} \leq C \|f\|_{L^2(2)} \cdot \left\{ 
	\begin{array}{l}
	\tau^{-\f{|\ga|}{\al}},  \ \ \ \ \ \ \ \  \tau \leq 1 \\
	e^{\left(1-\f{4-\be}{\al}+\eps \right) \tau}, \ \ \ \ \tau > 1 
	\end{array}  
	\right.
	\end{equation}
  		\end{proposition}

  \subsection{The decay estimates for $e^{\tau \cl}$ give a  description of the spectrum of $\cl$} 
  \label{sec:3.6} 
In this section, we show that the spectral inclusions in Proposition \ref{prop:10} are actually equalities. We also compute explicitly the Riesz projection $\cp_0$  onto the   eigenvalue of $\cl$ with the largest real part.  In Proposition \ref{prop:10}, we have already identified $G$ as being an eigenfunction for $\cl$ corresponding to an eigenvalue $\la_0=1-\f{3-\be}{\al}$. On the other hand, applying Proposition \ref{lem1}, for functions with $\hat{f}(0)=0$ and $\ga=(0,0)$, implies 
\begin{equation}
\label{110} 
\|e^{\tau \cl} f\|_{L^2(2)}\leq C_\eps  e^{\left(1-\f{4-\be}{\al}+\eps \right) \tau} \|f\|_{L^2(2)}.
\end{equation}
Denote the co-dimension one subspace $X_0=\{f\in L^2(2): \hat{f}(0)=0\}$. Clearly, the operator $\cl$ acts invariantly on $X_0$, since for every $f\in L^2(2): \int f(\xi)d\xi=0$, we have 
$ \int_{\rtwo}   \xi\cdot \nabla f d\xi=0$, whence $\int \cl f (\xi) d\xi=0$. 

Introduce  $\cl_0:=\cl|_{X_0}$, with domain $D(\cl_0)= D(\cl)\cap X_0=H^\al\cap X_0$.  By the Hille-Yosida theorem, this estimate \eqref{110} implies that the set $\{\la: \Re\la > \left(1-\f{4-\be}{\al} \right) \}$ is in the resolvent set of $\cl_0$, since the integral representing $(\la-\cl)^{-1}$, namely 
$
\int_0^\infty e^{-\la \tau} e^{\tau \cl} d\tau,
$
converges by  virtue of \eqref{110}. 
Combining this with the results from Proposition \ref{prop:10}, we conclude that  
$\si(\cl)\cap \{\la: \Re \la>\left(1-\f{4-\be}{\al} \right)\} $ is a singleton - the eigenvalue $\la_0=1-\f{3-\be}{\al}$, which  is simple, with eigenfunction $G$. We conclude that 
$$
\si(\cl)=\{1-\f{3-\be}{\al}\} \cup \si_{ess}(\cl); \ \ \si_{ess}(\cl)=\{\la: \Re\la\leq \left(1-\f{4-\be}{\al} \right)\},
$$
Moreover, its Riesz projection $\cp_0$,  a rank one operator,  is given by 
$$
\cp_0 f= \left(\int_{\rtwo} f(\xi) d\xi\right)G
$$
Clearly, such an operator is well-normalized, since $\cp_0^2 f = \dpr{G}{1} \cp_0 f= \hat{G}(0)  \cp_0 f=\cp_0 f$, since $\hat{G}(0)=1$. The projection $\cq_0:=Id-\cp_0$ over the complementary part of the spectrum, satisfies  $\cl_0=\cq_0 \cl \cq_0$.  Also, $\cq_0:L^2(2)\to X_0$. Now, \eqref{110} can be  reformulated as 
\begin{equation}
\label{405} 
\|\nabla^\ga e^{\tau \cl_0} f\|_{L^2(2)}\leq C_\eps  \f{e^{\left(1-\f{4-\be}{\al}+\eps \right) \tau}}{a(\tau)^{\f{|\ga|}{\al}}} \|f\|_{L^2(2)}.
\end{equation}
for any function $f$, since $e^{\tau \cl_0} f=e^{\tau \cl} \cq_0 f$ and the entry $\cq_0 f$ has mean value zero, so \eqref{110} is applicable.  
In addition, we can derive estimates for the action of the semigroup $e^{\tau \cl}$ on $L^2(2)$, without the   mean value zero property $\hat{f}(0)=0$. 
\begin{proposition}
	\label{prop:43}
	Let $f\in L^2(2)$ and $0<\eps<\f{1}{\al}$.  Then, there exists a constant $C=C_\eps$, so that 
	\begin{equation}
	\label{19}
	\|\nabla^\ga (e^{\tau \cl} f)\|_{L^2(2)}\leq C_\eps \f{e^{\left(1-\f{3-\be}{\al} \right) \tau}}{a(\tau)^{\f{|\ga|}{\al}}} \|f\|_{L^2(2)}.
	\end{equation}
\end{proposition}
\begin{proof}
We use the decomposition 
$$
f=\cp_0 f+ \cq_0 f= \dpr{f}{1} G + [f-\dpr{f}{1} G]. 
$$
Thus, $$
e^{\tau \cl} f= \dpr{f}{1} e^{\tau(1-\f{3-\be}{\al})} G+ e^{\tau \cl_0} [ f]
$$
It follows that 
\begin{eqnarray*}
\|e^{\tau \cl} f\|_{L^2(2)} &\leq &  C |\dpr{f}{1}| e^{\tau(1-\f{3-\be}{\al})} \|G\|_{L^2(2)} + C_\eps  e^{\left(1-\f{4-\be}{\al}+\eps \right) \tau} \|f\|_{L^2(2)}\\
& \leq & C e^{\tau(1-\f{3-\be}{\al})}\|f\|_{L^2(2)},
\end{eqnarray*}
where we have used \eqref{405} and $|\dpr{f}{1}|\leq C \|f\|_{L^2(2)}$.  Similar estimates can be derived, as before,  for $\nabla^\ga e^{\tau \cl}$, we omit the details.

\end{proof}

 \section{Local and global well-posedness of the SQG}
 \label{sec:4} 
 The local and global theory of the Cauchy problem for SQG has been well-studied in the literature. Local and global well-posedness holds under very general conditions on initial data. Regardless, we will present a few results for our problem \eqref{qq}. This is necessary, since we assume a  non-standard relation between $u$ and $z$, but also because we need precise properties, beyond the scope of the well-posedness. Then, we will turn to properties of the rescaled equation, \eqref{8}. We will do so, both in $L^p$ spaces as well as in $L^2(2)$ spaces - the reason is that we will use some of our preliminary results as {\it a priori} estimates  in the subsequent Lemmas. 
 
 Our first results  are about the well-posedness of the standard model \eqref{qq} in $L^p$ spaces\footnote{The results can be made more precise, in individual $L^p$ spaces, rather than in {\it all} $L^p$ spaces. We will not do so here, because our goal is to extend to $L^2(2)$, which is yet smaller space.}.
 \subsection{Global well-posedness  and a priori estimates in $L^p$ spaces}
 \begin{lemma}
 	\label{thma}
 	Suppose that $z_0 \in L^1\cap L^\infty=:X$. Then, \eqref{qq} is globally well-posed in the space $X$. 
 		Moreover, for every $p\in [1, \infty]$, $t\to \|z(\cdot, t)\|_{L^p}$ is non-increasing  in time. 
 \end{lemma}
 	\begin{proof}  
 			We first prove the local existence of the strong solution in the space $C([0,T); X)$, that is, with $T$ to be determined, we are looking for a fixed point of the integral  equation 
 			\begin{equation}
 			\label{140} 
 			z(\xi)= e^{- t |\nabla|^{\al}} z_0 -  \int_0^t e^{- (t- s) |\nabla|^{\al}} \nabla (u \cdot z) \ ds. 
 			\end{equation}	
 		  According to Lemma \eqref{aa}  
 			$\|e^{- t |\nabla|^{\al}} z_0\|_{L^1\cap L^\infty}\leq C_0 \|z_0\|_{L^1\cap L^\infty}$.  For any   $T>0$ and  $t\in (0,T)$, consider 
 			\begin{eqnarray*}
 			Q(z_1, z_2):= \int_0^t e^{- (t- s) |\nabla|^{\al}} \nabla (u_1 \cdot z_2) \ ds,
 			\end{eqnarray*}
 			where  $u_1$ is  given by   $u_1= (\nabla^{\perp})^{- \be} z_1$. For $t\in (0,T)$, using \eqref{2.8}
 			\begin{eqnarray*}
 			& & 	\|Q(z_1(t), z_2(t))\|_{L^1}= \| \int_0^t e^{- (t- s) |\nabla|^{\al}} \nabla (u_1 \cdot z_2) \ ds\|_{L^1} \leq  C \int_0^t \frac{1}{(t- s)^{\frac{1}{\al}}} \| (u_1 \cdot z_2)\|_{L^1} \ ds \\
 				&\leq& C t^{1-\f{1}{\al}} \sup_{0\leq s\leq T}\|u_1(s, \cdot)\|_{L^{\infty}} \sup_{0\leq s\leq T}\|z_2(s, \cdot)\|_{L^1}
 				\leq  \\
 				&\lesssim & T^{1-\f{1}{\al}} \sup_{0\leq s\leq T}(\|z_1(s, \cdot)\|_{L^{\f{2}{\be}+\eps}} + \|z_1(s, \cdot)\|_{L^{\f{2}{\be}-\eps}})\sup_{0\leq s\leq T}\|z_2(s, \cdot)\|_{L^1}
 				\lesssim  T^{1-\f{1}{\al}} \sup_{0\leq s\leq T}\|z_1\|_X \sup_{0\leq s\leq T} \|z_2\|_X.
 			\end{eqnarray*}
 			where we have used the Sobolev embedding estimate \eqref{31}. 	Similarly, 	
 			$$
 				\|Q(z_1, z_2)\|_{L^\infty}\leq C T^{1-\f{1}{\al}} \sup_{0\leq s\leq T}\|u_1\|_{L^\infty} \sup_{0\leq s\leq T}\|z_2\|_{L^\infty}\leq C T^{1-\f{1}{\al}} \sup_{0\leq s\leq T}\|z_1\|_X \sup_{0\leq s\leq T}\|z_2\|_X.
 			$$
 			Finally, following similar path, we also have 
 			$$
 			\|Q(z_1,z_1)-Q(z_2, z_2)\|_X\leq C T^{1-\f{1}{\al}} (\|z_1\|_X+\|z_2\|_X) \|z_1-z_2\|_X.
 			$$
 			Upon  introducing   
 			 $Y_T:= \{z: \sup_{0\leq t\leq T}\|z(t, \cdot)\|_X \leq 2 C_0 \|z_0\|_X\}$ and taking into account the estimates above, we realize that the mapping   \eqref{140} has a fixed point in the metric space $C([0,T],X)$, for small enough $T=T(\|z_0\|_X)$.   In fact, the argument shows that $T\sim \|z_0\|_X^{-\f{\al}{\al-1}}$. 
 			 
 			  For the global existence, we need to show that the $t\to \|z(t, \cdot)\|_{L^p}$ does not blow up in finite time. In fact, we show that the $t\to \|z(t, \cdot)\|_{L^p}$ is non-increasing, which will allow us to conclude global existence as well.  To that end, we dot product the equation \eqref{qq} with $|z|^{p- 2}z, p\in (1, \infty)$ to get
 			  \begin{equation*}
 			  \frac{1}{p} \partial_t \|z\|^p_{L^p}+ \int_{\rtwo}  |\nabla|^{\al} z \cdot |z|^{p- 2} z d \xi = 0.
 			  \end{equation*}
 			  By the positivity estimate \eqref{a:20}, we have 
 			  $\int_{\rtwo}  |\nabla|^{\al} z \cdot |z|^{p- 2} z d \xi \geq 0$.    Therefore, 
 			$ 	  \partial_t \|z\|^p_{L^p} \leq 0$, and $ t\to \|z(t, \cdot)\|_{L^p}$ is non-increasing  in time. For 
 			$p=1, p=\infty$ the monotonicity follows from an approximation argument from the cases $1<p<\infty$.

 		\end{proof}
 
 Our next result is about {\it a priori} estimates in $L^p$ spaces, but this time in the rescaled variable formulation, \eqref{8}. Note that the global existence of the rescaled equation is not in question anymore, due to Lemma \ref{thma}. However, we show   precise decay estimates for the norm of the solution $Z$. This fairly elementary Lemma already shows the advantage of the rescaled variables approach and its far reaching consequences. 
 \begin{lemma}
 	\label{lem2}
 	Let  $Z_0 \in L^1\cap L^\infty(\rtwo)$, $\al \in (1, 2)$, $0 \leq \be < 2$ and $p\in [1, \infty)$.  Then the unique global strong solution $Z$ of \eqref{8} satisfies   
 	\begin{equation}
 	\label{f2}
 	\|Z(\tau)\|_{L^p} \leq \|Z_0\|_{L^p}  e^{-\tau ( \f{2}{p \al}  - 1-\f{\be-1}{\al} )  }.
 	\end{equation}
 \end{lemma}
 \begin{proof}
 	If we dot product \eqref{8} with $Z |Z|^{p- 2}$, we have by the positivity estimate \eqref{a:20}, \\ 
 	$	\int_{\rtwo}  |\nabla|^{\al} Z \cdot |Z|^{p- 2} Z d \xi \geq 0$. Furthermore, 
 	 using  the divergent free property of $U (\xi)$ 
 	\begin{eqnarray}
 	\nonumber
 	\f{1}{p} \f{d}{d \tau} \|Z\|^p_{L^p}  &\leq &  \f{1}{\al} \int (\xi \cdot \nabla_{\xi} Z) Z |Z|^{p- 2} \ d \xi - \int (U \cdot \nabla_{\xi} Z) Z |Z|^{p- 2} d \xi+  \\
 	\label{eng}
 	&+& \left(1+\f{\be-1}{\al}\right) \|Z\|^p_{L^p} =  
 	 \left(1+\f{\be-1}{\al} -\f{2}{\al p}\right) \|Z\|^p_{L^p},
 	\end{eqnarray}
 	therefore, we arrive at 
 	\begin{eqnarray*}
 		\f{1}{p} \f{d}{d \tau} \|Z\|^p_{L^p}+ (\f{2}{\al p} - 1-\f{\be-1}{\al} ) \|Z\|^p_{L^p} \leq  0.
 	\end{eqnarray*}
 	Now we use the Gronwall's inequality to finish the proof. 
 \end{proof}
 The Lemma above  shows \emph{a priori} bound for $\|Z(\tau, \cdot)\|_{L^p}$, 
 for any $p \in [1, \infty]$, and a decay rate for 
 $p < \f{2}{\al+\be-1}$, but it is not giving any decay rate for $p \geq \f{2}{\al+\be-1}$. On the other hand, as we shall see later,  the decay rate predicted by Lemma \ref{lem2} is in fact optimal for $p=1$ (but certainly not so, for any other value of $p$). We can bootstrap the results of Lemma \ref{lem2} in the next Lemma to find, what {\it it will turn out to be, the optimal decay rate\footnote{for generic data}  for any $ p \geq 1$}.
 \begin{lemma}
 	\label{lem3}
 	Let  $Z_0 \in L^1\cap L^\infty(\rtwo), 1\leq p\leq \infty$ and $\al\in (1,2)$, $\al+\be\leq 3$. 
 	 Then, there exists constant $C=C_{p,\al, \be}$, so that  the unique global strong solution $Z$ of \eqref{8}  
 	satisfies 
 	\begin{equation}\label{f3}
 	\|Z(\tau, \cdot)\|_{L^p} \leq   C_{p,\al, \be} \|Z_0\|_{L^1\cap L^\infty}  e^{- (\f{3-\be-\al}{\al}) \tau}.
 	\end{equation}
 \end{lemma}
 \begin{proof}
 Recall  that the estimate \eqref{a:31} is available to us.  Taking   dot product $|Z|^{p- 2} Z$ and taking into account 	   
 	\eqref{a:31} which implies $\int_{\rtwo}  |\nabla|^{\al} Z \cdot |Z|^{p- 2} Z d \xi\geq 
 	c_{p, \al}  \|Z\|^2_{L^{\f{2p}{2-\al}}}$. We further  add  $C\|Z\|_{L^p}^p$, for some large $C$, to be determined. 
 	 We have   	
 	$$
 		\f{1}{p} \f{d}{d \tau} \|Z\|^p_{L^p} + C\|Z\|_{L^p}^p+  c_{p, \al} \|Z\|^2_{L^{\f{2p}{2-\al}}}   \leq  
 		\left(C+1+\f{\be-1}{\al}-
 		\f{2}{\al p}\right) \|Z\|^p_{L^p}
 	$$
  By Gagliardo-Nirenberg's, with $\ga= \f{2 p- 2}{2p- 2+ \al} $, 
 	$
 	\|Z\|_{L^p}\leq \| Z\|_{L^{\f{2p}{2- \al}}}^\ga \|Z\|_{L^1}^{1-\ga},
 	$
 	 whence by  Young's inequality 
 	\begin{eqnarray*}
 		\f{1}{p} \f{d}{d \tau} \|Z\|^p_{L^p}&+&  C\|Z\|_{L^p}^p + c_{p, \al}   \| Z\|^p_{L^{\f{2p}{2- \al}}}   \leq 
 	 	\left( C+1+\f{\be-1}{\al}-
 	 	\f{2}{\al p}\right)  \|Z\|^{p \ga}_{L^{\f{2p}{2- \al}}} \|Z\|_{L^1}^{p(1- \ga)} \leq \\
 	 	&\leq &  \epsilon_0 \|Z\|^{p}_{L^{\f{2p}{2- \al}}}+ \f{\left( C+1+\f{\be-1}{\al}-
 	 		\f{2}{\al p}\right)^{\f{1}{1-\ga}} }{\epsilon_0^{\f{\ga}{1-\ga}}}   \|Z\|_{L^1}^{p}
 	\end{eqnarray*}
 	  and $\epsilon_0> 0$ is a fixed number, say we select it $\eps_0= c_{p, \al} $.  Then 
 	\begin{eqnarray*}
 		\f{1}{p} \f{d}{d \tau} \|Z\|^p_{L^p} + C \|Z\|^p_{L^p}\leq \f{\left( C+1+\f{\be-1}{\al}-
 			\f{2}{\al p}\right)^{\f{1}{1-\ga}}  }{\epsilon_0^{\f{\ga}{1-\ga}} }  \|Z\|_{L^1}^{p} \leq  
 		\f{\left( C+1+\f{\be-1}{\al}-\f{2}{\al p}\right)^{\f{1}{1-\ga}}  }{\epsilon_0^{\f{\ga}{1-\ga}} } \|Z_0\|_{L^1}^p  e^{-p\tau(\f{3-\be-\al}{\al}) },
 	\end{eqnarray*}
 	where we have used   Lemma \eqref{lem2} to estimate $ \|Z(\tau, \cdot)\|_{L^1}$. Denoting $\mu:=(\f{3-\be-\al}{\al})\geq 0$, select $C=\mu+1$. We have 
 	$$
 	I'(\tau)+ p(\mu+1) I(\tau) \leq D \|Z_0\|_{L^1}^p e^{-p\mu \tau},
 	$$
 	where $I(\tau)=\|Z(\tau)\|^p_{L^p}$, $D=p^{1+\f{\ga}{1-\ga}} \f{\left( \mu+2+\f{\be-1}{\al}-\f{2}{\al p}\right)^{\f{1}{1-\ga}}  }{c_\al^{\f{\ga}{1-\ga}} } $.  	Now we use the Gronwall's inequality to derive the estimate 
 	$$
 	I(\tau)\leq e^{-p(\mu+1)\tau} I(0)+\f{D}{p} \|Z_0\|_{L^1}^p e^{-p \mu \tau}.
 	$$
 	Taking $p^{th}$ root and  simplifying yields the final estimate 
 	$$
 	\|Z(\tau)\|_{L^p}\leq (\|Z_0\|_{L^p}+\left(\f{D}{p}\right)^{\f{1}{p}}  \|Z_0\|_{L^1}) e^{-\mu \tau}
 	\leq (1+ \left(\f{D}{p}\right)^{\f{1}{p}})  	\|Z_0\|_{L^1\cap L^\infty} e^{-\mu \tau}.
 	$$
 	For the case $p=\infty$, we take limits in the previous identity, for fixed $\tau>0$, as $p\to \infty$. Note that by the explicit form of $D_p$, 
 	$\lim_{p\to \infty}  \left(\f{D}{p}\right)^{\f{1}{p}}=1$, so \eqref{f3} holds true in this case with $C=2$. 
 \end{proof}

 \subsection{Global solutions and a priori estimates in $L^2(2)$}
 From the previous section, we know that the SQG equation in its standard form, namely  \eqref{qq}, has global solutions in $L^p$. Thus, the  rescaled equation \eqref{8} also has unique global (strong) solutions in $L^p$. We now would like to understand the Cauchy problem in the smaller space $L^2(2)$. In particular, even if the initial data is well-localized, say  $Z(0, \cdot) \in L^2(2)$, it is not {\it a priori} clear why the solution $Z(\tau)$ will stay in $L^2(2)$ for (any) later time $\tau>0$. In other words, one needs to start with  the local well-posedness for \eqref{8}, and then we shall upgrade it to a global one, by means of {\it a priori} estimates on $\|Z(\tau)\|_{L^2(2)}$. 
 \begin{theorem}\label{thm1}
 	Suppose that $Z_0 \in L^2(2)(\rtwo)\cap L^\infty(\rtwo)=:X$.   Then \eqref{8} has an unique global strong solution 
 	$Z \in C^0([0, \infty]; L^2(2)(\rtwo) \cap L^\infty(\rtwo))$,  with $Z(0)= Z_0$. In addition,  there is the {\it a priori} estimate 
 	\begin{equation}
 	\label{200} 
 	\|Z(\tau)\|_{L^2(2)\cap L^\infty} \leq   C e^{- \tau (\f{3-\al-\be}{\al})} \|Z_0\|_{L^2(2)\cap L^\infty},
 	\end{equation} 
 	where $C$ is an absolute constant. 
 \end{theorem}
 \begin{proof}
 We set up a local well-posedness scheme for the  integral equation corresponding to \eqref{8}, with initial data $Z(0)=f$,  namely 
 \begin{equation}
 \label{210} 
 Z(\tau)= e^{\tau \mathcal{L}} f  -  \int_0^{\tau} e^{(\tau- s)\mathcal{L}} \nabla \cdot (U  Z)\ ds,
 \end{equation}
 	where $U= U_Z=  (|\nabla|^{\perp})^{- \be}  Z$. We have, according to \eqref{202} and \eqref{19}, 
 	$$
 	\|e^{\tau \mathcal{L}} f\|_{L^2(2)} + 	\|e^{\tau \mathcal{L}} f\|_{L^\infty} \leq C (e^{(1-\f{1-\be}{\al})\tau} + 
 	e^{(1-\f{3-\be}{\al})\tau})\|f\|_{L^2(2)\cap L^\infty}
 	$$
 	Thus, with $T\leq 1$ to be determined later, set 
 	$$
 	Y_T:=\{Z(\tau, \cdot)\in X: \sup_{0\leq s\leq T} \|Z(s, \cdot)\|_{X}\leq 2 C (e^{(1-\f{1-\be}{\al})} + 
 	e^{(1-\f{3-\be}{\al})})\|f\|_X \}, 
 	$$
 where the bound in $Y$ is selected to be twice the value of the bound above, at $\tau=1$. For the non-linear term, we have for each $\tau\in (0,T)$, 
 		\begin{eqnarray*}
 		& & 	\|\int_0^{\tau} e^{(\tau- s)\mathcal{L}} \nabla \cdot (U_{Z_1}  Z_2)\ ds\|_{L^\infty}  \leq  
 			C \int_0^\tau C \f{e^{(1-\f{2-\be}{\al})(\tau-s)}}{a(\tau-s)^{\f{1}{\al}}} \|U_{Z_1}(s) Z_2(s) \|_{L^\infty} ds\leq \\
 			&\leq &  
 			C \sup_{0\leq s\leq T} \|U_{Z_1}\|_{L^\infty}  \sup_{0\leq s\leq T} \|Z_2\|_{L^\infty} \int_0^\tau \f{1}{(\tau-s)^{\f{1}{\al}}} ds \leq \\
 			&\leq & C T^{1-\f{1}{\al}} \sup_{0\leq s\leq T}(\|Z_1\|_{L^{\f{2}{\be}+\eps}}+ \|Z_1\|_{L^{\f{2}{\be}-\eps}}) \sup_{0\leq s\leq T}\|Z_2\|_{L^\infty}\leq C T^{1-\f{1}{\al}} \sup_{0\leq s\leq T} \|Z_1\|_X  \sup_{0\leq s\leq T} \|Z_2\|_X,
 		\end{eqnarray*}
 	where we have used  \eqref{203}, $e^{(1-\f{2-\be}{\al})(\tau-s)}\leq 3$, 
 	$a(\tau-s)=1-e^{-(\tau-s)} \sim  (\tau-s)$, for $0<s<\tau\leq 1$, the Sobolev embedding estimate \eqref{31} and finally the fact that $X=L^2(2)\cap L^\infty \hookrightarrow L^1\cap L^\infty$. 
For the other norm in the definition of $X$, we have by Lemma \ref{lem0}, 
 	\begin{eqnarray*}
 & & 	 \|\int_0^{\tau} e^{(\tau- s)\mathcal{L}} \nabla \cdot (U_{Z_1} \cdot Z_2)\ ds\|_{L^2(2)}   =  \int_0^{\tau} 
 	 e^{-\f{(\tau- s)}{\al}}\| \nabla \cdot e^{(\tau-s) \mathcal{L}} (U_{Z_1} \cdot Z_2)\|_{L^2(2)}\ ds\\
 		&\leq& C \int_0^{\tau} \f{e^{-\f{(\tau- s)}{\al}} e^{(1- \f{3-\be}{\al})(\tau- s) }}{a(\tau- s)^{\f{1}{\al}}} 
 		\|  U_{Z_1}(s) \cdot Z_2(s)\|_{L^2(2)}\ ds \leq \\
 		&\leq & C  \sup_{0\leq s\leq T} \|U_{Z_1}(s)\|_{L^\infty}  \sup_{0\leq s\leq T} \|Z_2(s)\|_{L^2(2)} \int_0^{\tau} \f{1}{(\tau-s)^{\f{1}{\al}}} ds\leq     C T^{1-\f{1}{\al}} \sup_{0\leq s\leq T} \|Z_1\|_X  \sup_{0\leq s\leq T} \|Z_2\|_{L^2(2)}.
 	\end{eqnarray*}
 Having these two bilinear estimates allows us to conclude that for sufficiently small $T$, of the form 
 $T \sim  \|f\|_X^{-\f{\al}{\al-1}}$ (which should also be taken $T\leq 1$), we have local well-posedness in the space $X$. 
 	
 Regarding global existence in $X=L^2(2)\cap L^\infty$, we obviously need {\it a priori} estimates for the solution to prevent potential blow up. We already have those in $L^\infty$ and in $L^2$, by the results of Lemma \ref{lem3}. Thus, it remains to control the norm $J(\tau):=\int_{\rtwo} |\xi|^4 |Z(\tau, \xi)|^2 d\xi$. To this end, take a 
 dot product of the equation \eqref{8} with $|\xi|^4 Z$. We have 
 \begin{eqnarray*}
& & \p_\tau \f{1}{2} \int  |\xi|^{4} Z^2  d \xi+ \int |\xi|^{4} |\nabla|^{\al} Z \cdot Z d \xi = \\
&=& \f{1}{\al} \int (\xi \cdot \nabla_{\xi} Z) |\xi|^{4} Z\ d \xi- \int (U \cdot \nabla_{\xi} Z) |\xi|^{4} Z\ d \xi+ 
(1+\f{\be-1}{\al}) \int |\xi|^{4} Z^2 d \xi.
 \end{eqnarray*}
We first analyze the terms on the right hand-side. Integration by parts yields 
 	\begin{eqnarray*}
& &	\f{1}{\al} \int (\xi \cdot \nabla_{\xi} Z) |\xi|^{4} Z d \xi =  - \f{3}{\al} \int |\xi|^{4} Z^2 d \xi ; \ \ 
 \int (U \cdot \nabla_{\xi} Z) |\xi|^{4} Z\ d \xi = -2 \int |\xi|^{2} (\xi \cdot U)  Z^2 d \xi.
 	\end{eqnarray*}
 	Note that by Young's inequality, we have for all $\eps>0$ 
 	$$
 	|\int |\xi|^{2} (\xi \cdot U)  Z^2 d \xi|\leq C \int  |\xi|^3 \|U\|_{L^\infty} Z^2(\xi) d\xi\leq \eps \int |\xi|^4 Z^2(\xi)d\xi+ C \eps^{-3} \|U\|_{L^\infty}^{4} \|Z\|_{L^2}^2.
 	$$
 By the Sobolev embedding \eqref{31} and Lemma \ref{lem3}, we have 
 	$$
 	\|U\|_{L^\infty}\leq C(\|Z\|_{L^{\f{2}{\be}+\eps}}+ \|Z\|_{L^{\f{2}{\be}-\eps}})\leq C e^{-(\f{3-\be-\al}{\al})\tau},
 	$$
 	so for every $\eps>0$, we have the estimate 
 	$$
 	|\int |\xi|^{2} (\xi \cdot U)  Z^2 d \xi|\leq  \eps \int |\xi|^4 Z^2(\xi)d\xi+ C \eps^{-3} e^{-6\tau(\f{3-\be-\al}{\al})}.
 	$$
 	
 The term $\int |\xi|^{4} |\nabla|^{\al} Z \cdot Z d \xi$ will give rise to some harder error terms (involving commutators between the $|\nabla|^{\al/2}$ and the weights), which we need to eventually control. It turns out that the most advantageous way to reign in the error terms is to split the weight $|\xi|^4$ between the two entries. More precisely, 
 \begin{eqnarray*}
 & & \int |\xi|^{4} |\nabla|^{\al} Z \cdot Z d \xi  =  \int |\xi|^{2} |\nabla|^{\al} Z \cdot |\xi|^2 Z d \xi = 
 \dpr{|\xi|^{2} |\nabla|^{\al/2} [|\nabla|^{\al/2} Z]}{|\xi|^2 Z} = \\
 &=& \dpr{ |\nabla|^{\al/2} |\xi|^{2} [|\nabla|^{\al/2} Z]}{|\xi|^2 Z}-\dpr{ [|\nabla|^{\al/2}, |\xi|^{2}] [|\nabla|^{\al/2} Z]}{|\xi|^2 Z} = \\
 &=& \dpr{  |\xi|^{2} [|\nabla|^{\al/2} Z]}{|\nabla|^{\al/2}[|\xi|^2 Z]}-\dpr{ [|\nabla|^{\al/2}, |\xi|^{2}] [|\nabla|^{\al/2} Z]}{|\xi|^2 Z}=\\
 &=&  \dpr{  |\xi|^{2} |\nabla|^{\al/2} Z}{|\xi|^2 |\nabla|^{\al/2}  Z} + \dpr{  |\xi|^{2} |\nabla|^{\al/2} Z}{[|\nabla|^{\al/2},|\xi|^2]  Z}-\dpr{ [|\nabla|^{\al/2}, |\xi|^{2}] [|\nabla|^{\al/2} Z]}{|\xi|^2 Z}=\\
 &=& \int |\xi|^{4} ||\nabla|^{\f{\al}{2}} Z|^2 d \xi+ \dpr{  |\xi|^{2} |\nabla|^{\al/2} Z}{[|\nabla|^{\al/2},|\xi|^2]  Z}-\dpr{ [|\nabla|^{\al/2}, |\xi|^{2}] [|\nabla|^{\al/2} Z]}{|\xi|^2 Z}.
 \end{eqnarray*}
 	Denote the error terms  	
 	$
 	 	E:= \dpr{  |\xi|^{2} |\nabla|^{\al/2} Z}{[|\nabla|^{\al/2},|\xi|^2]  Z}-\dpr{ [|\nabla|^{\al/2}, |\xi|^{2}] [|\nabla|^{\al/2} Z]}{|\xi|^2 Z}. 	
 	 	$
 Putting it all together implies 
 \begin{eqnarray}
 \label{300}
& & 	\f{1}{2} J'(\tau) + (\f{4-\al-\be}{\al}-\eps) J(\tau)+\int |\xi|^{4} ||\nabla|^{\f{\al}{2}} Z|^2 d \xi   \leq    |E| +  C \eps^{-3} e^{-6\tau(\f{3-\be-\al}{\al})}\\
\nonumber 
	&\lesssim & \||\xi|^{2} |\nabla|^{\al/2} Z\|_{L^2} \|[|\nabla|^{\al/2},|\xi|^2]  Z\|_{L^2}+ 
	\|[|\nabla|^{\al/2}, |\xi|^{2}] [|\nabla|^{\al/2} Z]\|_{L^2} \||\xi|^2 Z\|_{L^2}+	
	 \eps^{-3} e^{-6\tau(\f{3-\be-\al}{\al})}. 
 \end{eqnarray}

 	At this point, it becomes clear that we need  to control the commutator expression above. In fact, we have the following Lemma. 
 	\begin{lemma}
 		\label{L_-50} 
 		Let  $\al\in (1,2)$. Then,  there is $C=C_{\al}$, so that 
 		\begin{equation}
 		\label{310} 
 	 \|[|\nabla|^{\al/2},|\xi|^2]  f\|_{L^2(\rtwo)}\leq C \| |\xi|^{2-\f{\al}{2}} f\|_{L^2(\rtwo)}.
 		\end{equation}
 	\end{lemma}
 	We postpone the proof of Lemma \ref{L_-50} for the Appendix, see Section \ref{app:10}. We finish the proof of Theorem \ref{thm1} based upon it. 
By  Gagliardo-Nirenberg's inequality 
 	$$
 	\| |\xi|^{2-\f{\al}{2}} g\|_{L^2}\leq  \| |\xi|^{2} g\|_{L^2}^{1-\f{\al}{4}} \|  g\|_{L^2}^{\f{\al}{4}}.
 	$$
 		Continuing with our arguments above (see \eqref{300}), we conclude from Lemma \ref{L_-50} that 
 	\begin{eqnarray*}
 	& & \f{1}{2} J'(\tau) + (\f{4-\al-\be}{\al}-\eps) J(\tau)+  \||\xi|^{2} |\nabla|^{\al/2} Z\|_{L^2}^2  \leq \eps \||\xi|^{2} |\nabla|^{\al/2} Z\|_{L^2}^2 + \eps \| |\xi|^2 Z\|_{L^2}^2 +   C_\eps \|Z\|_{L^2}^2
 	\end{eqnarray*}
 	All in all, for all $\eps<1$, we have by Lemma \ref{lem3}, 
 	$$
 	 \f{1}{2} J'(\tau) + (\f{4-\al-\be}{\al}-2 \eps) J(\tau)\leq C_\eps \|Z\|_{L^2}^2\leq C \|Z_0\|_{L^1\cap L^\infty}^2  e^{- 2\tau (\f{3-\be-\al}{\al})}.
 	$$
 By Gronwall's, we finally conclude that 
 $$
 J(\tau)\leq J(0) e^{-2\tau(\f{4-\al-\be}{\al}-2 \eps)}+ C \|Z_0\|_{L^1\cap L^\infty}^2  e^{- 2\tau (\f{3-\al-\be}{\al})}.
 $$	
 As a consequence 
 $$
 \||\xi|^2 Z(\tau)\|_{L^2}\leq C \|Z_0\|_{L^2(2)\cap L^\infty} e^{- \tau (\f{3-\al-\be}{\al})}.
 $$
 	This completes the proof of Theorem \ref{thm1}.  	
 \end{proof}

 \section{Local and global existence of the solutions to the Boussinesq system} 
 \label{sec:5}
 The results of this section closely mirror Section \ref{sec:4}. Consequently, we omit many of the arguments, when they are virtually the same. There are however a few important distinctions, which we will highlight herein. 
 \subsection{Global regularity for the vorticity $(\om, \theta)$ Boussinesq system in $L^p(\rtwo)$}
 Our first result is,  non-surprisingly, is  a local existence and uniqueness result in $L^p(\rtwo)$. Most of the claims in this Lemma are either well-known or follow classical arguments, but we provide a sketch of the proof for completeness. 
   \begin{lemma}
   	\label{thma1}
   	Suppose that $\om_0, \theta_0 \in L^p$, $1 \leq p \leq \infty$. Then
   	there exists $T= T(\|(\om_0, \theta_0)\|_{L^1 \cap L^{\infty}})$,  such that  unique strong solutions 
   	$\om, \theta \in C([0,T); L^1 \cap L^{\infty})$ exist. 
   	
   	Moreover, the solutions $\om(t), \theta(t)$ exist globally. In addition, the function $t\to \|\theta(t, \cdot)\|_{L^p}, 1\leq p\leq \infty$ is non-increasing, $\|\theta(t, \cdot)\|_{L^p}\leq \|\theta_0\|_{L^p}, 1<p<\infty$,  while 
   	$$
   	\|u(t, \cdot)\|_{L^2}\leq \|u_0\|_{L^2}+ t \|\theta_0\|_{L^2}.
   	$$
   \end{lemma}
 \begin{proof}
 	For the local existence,  we work in the space $X= L^1 \cap L^{\infty}= \cap L^p$. The strong solutions of the system of equations \eqref{BSQ1} are solutions of the integral equations 
 	\begin{equation}
 	\label{700}
 		\begin{cases}
 			\om(\xi, t)= e^{- t |\nabla|^{\al}} \om_0+ \int_0^t e^{- (t- s) |\nabla|^{\al}} \nabla (u \cdot \om) \ ds- 
 			\int_0^t e^{- (t- s) |\nabla|^{\al}} \partial_1 \theta \ ds,\\
 			\theta(\xi, t)= e^{- t |\nabla|^{\al}} \theta_0+ \int_0^t e^{- (t- s) |\nabla|^{\al}} \nabla (u \cdot \theta) \ ds.
 		\end{cases}
 	\end{equation}
 	By \eqref{2.7}, we have that 
 	$$
 	\|e^{- t |\nabla|^{\al}} \om_0\|_{X}+\|e^{- t |\nabla|^{\al}} \theta_0\|_{X}\leq C (\|\om_0\|_X+\|\theta_0\|_X)
 	$$
 	One can now consider the space $Y:= \{(\om, \theta): 
 	\sup_{0 \leq t \leq T}[\|\om\|_X +\|\theta\|_X ]\leq 2 C (\|\om_0\|_X+\|\theta_0\|_X)\}$.  For the bilinear forms 
 	$$
 	Q_1(\om_1, \om_2)=\int_0^t e^{- (t- s) |\nabla|^{\al}} \nabla (u \cdot \om) \ ds, 
 	Q_2(\om_1, \theta)=\int_0^t e^{- (t- s) |\nabla|^{\al}} \nabla (u \cdot \theta) \ ds
 	$$
 	where $u=(\nabla^\perp)^{-1} \om_1$, we establish the estimates 
 	\begin{eqnarray*}
	\|Q_1(\om_1, \om_2)-Q_1(\tilde{\om}_1, \tilde{\om}_2)\|_X &\leq &  C T^{1-\f{1}{\al}} (\|(\om_1, \om_2)\|_X+ 
	\|(\tilde{\om}_1, \tilde{\om}_2)\|_X)(\|\om_1-\tilde{\om}_1\|_X+ \|\om_2-\tilde{\om}_2\|_X) \\
	\|Q_2(\om_1, \theta)-Q_2(\tilde{\om}_1, \tilde{\theta})\|_X &\leq &  C T^{1-\f{1}{\al}} (\|(\om_1, \theta)\|_X+ 
	\|(\tilde{\om}_1, \tilde{\theta})\|_X)(\|\om_1-\tilde{\om}_1\|_X+ \|\theta-\tilde{\theta}\|_X)
 	\end{eqnarray*}
 	for $j=1,2$. This is done in an identical manner as in the proof of Lemma \ref{thma}. 
 	It remains to deal with the integral term $	\int_0^t e^{- (t- s) |\nabla|^{\al}} \partial_1 \theta \ ds$, for which we have 
 	$$
 	\|\int_0^t e^{- (t- s) |\nabla|^{\al}} \partial_1 (\theta- \tilde{\theta}) \ ds\|_{L^1\cap L^{\infty}} 
 	\leq C \int_0^t  \f{1}{(\tau- s)^{\f{1}{\al}}} \|\theta- \tilde{\theta}\|_{L^1\cap L^{\infty}} ds \leq C T^{1-\f{1}{\al}} 
 	\sup_{0<s<T}  \|\theta(s)- \tilde{\theta}(s)\|_{L^1\cap L^{\infty}},
 	$$
 	for $0<t<T$. All in all, we can guarantee that with an appropriate choice of $T$, the non-linear map given by \eqref{700}has a fixed point $\om, \theta$ in the space $X$. 
 	
 	Regarding the global well-posedness, we can continue the solution, as long as the norm $t\to \|\theta(t, \cdot)\|_{L^p}$ stay under control. First, for $1<p<\infty$, take dot product of the $\theta$ equation with $|\theta|^{p-2} \theta$, $p\in (1, \infty)$ and using the fact the positivity estimate \eqref{a:20}, we obtain 
 	$$
\f{1}{p}\p_t \|\theta(t, \cdot))\|^{p}_{L^p}  \leq 	\f{1}{p}\p_t \|\theta\|^{p}_{L^p}+ \int_{\rtwo} |\theta|^{p-2} \theta \cdot |\nabla|^\al \theta dx =0
 	$$
 	It follows that $t\to \|\theta(t, \cdot)\|_{L^p}$ is non-increasing in any interval $(0,t)$, whence the solution is global and 
 $	\|\theta(t, \cdot)\|_{L^p}\leq \|\theta_0\|_{L^p}$. For $p=1, p=\infty$, we use approximation arguments to establish the same result.

 Finally, we use this information to establish the global well-posedness of the $u$ equation in \eqref{BSQ10}. Taking dot product with  $u$, we obtain 
 $$
\f{1}{2} \p_t \|u(t, \cdot)\|_{L^2}^2 \leq  \f{1}{2} \p_t \|u(t, \cdot)\|_{L^2}^2 + \||\nabla|^{\f{\al}{2}} u\|_{L^2}^2 =\dpr{u_2}{\theta}\leq \|u_2\|_{L^2} \|\theta(t)\|_{L^2}\leq \|u_2(t)\|_{L^2} \|\theta_0\|_{L^2}
 $$
 	It follows that 
 	$$
 	\|u(t, \cdot)\|_{L^2}\leq \|u_0\|_{L^2}+ t \|\theta_0\|_{L^2},
 	$$
 	which provides the necessary bound to conclude global regularity, as stated. 
 \end{proof}
 The next Lemma provides a global existence and uniqueness result  for the $(\om, \theta)$ system. 
 \begin{lemma}
 	\label{L_-11}
 	Let $\al > 1$. Then, assuming $\om_0\in L^2, \theta_0\in H^{\f{\al}{2}}$, the Cauchy problem \eqref{BSQ1} has unique global solutions. In addition, for any $T> 0$,  there exists $C=C_{T, \|\om_0\|_{L^2}, \|\theta_0\|_{H^{\f{\al}{2}}}}> 0$,  so that the solutions satisfy 
 	\begin{eqnarray}
 	\label{710} 
 	\sup_{0 \leq t \leq T} \| \om\|_{L^2}+ \sup_{0 \leq t \leq T} \| |\nabla|^{\f{\al}{2}}\theta\|_{L^2} &\leq& C. 
 	\end{eqnarray}
 \end{lemma}
 {\bf Remark:} The constant $C_T$ obtained in this argument is exponential in $T$, which is very non-efficient. On the other hand, it is sufficient for our purposes in bootstrapping the solution. 
 \begin{proof}
 	The global regularity for \eqref{BSQ1} is of course very similar to the global regularity established in Lemma \ref{thma1}. 
  For the energy estimates, needed for \eqref{710}, 	
  we can dot product the first equation in \eqref{BSQ1} with $\om$ and the second one with $|\nabla|^{\al} \theta$ to get the following energy estimate
  \begin{eqnarray*}
  \f{1}{2} \f{d}{dt}\bigg(\|\om\|^2_{L^2}+ \||\nabla|^{\f{\al}{2}}\theta\|^2_{L^2}\bigg) +  \| |\nabla|^{\f{\al}{2}} \om\|^2_{L^2}+ \| |\nabla|^{\al}\theta\|_{L^2}^2 &\leq&  \bigg| \int \om \cdot \partial_1 \theta d \xi \bigg|+ 
  \bigg| \langle  [|\nabla|^{\f{\al}{2}}, u\cdot  \nabla] \theta, |\nabla|^{\f{\al}{2}} \theta\rangle\bigg|\\
  &:=& I_1+ I_2.
  \end{eqnarray*}
 	Then for some $0 < \ga< 1$,
 	\begin{eqnarray*}
 		I_1= \bigg| \int \om \cdot \partial_1 \theta d \xi \bigg| &\leq& \| |\nabla|^{\f{\al}{2}} \om\|_{L^2} \|\partial_1 |\nabla|^{-\f{\al}{2}} \theta\|_{L^2} \leq \epsilon \| |\nabla|^{\f{\al}{2}} \om\|^2_{L^2}+ C_{\epsilon} \|\partial_1 |\nabla|^{-\f{\al}{2}} \theta\|^2_{L^2}\\
 		& \leq&  \epsilon \| |\nabla|^{\f{\al}{2}} \om\|^2_{L^2}+ C_{\epsilon} \| |\nabla|^{\al} \theta\|^{2 \ga}_{L^2} \| \theta\|^{2(1- \ga)}_{L^2} \leq \epsilon \| |\nabla|^{\f{\al}{2}} \om\|^2_{L^2}+ \epsilon \| |\nabla|^{\al} \theta\|^{2}_{L^2}+ C_{\epsilon} \| \theta_0\|^2_{L^2}. 
 	\end{eqnarray*}
 	We also have
 	\begin{eqnarray*}
 		I_2=\bigg| \langle  [|\nabla|^{\f{\al}{2}}, u \cdot  \nabla] \theta, |\nabla|^{\f{\al}{2}} \theta\rangle\bigg| &\leq& \| |\nabla|^{- \f{\al}{2}} [|\nabla|^{\f{\al}{2}}, u\cdot  \nabla] \theta\|_{L^2} \| |\nabla|^{\al} \theta\|_{L^2} 
 	\end{eqnarray*}
 	We can make use of the inequality \eqref{201} with $a= 1, s_1= s_2= \f{\al}{2}, p= 2, q= \f{8}{4- \al}$ and $r= \f{8}{\al}$ to get
 	\begin{eqnarray*}
 		\| |\nabla|^{- \f{\al}{2}} [|\nabla|^{\f{\al}{2}}, u \cdot  \nabla] \theta(t)\|_{L^2} &\leq& C \|\theta\|_{L^{\f{8}{\al}}} 
 		\| \nabla u \|_{L^{\f{8}{4- \al}}} \leq 
 		C \|\theta_0\|_{L^{\f{8}{\al}}} \| \om \|_{L^{\f{8}{4- \al}}} 		
 		 \leq C \|\theta_0\|_{L^{\f{8}{\al}}} \| |\nabla|^{\f{\al}{4}} \om\|_{L^2}\\
 		 &\leq & C \|\theta_0\|_{L^{\f{8}{\al}}} \| |\nabla|^{\f{\al}{2}} \om\|_{L^2}^{\f{1}{2}} \|\om\|_{L^2}^{\f{1}{2}}.
 	\end{eqnarray*}
 	where we have used the Sobolev embedding and the Gagliardo-Nirenberg's inequality.  
 	Then,
 	\begin{eqnarray*}
 		I_2 \leq  \epsilon \| |\nabla|^{\f{\al}{2}} \om\|^2_{L^2}+ \epsilon \| |\nabla|^{\al} \theta\|^{2}_{L^2}+ C_\eps(\|\theta_0\|_{L^{\f{8}{\al}}}\|\om\|_{L^2}^{\f{1}{2}})^4.
 	\end{eqnarray*}
 	Therefore, for $\eps<\f{1}{2}$, we can hide the terms $\| |\nabla|^{\f{\al}{2}} \om\|^2_{L^2}$ and $\| |\nabla|^{\al} \theta\|^{2}_{L^2}$ and we obtain 
 	\begin{eqnarray*}
 		 	\f{d}{dt} \bigg(\|\om\|^2_{L^2}+  \||\nabla|^{\f{\al}{2}}\theta\|^2_{L^2}\bigg)  
 		 	\leq C \|\theta_0\|_{L^{\f{8}{\al}}}^4  \bigg(\|\om\|^2_{L^2}+  \||\nabla|^{\f{\al}{2}}\theta\|^2_{L^2}\bigg)  + C \|\theta_0\|_{L^2}^2.
 	\end{eqnarray*}
 	We use Gronwall's to conclude \eqref{710}.

 \end{proof}

 \subsection{Some  a priori estimates  for the scaled vorticity Boussinesq problem $(W, \Theta)$ in $L^p$}
 
 We now turn our attention to the scaled vorticity system. By the results of Lemma \ref{L_-11} and Lemma \ref{L_-12}, such solutions exist globally, by virtue of the change of variables.  Now that we have a global solution, together with the global estimate \eqref{dec:1}, we can actually obtain global {\it a priori} estimates for $\Theta$ in all $L^p$ spaces. 
 \begin{lemma}
 	\label{L_-12} 
 	Let $p \geq 1$,   and $ \Theta_0 \in L^1\cap L^\infty(\rtwo)\cap H^\al(\rtwo)$, $W_0\in L^2$. Then for any $\tau>0$, $\Theta \in C^0([0, \tau]; L^p)$,  there  
 	exists $C= C_{\al,p}$ such that 
 	\begin{equation}
 	\label{800}
 	\|\Theta(\tau, \cdot)\|_{L^p} \leq  C_{\al,p}  \|\Theta_0\|_{L^p(\rtwo)} e^{(2-\f{1}{\al}- \f{2}{\al p}) \tau}.
 	\end{equation}
 \end{lemma} 
 \begin{proof}
  We take a dot product of the $\Theta$ equation in \eqref{81}with $|\Theta|^{p-2} \Theta$, $p\geq 1$.  We obtain 
  $$
  \f{1}{p} \p_\tau \|\Theta\|_{L^p}^p+ \int_{\rtwo} |\nabla|^\al \Theta |\Theta|^{p-2} \Theta d\xi =(2-\f{1}{\al}- \f{2}{\al p}) \|\Theta\|_{L^p}^p.
  $$
  Recall however that $\int_{\rtwo} |\nabla|^\al \Theta |\Theta|^{p-2} \Theta d\xi\geq 0$, by Lemma \ref{le:90}. Thus, integrating this inequality  yields \eqref{800}. 
 \end{proof}
 Lemma \ref{L_-12} provides us with a decay rate for $\Theta(\tau, \cdot)$ for $1 \leq p < \f{2}{2 \al- 1}$, but clearly an increasing exponential bound for $p \geq \f{2}{2 \al- 1}$. However, we can use it  to get a decay rate for any $p \geq 1$. 
  \begin{lemma}
  		\label{L_-122} 
 	Let $p \geq 1$,   and $ \Theta_0 \in L^1\cap L^\infty(\rtwo)\cap H^\al(\rtwo)$, $W_0\in L^2$. Then for any $\tau>0$, $\Theta \in C^0([0, \tau]; L^p)$,  there  
 	exists $C= C_{\al,p}$ such that 
 	\begin{equation}
 	\label{8000}
 	\|\Theta(\tau, \cdot)\|_{L^p} \leq  C_{\al,p}  \|\Theta_0\|_{L^p(\rtwo)} e^{(2 - \f{3}{\al}) \tau}.
 	\end{equation}
 \end{lemma}
\begin{proof}
Similar to Lemma  \ref{lem3}, we have the following energy estimate 
	  $$
	\f{1}{p} \p_\tau \|\Theta \|_{L^p}^p+  c_{p,\al} \|\Theta\|_{L^{\f{2p}{2- \al}}}^{p} \leq (2-\f{1}{\al}- \f{2}{\al p}) \|\Theta\|_{L^p}^p.
	$$
	In other words 
	\begin{eqnarray*}
 \p_\tau \|\Theta\|_{L^p}^p&+& p c_{p,\al}  \|\Theta\|_{L^{\f{2p}{2- \al}}}^{p} \leq p (2-\f{1}{\al}- \f{2}{\al p}) \|\Theta\|_{L^p}^p \leq  (2-\f{1}{\al}- \f{2}{\al p})  \|\Theta\|^{\ga p}_{L^{\f{2p}{2- \al}}} \|\Theta\|_{L^1}^{(1- \ga)p}\\
	&\leq& \f{[p (2-\f{1}{\al}- \f{2}{\al p})]^{\f{1}{1- \ga}}}{\eps^{^{\f{\ga}{1- \ga}}}}  \|\Theta\|^{\ga p}_{L^{\f{2p}{2- \al}}} \|\Theta\|_{L^1}^{(1- \ga)p} \leq p (2-\f{1}{\al}- \f{2}{\al p}) \|\Theta\|_{L^1}^{p}+ \eps  \|\Theta\|^{ p}_{L^{\f{2p}{2- \al}}}. 
	\end{eqnarray*}
Now we use \eqref{800} with $p=1$ to get the following energy estimate 
	\begin{eqnarray*}
	\p_\tau \|\Theta\|_{L^p}^p&+& (pc_{p,\al}   - \eps) \|\Theta\|_{L^{\f{2p}{2- \al}}}^{p}  \leq p (2-\f{1}{\al}- \f{2}{\al p}) \|\Theta\|_{L^1}^{p} \leq p (2-\f{1}{\al}- \f{2}{\al p}) e^{p (2- \f{3}{\al}) \tau}. 
\end{eqnarray*}
Finally, we use Gronwall's inequality to finish the proof. 
\end{proof}

We can use above Lemma to find some decay rate for $U(\tau, \cdot)$. We need this to be able to get some bounds for $W$ in higher $L^p$ spaces.
\begin{lemma}
	\label{L_-1222} 
	Let  $U_0 \in L^2(\rtwo)$.  There  
	exists $C= C_{\al,p}$, such that  for any $\tau>0$, $U \in C^0([0, \tau]; L^2)$ and 
	\begin{equation}
	\label{888}
	\|U(\tau, \cdot)\|_{L^2} \leq  C_{\al,p}  \|U_0\|_{L^2(\rtwo)} e^{(2 - \f{3}{\al}) \tau}.
	\end{equation}
\end{lemma}
\begin{proof}
	If we dot product the equation \eqref{UU} with $U$ we get the following relation
	\begin{eqnarray*}
\f{1}{2} \partial_{\tau} \|U\|_{L^2}^2+ \| |\nabla|^{\f{\al}{2}}U\|_{L^2}^2= \f{1}{\al} \int (\xi \cdot \nabla U) U d\xi+ (1- \f{1}{\al}) \|U\|_{L^2}^2+ \int \theta \cdot U d\xi.
	\end{eqnarray*}
Then
	\begin{eqnarray*}
 \partial_{\tau} \|U\|_{L^2}^2+2  \| |\nabla|^{\f{\al}{2}}U\|_{L^2}^2&=& 2 (1- \f{2}{\al}) \|U\|_{L^2}^2+ \int \theta \cdot U d\xi \leq 2 (1- \f{2}{\al}) \|U\|_{L^2}^2+ \|\Theta\|_{L^2} \|U\|_{L^2}\\
 &\leq& 2 (1- \f{2}{\al}+ \eps) \|U\|_{L^2}^2+ C_{\eps} \|\Theta\|^2_{L^2} \leq 2 (1- \f{2}{\al}+ \eps) \|U\|_{L^2}^2+ C_{\eps} e^{2(2- \f{3}{\al}) \tau}.
\end{eqnarray*}
We finish the proof by the Gronwall's inequality.
\end{proof}
 The next lemma provides {\it a priori} estimates for $W$ and $\Theta$ in $L^2$ spaces, which allows  us to conclude global regularity. 
\begin{lemma}
	\label{L_-133} 
	Let   $\al \in (1, \f{3}{2})$,  $W_0\in L^2$.  Then the solution $W$ of \eqref{81}, satisfies    
	\begin{eqnarray}
	\label{dec:1}
	&&\  \|W(\tau, \cdot)\|_{L^2}+ \|\Theta(\tau, \cdot)\|_{L^2}\leq  C   e^{(2- \f{3}{\al}) \tau},\\ 
	\label{estt:1}
	&&	\sup_{0\leq \tau<\infty} \int_0^{\tau} \bigg( \||\nabla|^{\f{\al}{2}} W(s)\|_{L^2}^2+ \| |\nabla|^{\f{\al}{2}} 
	\Theta(s)\|_{L^2}^2 \bigg)  d s \leq C
	\end{eqnarray}
	for some $C=C(\|W_0\|_{L^2}, \|\Theta_0\|_{L^2}, \al)$, independent on $\tau$.  
\end{lemma}
\begin{proof}
	We dot product the first equation in \eqref{81} with $W$, and the second equation with $\Theta$. We also use the trick from  Lemma  \ref{lem3} - we add the term $A ( \|W\|_{L^2}^2+ \|\Theta\|_{L^2}^2)$, where  $A$ is a large constant to be determined. Then
	\begin{eqnarray*}
		\f{1}{2} \f{d}{d t} \bigg( \|W\|_{L^2}^2+ \|\Theta\|_{L^2}^2\bigg)&+&  A ( \|W\|_{L^2}^2+ \|\Theta\|_{L^2}^2)+  \||\nabla|^{\f{\al}{2}}W\|_{L^2}^2+ \| |\nabla|^{\f{\al}{2}} \Theta\|_{L^2}^2 \\
		&\leq& \big| \int \partial_1 \Theta W d\xi\big|+ (A+ 1- \f{1}{\al}) \|W\|_{L^2}^2+ (A+ 2- \f{2}{\al}) \|\Theta\|_{L^2}^2.
	\end{eqnarray*}
	But by Gagliardo-Nirenberg (and taking into account that $1-\f{\al}{2}<\f{\al}{2}$ ) and Young's inequalities, 
	\begin{eqnarray*}
		\big| \int \partial_1 \Theta W d\xi\big| &\leq& \| |\nabla|^{1- \f{\al}{2}} \Theta \|_{L^2}  \| |\nabla|^{ \f{\al}{2}} W \|_{L^2} \leq 
		\epsilon \| |\nabla|^{\f{\al}{2}} \Theta \|^2_{L^2}+ \epsilon \| |\nabla|^{ \f{\al}{2}} W \|^2_{L^2}+ C_{\epsilon} \| \Theta \|^2_{L^2}\\
		&\leq&	\epsilon \| |\nabla|^{\f{\al}{2}} \Theta \|^2_{L^2}+ \epsilon \| |\nabla|^{ \f{\al}{2}} W \|^2_{L^2}+ C_{\epsilon} e^{2(2- \f{3}{\al}) \tau}.
	\end{eqnarray*}
	where we have used the estimate for $\|\Theta\|_{L^2}$ from \eqref{800}, with $p=2$. We also have 
\begin{eqnarray*}
(A+ 1- \f{1}{\al}) \|W\|_{L^2}^2 &\leq& C (A+ 1- \f{1}{\al}) \|\nabla U\|_{L^2}^2 \leq C(A+ 1- \f{1}{\al}) \|U\|_{L^2}^{2\ga} \| |\nabla|^{1+ \f{\al}{2}} U\|_{L^2}^{2(1- \ga)}\\
 &\leq& C(A+ 1- \f{1}{\al}) \|U\|_{L^2}^{2\ga} \| |\nabla|^{\f{\al}{2}} W\|_{L^2}^{2(1- \ga)} \leq \epsilon \| |\nabla|^{ \f{\al}{2}} W \|^2_{L^2}+ \f{[C(A+ 1- \f{1}{\al})]^{\f{1}{1- \ga}}}{\eps^{\f{\ga}{1- \ga}}} \|U\|_{L^2}^2\\
&\leq& \epsilon \| |\nabla|^{ \f{\al}{2}} W \|^2_{L^2}+ \f{[C(A+ 1- \f{1}{\al})]^{\f{1}{1- \ga}}}{\eps^{\f{\ga}{1- \ga}}}\ e^{2(2- \f{3}{\al}) \tau}.
\end{eqnarray*}
Considering the estimate for $\|\Theta\|_{L^2}$ from \eqref{800}(with $p=2$) 
	\begin{eqnarray*}
		 \f{d}{d t} \bigg( \|W\|_{L^2}^2+  \|\Theta\|_{L^2}^2\bigg) &+& 2 A ( \|W\|_{L^2}^2+ \|\Theta\|_{L^2}^2)+  2 (1- 2 \epsilon) \||\nabla|^{\f{\al}{2}} W\|_{L^2}^2+ 2 (1-2  \epsilon) \| |\nabla|^{\f{\al}{2}} \Theta\|_{L^2}^2\\
		 &\leq&   \f{2[C(A+ 1- \f{1}{\al})]^{\f{1}{1- \ga}}}{\eps^{\f{\ga}{1- \ga}}}\ e^{2(2- \f{3}{\al}) \tau}.
	\end{eqnarray*}
	We choose $A= 2 (\f{3}{\al}- 2)$ (recall $\al<\f{3}{2}$).  Then the last relation has two consequences. First we can drop the term $2 (1- 2 \epsilon) \||\nabla|^{\f{\al}{2}}W\|_{L^2}^2+ 2 (1- 2 \epsilon) \| |\nabla|^{\f{\al}{2}} \Theta\|_{L^2}^2$, so 
	\begin{eqnarray*}
		\f{d}{d t} \bigg( \|W\|_{L^2}^2+ \|\Theta\|_{L^2}^2\bigg)+ 4 (\f{3}{\al}- 2) ( \|W\|_{L^2}^2+ \|\Theta\|_{L^2}^2)  \leq  \f{[C( \f{5}{\al}- 3)]^{\f{1}{1- \ga}}}{\eps^{\f{\ga}{1- \ga}}}\ e^{2(2- \f{3}{\al}) \tau}.
	\end{eqnarray*} 	
	and then  use the Gronwall's inequality for the following inequality and get the decay rate \eqref{dec:1}.  The second consequence is that we get 
	$$
 \int_0^\tau (\||\nabla|^{\f{\al}{2}} W(t)\|_{L^2}^2+  \| |\nabla|^{\f{\al}{2}} \Theta(t)\|_{L^2}^2) dt \leq \bigg( \|W_0\|_{L^2}^2+ 
	\|\Theta_0\|_{L^2}^2\bigg)+ \f{C_\eps}{2(\f{3}{\al}-2)}.
	$$
	This implies \eqref{estt:1}. 
\end{proof}

 We shall need some {\it a priori} estimates for $\|W\|_{L^p}$ for some $p>2$, as  these will be necessary in our subsequent considerations.  This turns out to be non-trivial.  It turns out that it is easier to control $\|W\|_{H^1}, \|\Theta\|_{H^1}$ and then use Sobolev embedding to control $\|W\|_{L^p}, \|\Theta\|_{L^p}, 1<p<\infty$.  In this way, we get the control needed, but we end up needing to require  smoother $H^1$ initial data. 
 \begin{proposition}
 	\label{prop:15} 
 	 $ W_0, \Theta_0 \in H^1$.  Then, the global solution satisfies 
 	 $W, \Theta \in C^0([0, \tau];H^1(\rtwo))$.  Moreover,  
 	\begin{eqnarray}
 	\label{560}
   \|W(\tau)\|_{H^1}+ \| \Theta(\tau)\|_{H^1} \leq C e^{(2- \f{3}{\al}) \tau}. 
 	\end{eqnarray}
 	$C=C(\|W_0\|_{H^1}, \|\Theta_0\|_{H^1}, \al)$, independent on $\tau$.  
 \end{proposition}
  \begin{proof}
 	Local well-posedness in the space $H^1$, for the original (unscaled) equations works as in Lemma \ref{L_-11}, so we omit it. Thus, we have local solutions for the scaled system as well. We now need to establish {\it a priori} estimates to show that these are global. 
 	
 	We differentiate each of the equations in \eqref{81}.  Then,  we dot product it with\footnote{Here $\p$ means either $\p_1$ or $\p_2$}  $\p W$ and $\p \Theta$ respectively. We add the two resulting equations to obtain the following energy inequality 
 	\begin{eqnarray*}
 	& & 	\f{1}{2} \f{d}{d t} \bigg( \|\p W\|_{L^2}^2+ \|\p \Theta\|_{L^2}^2\bigg)+  \||\nabla|^{\f{\al}{2}+1}W\|_{L^2}^2+ \| |\nabla|^{\f{\al}{2}+1} \Theta\|_{L^2}^2 \leq \\
 		&\leq & \big| \int \partial_1 \p \Theta \p W d\xi\big|+ (1- \f{1}{\al}) \|\p W\|_{L^2}^2+ 2(1- \f{1}{\al}) \|\p \Theta\|_{L^2}^2+
 		|\dpr{\p U \nabla W}{\p W}|+	|\dpr{\p U \nabla \Theta}{\p \Theta}|.
 	\end{eqnarray*}
 	By Gagliardo-Nirenbergs' and Young's 
 	$$
 	\|\p W\|_{L^2}^2+  \|\p \Theta\|_{L^2}^2\leq \eps(\| \nabla|^{\f{\al}{2}+1} W\|_{L^2}^2+ \|\nabla|^{\f{\al}{2}+1} \Theta\|_{L^2}^2)+ 
 	C_\eps(	\|  W\|_{L^2}^2+  \|  \Theta\|_{L^2}^2)
 	$$
 	Next, 
 	$$
 	\big| \int \partial_1 \p \Theta \p W d\xi\big|\leq C \||\nabla|^{\f{\al}{2}+1}\Theta\|_{L^2} \||\nabla|^{2-\f{\al}{2}}W\|_{L^2}\leq \eps(\| \nabla|^{\f{\al}{2}+1} W\|_{L^2}^2+ \|\nabla|^{\f{\al}{2}+1} \Theta\|_{L^2}^2)+ C_\eps 	\|  W\|_{L^2}^2,
 	$$
 	where in the last estimate we have used that $2-\f{\al}{2}<1+\f{\al}{2}$. Finally, 
 	\begin{eqnarray*}
|\dpr{\p U \cdot \nabla W}{\p W}| &=&  |\dpr{\nabla\cdot (\p U W)}{\p W}|\leq C\|\nabla|^{\f{\al}{2}+1} W\|_{L^2}
\||\nabla|^{1-\f{\al}{2}}(\p U\  W)\|_{L^2} \\
&\leq & \eps \|\nabla|^{\f{\al}{2}+1} W\|_{L^2}^2 + C_\eps \||\nabla|^{1-\f{\al}{2}}(\p U\  W)\|_{L^2} ^2 
 	\end{eqnarray*}
 	By product estimates, \eqref{kp} and Sobolev embedding
 		\begin{eqnarray*}
 & & 	\||\nabla|^{1-\f{\al}{2}}(\p U\  W)\|_{L^2} \leq C (\||\nabla|^{1-\f{\al}{2}} \p U\|_{L^{\f{8}{4-\al}}} \|W\|_{L^{\f{8}{\al}}}+ 
 	\||\nabla|^{1-\f{\al}{2}} W\|_{L^{\f{8}{4-\al}}} \|\p U\|_{L^{\f{8}{\al}}}  ) \\
 	&\leq & C \||\nabla|^{1-\f{\al}{4}} \p U\|_{L^2} \||\nabla|^{1-\f{\al}{4}} W\|_{L^2}\leq C \||\nabla|^{1-\f{\al}{4}} W\|_{L^2}^2 \leq 
 	\||\nabla|^{1+\f{\al}{2}} W\|_{L^2}^{\f{2-\f{\al}{2}}{1+\f{\al}{2}}} \|W\|_{L^2}^{\f{\f{3\al}{2}}{1+\f{\al}{2}}},
\end{eqnarray*}
where we have used $\p U\sim W$ (in all Sobolev spaces) and Gagliardo-Nirenberg's. This allows us to estimate by Young's 
$$
|\dpr{\p U \cdot \nabla W}{\p W}| \leq 2\eps \|\nabla|^{\f{\al}{2}+1} W\|_{L^2}^2+C_\eps \|W\|_{L^2}^{\f{3 \al}{\al- 1}}. 
$$
Clearly, the appropriate estimate, obtained in the same way holds for 
$$
|\dpr{\p U \nabla \Theta}{\p \Theta}|\leq 2\eps \|\nabla|^{3 \f{\al}{2}+1} \Theta\|_{L^2}^2+ C_\eps \|W\|_{L^2}^{\f{\al}{\al- 1}}. 
$$
All in all, we obtain 
$$
	\f{1}{2} \f{d}{d t} \bigg( \|\p W\|_{L^2}^2+ \|\p \Theta\|_{L^2}^2\bigg)+ (1-6 \eps)( \||\nabla|^{\f{\al}{2}+1}W\|_{L^2}^2+ \| |\nabla|^{\f{\al}{2}+1} \Theta\|_{L^2}^2) \leq C_\eps (\|W\|_{L^2}^{\f{3\al}{\al- 1}}+ \|  W\|_{L^2}^2+  \|  \Theta\|_{L^2}^2).
$$
Set $\eps=\f{1}{10}$. For every $A>0$, there is $c_{\al, A}$, so that $\||\nabla|^{\f{\al}{2}+1}W\|_{L^2}^2\geq A \|\p W\|_{L^2}^2-c_{A, \al}  \|W\|_{L^2}^2$ and similar for $\Theta$, so we end up with 
$$
\f{d}{d t} \bigg( \|\p W\|_{L^2}^2+ \|\p \Theta\|_{L^2}^2\bigg)+ A\bigg( \|\p W\|_{L^2}^2+ \|\p \Theta\|_{L^2}^2\bigg)\leq C_{A, \al}
 e^{2(2-\f{3}{\al})\tau}. 
$$
where we have used the exponential bounds from \eqref{dec:1}. Setting sufficiently large $A$, namely $A=2(\f{3}{\al}-2)$,  and applying Gronwall's yields the result. 
\end{proof}
As an immediate corollary, we have control of the $L^p$ norms for $W$. 
\begin{corollary}
	\label{cor:lp} 
	Let $W_0, \Theta_0\in H^1$. Then, for all $p\in (2, \infty)$, there is the bound 
	\begin{equation}
	\label{540}
		\|W(\tau, \cdot)\|_{L^p}\leq C(\|W_0\|_{H^1}, \|\Theta_0\|_{H^1}, \al,p)  e^{(2- \f{3}{\al}) \tau}. 
	\end{equation}
\end{corollary}

 \subsection{Global regularity for the scaled vorticity Boussinesq problem $(W, \Theta)$ in $L^2(2)\cap L^\infty(\rtwo)$}
 The next Lemma is a local well-posedness result, which is a companion to Theorem \ref{thm1}. 
  	\begin{lemma}
  		\label{thm11}
  		Suppose that $W_0, \Theta_0 \in L^2(2)\cap L^\infty$. Then, there exists time 
  		$T=T(\|(W_0, \Theta_0)\|_{L^2(2)\cap L^\infty})$, so that  the system of equation  \eqref{81} has unique local solution 
  		 $W, \Theta \in C^0([0, T]; L^2(2)\cap L^\infty)$ with $W(0)= W_0$ and  $\Theta(0)= \Theta_0$.

  	\end{lemma}
 \begin{proof}
We are looking for  strong solutions in the space $X=L^2(2)\cap L^\infty$,  that is a solutions of the following system of integral equations 
\begin{eqnarray*}
	 W(\tau)&=& e^{\tau \mathcal{L}} W_0- \int_0^{\tau} e^{(\tau- s) \mathcal{L}} \nabla (U \cdot W) ds+ \int_0^{\tau} e^{(\tau- s) \mathcal{L}} (\partial_1 \Theta) ds,\\
	 \Theta(\tau)&=& e^{\tau (\cl+1-\f{1}{\al})} \Theta_0- \int_0^{\tau} e^{(\tau- s) (\cl+1-\f{1}{\al})} \nabla (U \cdot \Theta) ds		
\end{eqnarray*}	
For the free solutions, according to \eqref{19} and \eqref{202}, 
\begin{eqnarray*}
	\|e^{\tau \mathcal{L}} W_0\|_{L^2(2)\cap L^\infty}+ \|e^{\tau (\cl+1-\f{1}{\al})} \Theta_0\|_{L^2(2)\cap L^\infty} &\leq& 
	C e^\tau (\|W_0\|_{L^2(2)\cap L^\infty}+ \|\Theta_0\|_{L^2(2)\cap L^\infty}). 
\end{eqnarray*}
For $0<T<1$, to be determined,  introduce the space  
$$
Y_T:= \{(W, \Theta): \sup_{0 \leq \tau \leq T}[\|W(\tau, \cdot)\|_X + \|\Theta(\tau, \cdot)\|_X]\leq 2 C e  (\|W_0\|_{L^2(2)\cap L^\infty}+ \|\Theta_0\|_{L^2(2)\cap L^\infty}). \}.
$$
According to \eqref{70} and \eqref{19},  
\begin{eqnarray*}
	&& \|\int_0^{\tau} e^{(\tau- s) \mathcal{L}} \nabla (U \cdot W) ds\|_{L^2(2)\cap L^\infty} \leq \int_0^{\tau} \f{e^{-\f{(\tau- s)}{\al}} 
		(e^{(1- \f{2}{\al})(\tau- s) }+e^{\tau-s})}{a(\tau- s)^{\f{1}{\al}}} \|  U \cdot W\|_{L^2(2)\cap L^\infty}\ ds \\ 
	&\leq&  C  \sup_{0 \leq \tau \leq T}\|U W\|_{L^2(2)\cap L^\infty}\int_0^{\tau} 
	\f{1}{|\tau-s|^{\f{1}{\al}}} ds
	\leq C  T^{1-\f{1}{\al}}   \sup_{0 \leq \tau \leq T}\|U\|_{L^{\infty}} \sup_{0 \leq \tau \leq T}\|W\|_{L^2(2)\cap L^\infty}. 
\end{eqnarray*}
and similarly 
\begin{eqnarray*}
	&& \|\int_0^{\tau} e^{(\tau- s) (\cl+1-\f{1}{\al})} \nabla (U \cdot \Theta) ds\|_{L^2(2)\cap L^\infty}  	\leq C  T^{1-\f{1}{\al}}   \sup_{0 \leq \tau \leq T}\|U\|_{L^{\infty}} \sup_{0 \leq \tau \leq T}\|\Theta\|_{L^2(2)\cap L^\infty}. 
\end{eqnarray*}
Recalling that $U=(\nabla^\perp)^{-1} W$, we further estimate by \eqref{31}, 
$$
\|U\|_{L^{\infty}}\leq C(\|W\|_{L^{2+\eps}}+\|W\|_{L^{2-\eps}}) \leq C \|W\|_{L^2(2)\cap L^\infty},
$$
since $L^2(2) \hookrightarrow L^{2-\eps}$ and $L^2(2)\cap L^\infty\hookrightarrow L^1 \cap L^\infty\hookrightarrow L^{2+\eps}$. Finally, 
	\begin{eqnarray*}
		&& \|\int_0^{\tau} e^{(\tau- s) \mathcal{L}} (\partial_1  \Theta) ds\|_{L^2(2)\cap L^\infty} \leq  
		C T^{1-\f{1}{\al}}    \sup_{0 \leq \tau \leq T}\|\Theta\|_{L^2(2)\cap L^\infty}.\\
	\end{eqnarray*}		
Clearly, appropriate estimate hold for the differences, whence the integral equations provide a contraction mapping in the space $Y_T$, provided, $T^{1-\f{1}{\al}} << \f{1}{2 C e  (\|W_0\|_{L^2(2)\cap L^\infty}+ \|\Theta_0\|_{L^2(2)\cap L^\infty})}$.
 \end{proof}

 Our next result provides a global regularity for the $W, \Theta$ system in the space $L^2(2)$. 
 	\begin{lemma}
 		\label{thm22}
 	The system of equations \eqref{BSQ1}, with $W_0, \Theta_0 \in X=L^2(2)\cap L^\infty$, and also $W_0, \Theta_0\in H^1(\rtwo)$ has an unique global solution, in   space $X$. There exists $C=C(\|W_0\|_X, \|\Theta\|_X)$ such that  
 		\begin{eqnarray}
 			\label{907} 
 	\sup_{0\leq \tau<\infty} 	\|W(\tau, \cdot)\|_{L^2(2)} + \|\Theta(\tau, \cdot)\|_{L^2(2)}\leq C.
 		\end{eqnarray}
 	\end{lemma}	
 	{\bf Remark:} The decay rate by a constant is very inefficient. One could improve the argument below, at a considerable technical price,  to obtain  better decay estimates. Since the results in Section \ref{sec:7} will supersede these anyway, we choose to present the simpler arguments.  
 \begin{proof}
 	The existence of a local solutions are guaranteed by Lemma \ref{thm11}. So, it remains to establish energy estimates, which keep the relevant $L^2(2)$ norms under control. Note that the unweighted portion of the norm has an  exponential decay,  by \eqref{800}and \eqref{dec:1}. So, it remains to control the weighted norms. 
 	
 	We run a preliminary argument only on the $\Theta$ variable. As usual, this is easier, due to the lack of problematic term $\p_1 \Theta$, which appears in the equation for $W$. 
We dot product the $\Theta$ equation in \eqref{81} with $|\xi|^{4} \Theta$. We have 
 	\begin{eqnarray*}
 		&&	\f{1}{2} \f{d}{d \tau} \int |\xi|^{4}  \Theta^2 d \xi+  \int |\xi|^{4} |\nabla|^{\al} \Theta \cdot \Theta d \xi
 		+ (\f{4}{\al}- 2) \int |\xi|^{4}  \Theta^2 d \xi
 		= - \int (U \cdot \nabla_{\xi} \Theta) |\xi|^4 \Theta d \xi.
 	\end{eqnarray*}	
 Then
 		\begin{eqnarray*}
 			- \int (U \cdot \nabla_{\xi} \Theta) |\xi|^4 \Theta d \xi&=& 2 \int |\xi|^2 (\xi \cdot U) \Theta^2 d \xi.
 		\end{eqnarray*}
 	But 
 	$$
 	\bigg| \int |\xi|^2 (\xi \cdot U) \Theta^2 d \xi \bigg| \leq C \int |\xi|^3 \|U\|_{L^{\infty}} |\Theta|^2 d\xi  \leq \epsilon \int |\xi|^4 |\Theta|^2 d\xi+ C \epsilon^{-3} \|U\|^4_{L^{\infty}} \|\Theta\|_{L^2}^2.
 	$$
 	Now, according to 	\eqref{30}, for every $\de>0$
 	\begin{eqnarray*}
 		\|U\|_{L^\infty}\leq C_\de(\|W\|_{L^{2+\de}}+\|W\|_{L^{2-\de}} )&\leq &  C_\de( e^{(2-\f{3}{\al})\tau}+\|W\|_{L^2}^{\f{2-2\de}{2-\de}} \|W\|_{L^1}^{\f{\de}{2-\de}}) \\
 		&\leq & C_\de+ C_\de \|W\|_{L^2(2)}^{\f{\de}{2-\de}}. 
 	\end{eqnarray*}
 	We also have 
 		\begin{eqnarray*}
 			& & 	\int |\xi|^{4} \Theta |\nabla|^{\al} \Theta   d \xi =   \dpr{|\xi|^2 |\nabla|^{\f{\al}{2}}  |\nabla|^{\f{\al}{2}} \Theta}{|\xi|^2 \Theta}=  \\
 			&=& 
 			\dpr{|\nabla|^{\f{\al}{2}}[|\xi|^2   |\nabla|^{\f{\al}{2}} \Theta]}{|\xi|^2 \Theta} - \dpr{[|\nabla|^{\f{\al}{2}},|\xi|^2]\ [ |\nabla|^{\f{\al}{2}} \Theta]}{|\xi|^2 \Theta} = \\ 
 			&=& 		\dpr{|\xi|^2   |\nabla|^{\f{\al}{2}} \Theta}{|\xi|^2 |\nabla|^{\f{\al}{2}} \Theta}  +
 			\dpr{|\xi|^2   |\nabla|^{\f{\al}{2}} \Theta}{[|\nabla|^{\f{\al}{2}},|\xi|^2]  \Theta} - \dpr{[|\nabla|^{\f{\al}{2}},|\xi|^2]\ [  |\nabla|^{\f{\al}{2}} \Theta]}{|\xi|^2 \Theta}
 			\\
 			&=& 	\int |\xi|^{4} | |\nabla|^{\f{\al}{2}} \Theta |^2 d \xi+ \langle |\xi|^2 |\nabla|^{\f{\al}{2}} \Theta, [|\nabla|^{\f{\al}{2}}, |\xi|^2] \Theta \rangle- \langle  [|\nabla|^{\f{\al}{2}}, |\xi|^2]\ [ |\nabla|^{\f{\al}{2}} \Theta], |\xi|^2  \Theta \rangle\\
 		\end{eqnarray*}	
 	Now if we define $I(\tau)=  \int |\xi|^4 \Theta^2 d\xi$, and put all above together we have the following relation
 	\begin{eqnarray*}
 		&&\f{1}{2} I'(\tau)+ \left(\f{4}{\al} - 2 - 10 \epsilon\right) I(\tau)  +    \int |\xi|^{4} ||\nabla|^{\f{\al}{2}} \Theta|^2 d \xi \\
 		&& \leq     |\langle |\xi|^2 |\nabla|^{\f{\al}{2}} \Theta, [|\nabla|^{\f{\al}{2}}, |\xi|^2] \Theta \rangle|+ |\langle  [|\nabla|^{\f{\al}{2}}, |\xi|^2] [ |\nabla|^{\f{\al}{2}} \Theta], |\xi|^2  \Theta \rangle |  + C_{\de, \eps} \|W(\tau, \cdot)\|_{L^2(2)}^{\f{4\de}{2-\de}}. 
 	\end{eqnarray*}
 	We can use Lemma \ref{L_-50} to get
 	\begin{eqnarray*}
 	& &   | \langle |\xi|^2 |\nabla|^{\f{\al}{2}} \Theta, [|\nabla|^{\f{\al}{2}}, |\xi|^2] \Theta \rangle| \leq   \| |\xi|^2 |\nabla|^{\f{\al}{2}} \Theta\|_{L^2}  \|[|\nabla|^{\f{\al}{2}}, |\xi|^2] \Theta\|_{L^2}\\
 		&\leq& \| |\xi|^2 |\nabla|^{\f{\al}{2}} \Theta\|_{L^2}  \| |\xi|^{2- \f{\al}{2}}  \Theta\|_{L^2} \leq \| |\xi|^2 |\nabla|^{\f{\al}{2}} \Theta\|_{L^2}  
 		\| |\xi|^{2}  \Theta\|^{1- \f{\al}{4}}_{L^2}  \|  \Theta\|^{ \f{\al}{4}}_{L^2}\\
 		&\leq& \epsilon (\| |\xi|^2 |\nabla|^{\f{\al}{2}} \Theta\|^2_{L^2}+  \||\xi|^2 \Theta\|_{L^2}^2)+ C_{\epsilon}.
 	\end{eqnarray*}
 For the other term we have 
  	\begin{eqnarray*}
 	& &    |\langle  [|\nabla|^{\f{\al}{2}}, |\xi|^2] [ |\nabla|^{\f{\al}{2}} \Theta], |\xi|^2  \Theta \rangle | \leq   \| |\xi|^2  \Theta\|_{L^2}  \| [|\nabla|^{\f{\al}{2}}, |\xi|^2] [|\nabla|^{\f{\al}{2}} \Theta]\|_{L^2}\\
 	&\leq& \| |\xi|^2 \Theta\|_{L^2}  \| |\xi|^{2- \f{\al}{2}}  [|\nabla|^{\f{\al}{2}} \Theta]\|_{L^2} \leq \| |\xi|^2  \Theta\|_{L^2}  
 	\| |\xi|^{2} |\nabla|^{\f{\al}{2}} \Theta\|^{1- \f{\al}{4}}_{L^2}  \|  |\nabla|^{\f{\al}{2}} \Theta\|^{ \f{\al}{4}}_{L^2}\\
 	&\leq& \epsilon (\| |\xi|^2 |\nabla|^{\f{\al}{2}} \Theta\|^2_{L^2}+  \||\xi|^2 \Theta\|_{L^2}^2)+ C_{\eps} \|  |\nabla|^{\f{\al}{2}} \Theta\|^2_{L^2}.
 \end{eqnarray*}

 	It follows that 
 	$$
 	\f{1}{2} I'(\tau)+ \left(\f{4}{\al} - 2 - 20 \epsilon\right) I(\tau)  +  (1-5 \eps)  \int |\xi|^{4} ||\nabla|^{\f{\al}{2}} \Theta|^2 d \xi \leq C_\eps+  C_{\de, \eps} \|W(\tau, \cdot)\|_{L^2(2)}^{\f{4\de}{2-\de}}+ C_{\eps} \|  |\nabla|^{\f{\al}{2}} \Theta\|^2_{L^2}. 
 	$$
 	Choose $\eps=\f{1}{200}$, apply  Gronwall's  and then  \eqref{estt:1}  implies that for every $\de>0$, there is $C_\de$, so that 
 	\begin{equation}
 	\label{899} 
 		\||\xi|^2 \Theta(\tau, \cdot)\|_{L^2}\leq C_\eps+ C_\de e^{-(\f{4}{\al}-2-\de)\tau} + C_{\de} \sup_{0<s<\tau}\|W(s, \cdot)\|_{L^2(2)}^{\f{2\de}{2-\de}}. 
 	\end{equation}
 	for every $\de>0$. In addition, we obtain the $L^2_\tau$ bound 
\begin{equation}
\label{900}
	\int_0^\tau \||\xi|^{2} |\nabla|^{\f{\al}{2}} \Theta(\tau, \cdot)\|_{L^2}^2 d \tau \leq C + C_{\de} \sup_{0<s<\tau}\|W(s, \cdot)\|_{L^2(2)}^{\f{4\de}{2-\de}}.
\end{equation}

 We are now ready for the bounds for $W$, which are always harder. 
 	If we dot product in \eqref{81}, the first equation with $|\xi|^{4} W$,   we have the energy equality 
 	\begin{eqnarray*}
 		&&	\f{1}{2} \f{d}{d \tau} \int |\xi|^{4}  W^2 d \xi+   \int |\xi|^{4} |\nabla|^{\al} W \cdot W d \xi
 		+ (\f{3}{\al}- 1) \int |\xi|^{4}  W^2 d \xi\\
 		&=& - \int (U \cdot \nabla_{\xi} W) |\xi|^4 W d \xi+  \int \partial_1 \Theta\ |\xi|^{4} W d \xi
 	\end{eqnarray*}
 	Then $- \int (U \cdot \nabla_{\xi} W) |\xi|^4 W d \xi = 2 \int |\xi|^2 (\xi \cdot U) W^2 d \xi$. 
 	We can bound this term as follows
 	\begin{eqnarray*}
 		\bigg| \int |\xi|^2 (\xi \cdot U) W^2 d \xi \bigg| &\leq& C \int |\xi|^3 \|U\|_{L^{\infty}} |W|^2 d\xi  \leq \epsilon \int |\xi|^4 |W|^2 d\xi+ C \epsilon^{-3} \|U\|^4_{L^{\infty}} \|W\|_{L^2}^2. 
 	\end{eqnarray*}
 Again, according to 	\eqref{30}, for every $\de>0$
  	\begin{eqnarray*}
  	 \|U\|_{L^\infty}\leq C_\de(\|W\|_{L^{2+\de}}+\|W\|_{L^{2-\de}} )\leq C( e^{(2-\f{3}{\al})\tau}+\|W\|_{L^2}^{\f{2-2\de}{2-\de}} \|W\|_{L^1}^{\f{\de}{2-\de}}). 
  	\end{eqnarray*}
 	Taking into account \eqref{800}, \eqref{540}, $L^2(2)\hookrightarrow L^1$  and Young's inequality,  allows us to estimate 
 	$$
 	\bigg| \int |\xi|^2 (\xi \cdot U) W^2 d \xi \bigg|   \leq 
 	2\eps \int |\xi|^4 |W|^2 d\xi +C_{\eps,\de}  \|W(\tau, \cdot)\|_{L^2(2)}^{\f{4\de}{2-\de}}. 
 	$$
 	We also have, similar to the $\Theta$ variable calculation,  	
 	\begin{eqnarray*}
 		\int |\xi|^{4} W |\nabla|^{\al} W   d \xi =   	\||\xi|^{2} | |\nabla|^{\f{\al}{2}} W\|_{L^2}^2 + \langle |\xi|^2 |\nabla|^{\f{\al}{2}} W, [|\nabla|^{\f{\al}{2}}, |\xi|^2] W \rangle- \langle  [|\nabla|^{\f{\al}{2}}, |\xi|^2] [ |\nabla|^{\f{\al}{2}} W], |\xi|^2  W \rangle\\
 	\end{eqnarray*}	 	 
Now if we take  $J(\tau)= \int |\xi|^4 W^2 d\xi$, and put all above together we have the following relation
\begin{eqnarray*}
	&&\f{1}{2} J'(\tau)+ \left(\f{3}{\al}-1- 10 \epsilon\right) J(\tau)  +   \int |\xi|^{4} ||\nabla|^{\f{\al}{2}} W|^2 d \xi  \\
	&& \leq | \langle |\xi|^2 |\nabla|^{\f{\al}{2}} W, [|\nabla|^{\f{\al}{2}}, |\xi|^2] W \rangle| + |\langle  [|\nabla|^{\f{\al}{2}}, |\xi|^2] [ |\nabla|^{\f{\al}{2}} W], |\xi|^2  W \rangle|  + \bigg| \int |\xi|^4 (\partial_1 \Theta)  W d\xi \bigg| + \\
	&+& C_\eps+C_{\eps,\de}  \|W(\tau, \cdot)\|_{L^2(2)}^{\f{4\de}{2-\de}}= I_1+ I_2+ I_3+  
	C_\eps+C_{\eps,\de}  \|W(\tau, \cdot)\|_{L^2(2)}^{\f{4\de}{2-\de}} 
\end{eqnarray*}
We can use Lemma \ref{L_-50} to get
\begin{eqnarray*}
	I_1&=&  | \langle |\xi|^2 |\nabla|^{\f{\al}{2}} W, [|\nabla|^{\f{\al}{2}}, |\xi|^2] W \rangle| \leq   \| |\xi|^2 |\nabla|^{\f{\al}{2}} W\|_{L^2}  \||[|\nabla|^{\f{\al}{2}}, |\xi|^2] W\|_{L^2}\\
	&\leq& \| |\xi|^2 |\nabla|^{\f{\al}{2}} W\|_{L^2}  \| |\xi|^{2- \f{\al}{2}}  W\|_{L^2} \leq \| |\xi|^2 |\nabla|^{\f{\al}{2}} W\|_{L^2}  
	\| |\xi|^{2}  W\|^{1- \f{\al}{4}}_{L^2}  \|  W\|^{ \f{\al}{4}}_{L^2}\\
	&\leq& \epsilon (\| |\xi|^2 |\nabla|^{\f{\al}{2}} W\|^2_{L^2}+  \||\xi|^2 W\|_{L^2}^2)+ C_{\epsilon},
\end{eqnarray*}
where we have used the bounds \eqref{dec:1} for $\|  W\|_{L^2}$. 
Next, regarding $I_2$, we have 
\begin{eqnarray*}
	I_2&=& |\langle  [|\nabla|^{\f{\al}{2}}, |\xi|^2] [ |\nabla|^{\f{\al}{2}} W], |\xi|^2  W \rangle| \leq \||\xi|^2  W\|_{L^2} \| [|\nabla|^{\f{\al}{2}}, |\xi|^2] [ |\nabla|^{\f{\al}{2}} W]\|_{L^2}  \\ 
	&\leq & \||\xi|^2  W\|_{L^2} \|  |\xi|^{2- \f{\al}{2}}  |\nabla|^{\f{\al}{2}} W\|_{L^2} \leq  \||\xi|^2  W\|_{L^2} \|  |\xi|^{2}  |\nabla|^{\f{\al}{2}} W\|^{1- \f{\al}{4}}_{L^2}  \|  |\nabla|^{\f{\al}{2}} W\|^{ \f{\al}{4}}_{L^2}  \\
	&\leq & 
	   \epsilon(\||\xi|^2 W\|^2_{L^2}+   \|  |\xi|^{2}  |\nabla|^{\f{\al}{2}} W\|^2_{L^2}) +  C_{\epsilon}  \|  |\nabla|^{\f{\al}{2}} W\|^2_{L^2}.
\end{eqnarray*}
   $I_3$ is normally a problematic term, but now we have the decay estimates for $\|\Theta\|_{L^2(2)}$, which we have proved in the first part of this Lemma. We have  
\begin{eqnarray*}
	I_3&=& \bigg| \langle \partial_1 \Theta, |\xi|^4 W  \rangle \bigg| \leq \bigg| \langle |\xi|^2 \partial_1 \Theta, |\xi|^2 W  \rangle \bigg| \leq \bigg| \langle \partial_1  |\nabla|^{-\f{\al}{2}} |\xi|^2 |\nabla|^{\f{\al}{2}} \Theta, |\xi|^2 W  \rangle \bigg|\\
	&+& \bigg| \langle [\partial_1  |\nabla|^{-\f{\al}{2}}, |\xi|^2] \ [|\nabla|^{\f{\al}{2}} \Theta], |\xi|^2 W  \rangle \bigg|:= I_{3,1}+ I_{3, 2}.
\end{eqnarray*}
 $I_{3, 1}$ is estimated as follows 
\begin{eqnarray*}
 	I_{3, 1}&=&  \bigg| \langle \partial_1  |\nabla|^{-\f{\al}{2}} |\xi|^2 |\nabla|^{\f{\al}{2}} \Theta, |\xi|^2 W  \rangle \bigg| \leq  C \||\xi|^2 |\nabla|^{\f{\al}{2}} \Theta\|_{L^2} \|  |\nabla|^{1-\f{\al}{2}} [|\xi|^2 W ]\|_{L^2}  \\
 	&\leq &  	 \||\xi|^2 |\nabla|^{\f{\al}{2}} \Theta\|_{L^2} \| |\xi|^2 W \|^{\f{2\al-2}{\al}}_{L^2}  
 	 \| |\nabla|^{\f{\al}{2}} [|\xi|^2 W ]\|_{L^2}^{\f{2-\al}{\al}}\leq C_\eps \||\xi|^2 |\nabla|^{\f{\al}{2}} \Theta\|_{L^2}^2+ \\
 	&+& \eps( \|  |\xi|^2 W \|_{L^2}^2+ \| |\nabla|^{\f{\al}{2}} [|\xi|^2 W ]\|_{L^2}^2)
\end{eqnarray*}
We bound the last term, by Lemma \ref{L_-50}, 
\begin{eqnarray*}
  \| |\nabla|^{\f{\al}{2}} [|\xi|^2 W ]\|_{L^2} &\leq & \|  |\xi|^2 |\nabla|^{\f{\al}{2}} W \|_{L^2}+  \| [|\nabla|^{\f{\al}{2}}, |\xi|^2] W \|_{L^2} \leq 
 \| |\xi|^2 |\nabla|^{\f{\al}{2}} W \|_{L^2}+ C \||\xi|^{2-\f{\al}{2}} W\|_{L^2} \\
  &\leq &  \| |\xi|^2 |\nabla|^{\f{\al}{2}} W \|_{L^2}+ C (\|W\|_{L^2}+\||\xi|^2 W\|_{L^2}).
\end{eqnarray*}
Collecting terms together yields the following estimate for $I_{3,1}$ and using \eqref{560}, 
$$
I_{3,1}\leq 2\eps(\||\xi|^2 W\|_{L^2}^2+ \| |\xi|^2 |\nabla|^{\f{\al}{2}} W \|_{L^2}^2)+ 
C_\eps \||\xi|^2 |\nabla|^{\f{\al}{2}} \Theta\|_{L^2}^2+C  e^{2(2-\f{3}{\al})\tau}.
$$

We can easily bound $I_{3, 2}$, provided we know an appropriate estimate for the commutator 
 $[\partial_1  |\nabla|^{-\f{\al}{2}}, |\xi|^2]$. In fact, this commutator is morally like $[|\nabla|^{1-\f{\al}{2}}, |\xi|^2]$, which was indeed  considered in Lemma \ref{L_-50}. However, there  does not appear to be an easy way to transfer the estimate \eqref{310} to it, so we state the relevant estimate  here 
 \begin{equation}
 \label{903} 
 \|[\partial_1  |\nabla|^{-a}, |\xi|^2] f\|_{L^2}\leq C \||\xi|^{1+a} f\|_{L^2}, \ \  a\in (0,1)
 \end{equation}
 The proof of \eqref{903} is postponed to the Appendix\footnote{In fact, it can be reduced to a similar expression as in the proof of \eqref{310}, so we prove them simultaneously.}. 
Assuming the validity of \eqref{903}, we proceed to bound $I_{3,2}$. 
\begin{eqnarray*}
	I_{3, 2}&=& \bigg| \langle [\partial_1  |\nabla|^{-\f{\al}{2}}, |\xi|^2] \ [|\nabla|^{\f{\al}{2}} \Theta], |\xi|^2 W  \rangle \bigg| \leq \| |\xi|^2 W\|_{L^2} \| [\partial_1  |\nabla|^{-\f{\al}{2}}, |\xi|^2] \ [|\nabla|^{\f{\al}{2}} \Theta]\|_{L^2}\\
	&\leq& \| |\xi|^2 W\|_{L^2} \| |\xi|^{1+ \f{\al}{2}} |\nabla|^{\f{\al}{2}} \Theta\|_{L^2} \leq \| |\xi|^2 W\|_{L^2} \| |\xi|^{2} |\nabla|^{\f{\al}{2}} \Theta\|^{\f{2+ \al}{4}}_{L^2} \| |\nabla|^{\f{\al}{2}} \Theta\|^{\f{2- \al}{4}}_{L^2} \\
	&\leq& \epsilon \| |\xi|^2 W\|^2_{L^2}+  \| |\nabla|^{\f{\al}{2}}\Theta\|^2_{L^2}+C_\epsilon \| |\xi|^{2} |\nabla|^{\f{\al}{2}} \Theta\|^2_{L^2}\\
	&\leq & 
	\epsilon \| |\xi|^2 W\|^2_{L^2}+ C+C_\de \|  W\|_{L^2(2)}^{\f{4\de}{2-\de}}+  C_\epsilon \| |\xi|^{2} |\nabla|^{\f{\al}{2}} \Theta\|^2_{L^2},
\end{eqnarray*}
where we have made use of \eqref{900}. 
Combining all the estimates, we obtain the following energy inequality 
\begin{eqnarray*}
	&&\f{1}{2} J'(\tau)+ \left(\f{3}{\al}-1- 20 \epsilon\right) J(\tau)  +  (1-5 \eps)  \int |\xi|^{4} ||\nabla|^{\f{\al}{2}} W|^2 d \xi  \\
	&\leq & 
	C_\eps +C_\de \|  W\|_{L^2(2)}^{\f{4\de}{2-\de}}+C_\epsilon (\| |\xi|^{2} |\nabla|^{\f{\al}{2}} \Theta\|^2_{L^2}+  \|  |\nabla|^{\f{\al}{2}} W\|^2_{L^2})
\end{eqnarray*}
Applying Gronwall's and taking into account the $L^2_\tau$ integrability results \eqref{estt:1} and  \eqref{900}, and 
$ \|  W\|_{L^2(2)}^2\leq J(\tau)+C$,   we conclude for every $\de>0$ 
\begin{eqnarray*}
 J(\tau) &\leq &  J(0) e^{-2(\f{3}{\al}-1- 20 \epsilon)\tau}+ C_\eps \tau e^{-2(\f{3}{\al}-1- 20 \epsilon)\tau}+ C_\de \sup_{0<s<\tau} J(\tau)^{\f{2\de}{2-\de}} +  \\
 &+& C_\eps \int_0^\tau (\| |\xi|^{2} |\nabla|^{\f{\al}{2}} \Theta(s, \cdot)\|^2_{L^2}+  \|  |\nabla|^{\f{\al}{2}} W(s, \cdot)\|^2_{L^2}) 
 ds \leq C_\eps + C_\de \sup_{0<s<\tau} J(\tau)^{\f{2\de}{2-\de}}
\end{eqnarray*}
Selecting small $\eps$ and solving this inequality for $\sup_{0<s<\tau} J(\tau)$ implies the $\sup_{0<s<\tau} J(\tau)\leq C$, for all times $\tau$. Inputting this last estimate in  \eqref{899} implies the desired  bound for $\|\Theta\|_{L^2(2)}$ as well. 
 \end{proof}

 \section{Global dynamics of the solutions of the SQG model}
 \label{sec:6} 
 Theorem \ref{thm1} already provides pretty good estimate about the behavior of the solutions to the rescaled equation \eqref{8}, in particular the solution $Z$  disperses at $\infty$, with the rate $e^{- \tau (\f{3-\al-\be}{\al})}$. An obviously   question is whether  or not this is optimal, that is whether there is a lower bound with  the same exponential function, at least for generic data. It turns out that this is indeed the case. In fact, we have a more precise result, namely an asymptotic expansion. 
 
 Before we continue with the formal statement of the main result, we need a simple algebraic  observation, which is important in the sequel. Recall the generalized Biot-Savart law that we imposed, $   u=u_z=(|\nabla|^{\perp})^{- \be} z$. This naturally transformed into the relation $U=U_Z=(|\nabla|^{\perp})^{- \be} Z$ between the ``scaled'' velocity $U$ and its vorticity $Z$. We claim that 
 \begin{equation}
 \label{400}
 U_G\cdot \nabla G=0.
 \end{equation}
 Indeed, since $G$ is a radial function\footnote{as the Fourier transform of a radial one}, say $G(\xi)=\zeta(|\xi|)$, we have that $\nabla G= (\xi_1, \xi_2) \f{\zeta'(|\xi|)}{|\xi|}.$   On the other hand, $U_G=(|\nabla|^{\perp})^{- \be} G= |\nabla|^{\perp} m_{-\be-1}(|\nabla|) G$, so $U_G= |\nabla|^{\perp} h(|\xi|)$, where $h$ is a radial function representing $[m_{-\be-1}(|\nabla|) G]$. That is,  $h(|\xi|)=[m_{-\be-1}(|\nabla|) G](\xi)$. It follows that $U_G=(-\xi_2, \xi_1) \f{h'(|\xi|)}{|\xi|}$. Thus, 
 $$
  U_G\cdot \nabla G=(-\xi_2, \xi_1) \f{h'(|\xi|)}{|\xi|} \cdot (\xi_1, \xi_2) \f{\zeta'(|\xi|)}{|\xi|}=0.
 $$
 We are now ready to state the main theorem of this section. 
 \begin{theorem}
 	\label{theo:4} 
 	Let $Z_0\in L^2(2)\cap L^\infty(\rtwo)$, $\eps>0$, $\al\in (1,2), \al+\be\leq 3$.  Denote $\ga(0):= \int_{\rtwo} Z_0(\xi) d \xi$. Then there exists $C_\eps > 0$ such that for any $\tau > 0$,
 	\begin{equation}
 	\label{450}
 	\|Z(\tau, \cdot) - \ga(0) e^{-\tau (\f{3-\al-\be}{\al})} G\|_{L^2(2)}  \leq C_\eps  e^{-\tau (\f{4-\al-\be}{\al}-\eps)}.
 	\end{equation}
Assuming in addition that $\be>1$, we also have 
 		\begin{equation}
 		\label{452}
 		\|\nabla[Z(\tau, \cdot) - \ga(0) e^{-\tau (\f{3-\al-\be}{\al})} G]\|_{L^2(2)}  \leq C_\eps  e^{-\tau (\f{4-\al-\be}{\al}-\eps)}.
 		\end{equation}
 	In particular if  $\int_{\rtwo} Z_0(\xi) d \xi= 0$, then
 $
 	\|Z\|_{L^2(2)} \leq C_\eps  e^{-\tau (\f{4-\al-\be}{\al}-\eps)}.
 $
 \end{theorem}
 {\bf Remarks:} 
 \begin{itemize}
 	\item We would like to point out that the existence of solution $Z$ (and subsequently $\ga(\tau)$ and $\tilde{Z}(\tau)$) is not in question anymore, due to the results obtained in Theorem \ref{thm1}. The purpose of this theorem is just to obtain better {\it a priori}  estimates, in the form described in above.  
 	
 	\item The requirement $\be>1$, imposed so that \eqref{452} holds is likely only a technical one, but we cannot remove it with our methods.  	
 \end{itemize}

 \begin{proof}(Theorem \ref{theo:4}) 
 	
According to the results in Section \ref{sec:3.6}, $\la_0= 	-\f{3-\al-\be}{\al} \leq 0$ is an isolated and simple eigenvalue for the operator $\cl$ on $L^2(2)$, with eigenfunction $G$, while the rest of the spectrum is the essential spectrum, which we have identified before,  $\si_{ess}(\cl)=\{\la: \Re \la\leq -\f{4-\al-\be}{\al}\}$. We have also found  the spectral projection $\cp_0 f=\dpr{f}{1} G$ and $\cq_0=Id-\cp_0$. Thus, we can write 
\begin{equation}
\label{f101}
Z(\tau, \cdot)= \ga (\tau) G(\xi)+ \tilde{Z}(\tau, \cdot),
\end{equation}
 	where $\ga(\tau)=\dpr{Z(\tau, \cdot)}{1}=\int_{\rtwo} Z(\tau, \xi)d\xi$, $ \tilde{Z}(\tau)=\cq_0 Z(\tau, \cdot)$. Projecting the equation \eqref{8}, with respect to the spectral decomposition provided by $\cp_0$ and $\cq_0$, we obtain an ODE for $\ga$ and a PDE for 
 	$\tilde{Z}(\tau)$. More precisely, 
 	\begin{eqnarray*}
	\p_\tau \ga &=&  \dpr{\cl Z}{1} - \dpr{U\cdot \nabla Z}{1}= \dpr{- |\nabla|^{\al} Z+ \f{1}{\al} \xi \cdot \nabla_{\xi} Z+ \left(1+\f{\be-1}{\al}\right)Z}{1} - \dpr{\nabla(U\cdot  Z)}{1}=\\
	&=& \f{\al+\be-3}{\al} \ga(\tau).
 	\end{eqnarray*}
 Integrating this first order ODE yields the formula  $ \ga(\tau)=\ga(0) e^{-\tau\f{3-\al-\be}{\al}}$.  
 For the PDE governing $\tilde{Z}(\tau)$, and recalling $\cl_0=\cl\cq_0$, we   obtain 
 	\begin{equation*}
 	\tilde{Z}_{\tau}= \mathcal{L}_0 \tilde{Z}- \cq_0[U \cdot \nabla Z]= \mathcal{L}_0 \tilde{Z}- \cq_0[U \cdot 
 	\nabla ( \ga(0) e^{-\tau\f{3-\al-\be}{\al}} G+ \tilde{Z})].
 	\end{equation*}
  In its equivalent integral formulation, 
 \begin{equation}
 \label{410}
 \tilde{Z} (\tau)= e^{\tau \mathcal{L}_0} \tilde{Z}(0)- \int_0^{\tau} e^{(\tau- s) \mathcal{L}_0} \cq_0[U \cdot 
 \nabla (\ga (0)\ e^{-\tau\f{3-\al-\be}{\al}}  G+ \tilde{Z} (s, \cdot)] \ ds.
 \end{equation}	
 	Note  the commutation relation  $\cq_0 \nabla = \nabla$, whence one can remove $\cq_0$ in front of the nonlinearity. 	By \eqref{405}, we can estimate  
 	\begin{eqnarray*}
 		&&\|\tilde{Z}(\tau)\|_{L^2(2)} \leq \|e^{\tau \mathcal{L}_0} \tilde{Z}(0)\|_{L^2(2)}+ \int_0^{\tau}  
 		\|  e^{(\tau- s) \mathcal{L}_0} \bigg( (U_G+ U_{\tilde{Z}}) \nabla \cdot  (\ga (0)\ e^{-s\f{3-\al-\be}{\al}} G+ \tilde{Z}(s) \bigg) \|_{L^2(2)} ds\\
 		&\leq & \|e^{\tau \mathcal{L}_0} \tilde{Z}(0)\|_{L^2(2)}+ |\ga(0)| \int_0^{\tau} e^{- \frac{ (\tau- s)}{\al}}  e^{-s\f{3-\al-\be}{\al}} 
 		   \| \nabla \cdot e^{(\tau- s) \mathcal{L}_0}  (U_{\tilde{Z}} \cdot G) \|_{L^2(2)} ds+\\
 		&+&  \int_0^{\tau} e^{- \frac{ (\tau- s)}{\al}} \|\nabla \cdot e^{(\tau- s) \mathcal{L}} (U \cdot \tilde{Z}) \|_{L^2(2)} ds =: I_1+ I_2+ I_3,
 	\end{eqnarray*}
 	where we have used \eqref{400}. Clearly by \eqref{18}
 	\begin{eqnarray*}
 		I_1 \leq C e^{-\tau\left(\f{4-\be-\al}{\al}-\eps \right)}  \|\tilde{Z}(0)\|_{L^2(2)}  
 	\end{eqnarray*}
 	Regarding $I_2$, we have 
 	\begin{eqnarray*}
 		I_2 &\leq& |\ga(0)|  \int_0^{\tau} \f{e^{- \frac{ (\tau- s)}{\al}} e^{-s\f{3-\al-\be}{\al}} e^{-(\tau-s)\left(\f{3-\be-\al}{\al} \right)}  \| U_{\tilde{Z}} \cdot G \|_{L^2(2)}}{a(\tau- s)^{\f{1}{\al}}} ds 
 	\end{eqnarray*}
 	Now to bound $\| U_{\tilde{Z}} \cdot G\|_{L^2(2)}$ we look at two different cases, namely $0 \leq \be < 1$ and $1 \leq \be < 2$.
If  $0 \leq \be \leq  1$,  then we can use Lemma \ref{GG} to get
 	\begin{eqnarray*}
 		\| U_{\tilde{Z}} \cdot G\|_{L^2(2)}  &   \leq &  \|U_{\tilde{Z}}\|_{L^{\f{2}{1- \be}}} \| (1+ |\xi|^2) G\|_{L^{\f{2}{\be}}} \\
 		&\leq&C  \|U_{\tilde{Z}}\|_{L^{\f{2}{1- \be}}} \leq C \| |\nabla|^{\be} U_{\tilde{Z}}\|_{L^2} \leq C \|\tilde{Z}\|_{L^2} \leq \|\tilde{Z}\|_{L^2(2)} .
 	\end{eqnarray*}
 	If $1 \leq \be < 2$, then for some $0 <\epsilon << 1$ we have
 	\begin{eqnarray*}
 		\| U_{\tilde{Z}} \cdot G\|_{L^2(2)}  & \leq & \|U_{\tilde{Z}}\|_{L^{\f{2}{\epsilon}}} \| (1+ |\xi|^2)  G\|_{L^{\f{2}{1- \epsilon}}} \\
 		&\leq&C  \|U_{\tilde{Z}}\|_{L^{\f{2}{\epsilon}}} \leq C \| |\nabla|^{\be} U_{\tilde{Z}}\|_{L^{\f{2}{\be+ \epsilon}}} \leq C \|\tilde{Z}\|_{L^{\f{2}{\be+ \epsilon}}} \leq C \|\tilde{Z}\|_{L^2(2)}.
 	\end{eqnarray*}
 	In the last inequality we used the fact that for $1< p < 2$, $L^p \hookrightarrow L^2(2)$ and Lemma \eqref{GG}. 
 	Therefore 
 	\begin{eqnarray*}
 		I_2 \leq C  \int_0^{\tau}  \f{  e^{-(\tau-s)\left(\f{4-\be-\al}{\al} \right)}  e^{-s\f{3-\al-\be}{\al}}
 			 }{(\min(1, |\tau-s|)^{\f{1}{\al}}} \| \tilde{Z}(s)\|_{L^2(2)}ds. 
 	\end{eqnarray*}
 	Finally, we make use of \eqref{19} to get
 	\begin{eqnarray*}
 		I_3 &\leq& \int_0^{\tau}  \f{e^{- \frac{ (\tau- s)}{\al}} e^{-(\tau-s)\left(\f{3-\be-\al}{\al} \right)}  \|U(s)\|_{L^{\infty}} \|\tilde{Z}(s)\|_{L^2(2)}}{a(\tau-s)^{\f{1}{\al}}} \ ds \\
 		&\leq &  C \int_0^{\tau}  \f{ e^{-(\tau-s)\left(\f{4-\be-\al}{\al}-\eps \right)} 
 			e^{-s\left(\f{3-\be-\al}{\al}\right)}    }{(\min(1, |\tau-s|)^{\f{1}{\al}}} \|\tilde{Z}(s)\|_{L^2(2)}\ ds,
 		  	\end{eqnarray*}
 where we have  used that $a(\tau)\sim \min(1, \tau)$, the Sobolev inequality and Theorem \ref{thm1} to conclude  
 \begin{equation}
 \label{u:12} 
  \|U(s)\|_{L^{\infty}} \leq C(\|Z(s)\|_{L^{\f{2}{\be}+\eps}} + \|Z(s)\|_{L^{\f{2}{\be}-\eps}})\leq C  e^{-s\left(\f{3-\be-\al}{\al}\right)}.
 \end{equation}

  We are now in a position to use the Gronwal's inequality, more precisely the version displayed in Lemma \ref{gro}.  We apply it with $I(\tau)= \|\tilde{Z}(\tau)\|_{L^2(2)}$, $\mu=\f{4-\al-\be}{\al}-\eps, \si=\f{4-\al-\be}{\al}, \ka=\f{3-\al-\be}{\al}$ and $a=\f{1}{ \al}<1$, for $\eps<<1$. Recall that by the {\it a priori} estimates in Theorem \ref{thm1}, we have 
  $$
  \|\tilde{Z}(\tau)\|_{L^2(2)}\leq \|Z(\tau)\|_{L^2(2)}+ |\ga(0)| e^{-\tau(\f{3-\al-\be}{\al})} \|G\|_{L^2(2)}\leq C e^{-\tau(\f{3-\al-\be}{\al})}\leq C,
  $$
  for all $\tau>0$, since $3\geq \al+\be$.  Thus, all the requirements of  Lemma \ref{gro} are met and we obtain the bound 
  \begin{equation}
  \label{z} 
   \|\tilde{Z}(\tau)\|_{L^2(2)}\leq C_\eps e^{-\tau(\f{4-\al-\be}{\al}-\eps)}.
  \end{equation}
  
Regarding the proof of \eqref{452}, we proceed in a similar fashion. We need to control $\|\p \tilde{Z}\|_{L^2(2)}$, for large $\tau$, say $\tau\geq 1$. Applying $\p=\p_1, \p_2$ to the integral equation \eqref{410} and taking $\|\cdot\|_{L^2(2)}$, we obtain 
\begin{eqnarray*}
& &  \|\p \tilde{Z}(\tau)\|_{L^2(2)}  \lesssim  e^{-\tau(\f{4-\al-\be}{\al}-\eps)} \|\tilde{Z}(0)\|_{L^2(2)} +   \int_0^{\tau} \f{e^{- \frac{ (\tau- s)}{\al}}  e^{-s\f{3-\al-\be}{\al}} }{\min(1, \tau-s)^{\f{1}{\al}}} 
	\| e^{(\tau- s) \mathcal{L}_0}  \nabla (U_{\tilde{Z}} \cdot G) \|_{L^2(2)} ds\\
	&+&  \int_0^{\tau} \f{e^{- \frac{ (\tau- s)}{\al}}}{\min(1, \tau-s)^{\f{1}{\al}}}\| e^{(\tau- s) \mathcal{L}_0} \nabla (U \cdot \tilde{Z}) \|_{L^2(2)} ds\lesssim e^{-\tau(\f{4-\al-\be}{\al}-\eps)}+ \\
	&+& 
	\int_0^{\tau}  \f{e^{-s\f{3-\al-\be}{\al}} e^{-(\tau-s)(\f{5-\al-\be}{\al}-\eps)}}{\min(1, \tau-s)^{\f{1}{\al}}} \|\nabla[U_{\tilde{Z}}(s) G]\|_{L^2(2)} ds+	\int_0^{\tau}  \f{ e^{-(\tau-s)(\f{5-\al-\be}{\al}-\eps)}}{\min(1, \tau-s)^{\f{1}{\al}}} \|\nabla[U(s) \tilde{Z}(s)]\|_{L^2(2)} ds
\end{eqnarray*}
We estimate 
$
	 \|\nabla[U_{\tilde{Z}}(s) G]\|_{L^2(2)}\leq  \|\nabla U_{\tilde{Z}}(s) G\|_{L^2(2)}+
	 \|U_{\tilde{Z}}(s) \nabla  G\|_{L^2(2)}. 
$
Following the strategy above, for $\be\leq 1$ and then for $\be>1$, we arrive at 
$$
 \|\nabla[U_{\tilde{Z}}(s) G]\|_{L^2(2)}\lesssim \|\tilde{Z}(s)\|_{L^2(2)}+\|\p \tilde{Z}(s)\|_{L^2(2)}\lesssim e^{-s(\f{4-\al-\be}{\al}-\eps)}+ \|\p \tilde{Z}(s)\|_{L^2(2)},
$$
where we have used \eqref{z}. For the other term, it is relatively easy to bound 
$\|\nabla[U(s) \tilde{Z}(s)]\|_{L^2(2)}$, when $\be>1$, 
\begin{eqnarray*}
\|\p[U(s) \tilde{Z}(s)]\|_{L^2(2)}&\lesssim &  \|\p U(s)\|_{L^\infty} \|\tilde{Z}(s)\|_{L^2(2)}+ \|U(s)\|_{L^\infty} 
\|\p \tilde{Z}(s)\|_{L^2(2)}\\
&\lesssim &  e^{-s(\f{3-\al-\be}{\al})} e^{-s(\f{4-\al-\be}{\al}-\eps)}+e^{-s\left(\f{3-\be-\al}{\al}\right)}\|\p \tilde{Z}(s)\|_{L^2(2)}. 
\end{eqnarray*}
where we have used (recalling $U\sim |\nabla|^{-\be} Z$), 
$\|\p U(s)\|_{L^\infty}\leq C(\|Z\|_{L^{\f{2}{\be-1}+\eps}}+ \|Z\|_{L^{\f{2}{\be-1}-\eps}})\leq C e^{-s(\f{3-\al-\be}{\al})}$,  \eqref{z},  \eqref{u:12}. Plugging it together yields 
\begin{equation}
\label{dz} 
\|\p \tilde{Z}(\tau)\|_{L^2(2)}  \lesssim e^{-\tau(\f{4-\al-\be}{\al}-\eps)}+ \int_0^\tau  \f{ e^{-(\tau-s)(\f{5-\al-\be}{\al}-\eps)}  e^{-s(\f{3-\al-\be}{\al})}}{\min(1, \tau-s)^{\f{1}{\al}}} 
\|\p \tilde{Z}(s)\|_{L^2(2)}.
\end{equation}
This puts us in a position to use the Gronwal's Lemma \ref{gro}. Note that in order to do that, we 
need any {\it a priori} exponential bound on $\|\p Z(\tau)\|_{L^2(2)}$, similar to Theorem \ref{thm1} for 
$\|Z(\tau)\|_{L^2(2)}$. This is actually easy to achieve, one just has to differentiate the equation and perform very coarse energy estimates\footnote{which will give very inefficient exponential bounds on $\|\p Z(\tau)\|_{L^2(2)}$, but that is all we need to jump start Lemma \ref{gro}}. As a result, Lemma \ref{gro} applies and we obtain 
$$
\|\p \tilde{Z}(\tau)\|_{L^2(2)}  \lesssim e^{-\tau(\f{4-\al-\be}{\al}-\eps)},
$$
as is the statement of \eqref{452}. 

 \end{proof}
 It is now easy to conclude the main result, Theorem \ref{theo:10}. Realizing  that 
  $L^2(2)\hookrightarrow L^p, 1\leq p\leq 2$, one  just needs to translate the $L^p$ estimates for $Z$, 
   in the language of the original variable $z$.

  \section{Global dynamics of the solutions of the Boussinesq  model}
 \label{sec:7} 
 In this section,  we compute the optimal decay rate  in $L^2(2)$ for the solution of the Boussinesq model  \eqref{81}. Recall that the relevant operator $\cl$ has the form 
 $$
 \cl=-|\nabla|^\al+ \f{1}{\al}\xi\cdot \nabla_\xi + 1,
 $$
 with $\la_0(\cl)=1-\f{2}{\al}$ and $\si_{ess}(\cl)\subset \{\la: \Re \la\leq 1-\f{3}{\al} \}$. 
  \begin{theorem}
  	\label{osn} 
  Suppose $\al\in (1, \f{3}{2})$ and $W_0, \Theta_0 \in Y:=L^2(2)(\rtwo)\cap L^\infty(\rtwo)\cap H^1(\rtwo)$. Then  for every $\de>0$, there exists $C=C_\de(\|W_0\|_Y, \|\Theta_0\|_Y) > 0$,  such that for any $\tau > 0$, the solutions $W, \Theta$   for the system of equations \eqref{81} obey 
  	\begin{equation}
  	\label{940}
  	\|W- \ga_2(0)e^{-(\f{3}{\al}-2) \tau} \p_1 G- \ga_1(0) e^{-(\f{2}{\al}-1) \tau} G\|_{L^2(2)}+ \|\Theta- \ga_2(0) e^{-(\f{3}{\al}-2) \tau} G\|_{L^2(2)} \leq 
  	C e^{-2(\f{3}{\al} -2 -  \de) \tau}. 
  	\end{equation}
  	where  $\ga_1(0):= \int W_0 (\xi) d\xi$, and $\ga_2(0):= \int \Theta_0 (\xi) d\xi$.  
  	In particular, if $\widehat{W}_0(0)= \widehat{\Theta}_0(0)= 0$ then
  	\begin{equation}
  	\label{945} 
  	\|W\|_{L^2(2)}+	 \|\Theta\|_{L^2(2)} \leq C_\de   e^{-2(\f{3}{\al}-2- \de) \tau}. 
  	\end{equation} 
  \end{theorem}
 \begin{proof}	
 	Using the spectral decomposition for $\cl$, described in Section \ref{sec:3.6}, write 
 	\begin{eqnarray}
 	\label{f1}
 	W(\tau)&=& \ga_1 (\tau) G(\xi)+ \widetilde{W}(\tau)\\\label{f11}
 	\Theta(\tau)&=& \ga_2 (\tau) G(\xi)+ \widetilde{\Theta}(\tau)
 	\end{eqnarray}
 	 where   $\ga_1 (\tau):= \langle W(\tau), 1\rangle$, $\ga_2 (\tau):= \langle \Theta(\tau), 1\rangle$, $\widetilde{W}= \cq_0  W(\tau, \cdot)$ and $\widetilde{\Theta}= \cq_0 \Theta (\tau, \cdot)$. Then, we derive the equations for $\ga_1, \ga_2$ as before - namely 
 	\begin{eqnarray*}
 		\p_\tau \ga_1 &=&   \langle W_{\tau}, 1\rangle=  \langle \cl W, 1\rangle-  \langle  U \cdot \nabla W, 1\rangle+ \langle  \partial_1 \Theta, 1\rangle\\
 		&=& \langle\cl W, 1\rangle
 		= \langle W, \mathcal{L^*} 1\rangle= (1- \f{2}{\al}) \langle W, 1\rangle= (1- \f{2}{\al}) \ \ga_1(\tau)
 	\end{eqnarray*}
 	Similarly, $\p_\tau \ga_2 = (2- \f{3}{\al}) \ \ga_2(\tau)$. Solving the ODE's results in the formulas 
 	\begin{eqnarray*}
 		\ga_1 (\tau)= \ga_1(0) e^{ (1- \f{2}{\al}) \tau}, 	\ga_2 (\tau)= \ga_2(0) e^{ (2- \f{3}{\al}) \tau}.
 	\end{eqnarray*}
 	For the projections over the essential spectrum, we have the following PDE's 
 	\begin{eqnarray*}
 		\widetilde{W}_{\tau}&=& \mathcal{L} \widetilde{W}- \cq_0 [U \cdot \nabla W- \partial_1 \Theta]= \mathcal{L} \widetilde{W}- \cq_0 [U \cdot \nabla (\ga_1 (0)\ e^{(1- \f{2}{\al}) \tau} G+ \widetilde{W})]+ \\
 		&+& \cq_0[\partial_1 (\ga_2 (0)\ e^{(1- \f{2}{\al}) \tau} G+ \widetilde{\Theta})],\\
 		\widetilde{\Theta}_{\tau}&=& (\cl+1-\f{1}{\al}) \widetilde{\Theta}- \cq_0[U \cdot \nabla \Theta]= (\cl+1-\f{1}{\al}) \widetilde{\Theta}- \cq_0[U \cdot \nabla (\ga_2 (0)\ e^{(2- \f{3}{\al}) \tau} G+ \widetilde{\Theta})].
 	\end{eqnarray*}
 We represent them via the Duhamel's formula 
 	\begin{eqnarray*}
 		\widetilde{W} (\tau)&=& e^{\tau \cl} \widetilde{W_0}-  \int_0^{\tau} e^{(\tau-s) \mathcal{L}} \cq_0 [U 
 		\cdot \nabla (\ga_1 (0)\ e^{(1- \f{2}{\al}) s} G+ \widetilde{W}(s))] \ ds +  \int_0^{\tau} e^{( \tau-s) \mathcal{L}} \cq_0 
 		[\partial_1 \Theta (s)] \ ds,\\
 		\widetilde{\Theta} (\tau)&=& e^{\tau (\cl+1-\f{1}{\al})} \widetilde{\Theta_0}-\int_0^{\tau} e^{(\tau-s) (\cl+1-\f{1}{\al})} \cq_0 
 		[U \cdot \nabla (\ga_2 (0)\ e^{(2- \f{3}{\al})s} G+ \widetilde{\Theta}(s))] \ ds.
 	\end{eqnarray*}
 	One term deserves a special attention, as it is explicit. Note that $\cq_0 \p_1 = \p_1$, since $\cp_0 \p_1=0$. Also for $\ka>0$, since $G$ is an eigenfunction, with eigenvalue $1-\f{2}{\al}$, we have $e^{\ka \mathcal{L}} G=e^{(1-\f{2}{\al})\ka} G$.  By Lemma \ref{lem0}, 
 		\begin{eqnarray*}
 			& &    \int_0^{\tau} e^{( \tau-s) \mathcal{L}} \cq_0 
 			[\partial_1 \Theta (s)] \ ds =   \int_0^{\tau} e^{( \tau-s) \mathcal{L}}  
 			[\partial_1[\ga_2 (0)\ e^{(2- \f{3}{\al})s} G+ \widetilde{\Theta}(s))]] \ ds= \\
 			&=& \ga_2 (0) \int_0^{\tau} e^{(2- \f{3}{\al})s} e^{-\f{\tau-s}{\al}} \p_1 e^{( \tau-s) \mathcal{L}}  
 			[ G] ds +  \int_0^{\tau} e^{-\f{\tau-s}{\al}} \partial_1   e^{( \tau-s) \mathcal{L}}  
 			\widetilde{\Theta}(s)  ds = \\
 			& = & \ga_2 (0) \p_1 G \int_0^{\tau} e^{(2- \f{3}{\al})s} e^{-\f{\tau-s}{\al}} e^{(1-\f{2}{\al})(\tau-s)} ds+  \int_0^{\tau} e^{-\f{\tau-s}{\al}} \partial_1 e^{( \tau-s) \mathcal{L}}  
 			  \widetilde{\Theta}(s) ds =\\
 			  &=&  \ga_2(0)(e^{(2-\f{3}{\al})\tau} - e^{(1-\f{3}{\al})\tau}) \p_1 G+  \int_0^{\tau} e^{-\f{\tau-s}{\al}} \partial_1 e^{( \tau-s) \mathcal{L}}  
 			  \widetilde{\Theta}(s) ds. 
 			 		\end{eqnarray*}
At this point, it makes more sense to introduce the new variable, 
 $$
 W_1(\tau, \xi):=\tilde{W}(\tau, \xi)-  \ga_2(0)(e^{(2-\f{3}{\al})\tau} - e^{(1-\f{3}{\al})\tau}) \p_1 G=: \tilde{W} -  e^{(2-\f{3}{\al})\tau} G_1(\tau, \xi).
 $$			 		
 Note that the decay rate $e^{(2-\f{3}{\al})\tau}$ along the $G_1$ direction  is slower than the decay rate $e^{(1-\f{2}{\al})\tau}$ of the evolution along the $G$ direction.  Also, $G_1$ is basically $\p_1 G$ multiplied by a bounded  function of $\tau$ and hence an element of $L^2(2)\cap L^\infty$. 
 For future reference, 
\begin{equation}
\label{973} 
 \|W_1\|_X - C e^{(2-\f{3}{\al})\tau}\leq \|\tilde{W}\|_X \leq  \|W_1\|_X +  C e^{(2-\f{3}{\al})\tau}. 
\end{equation}
for all Banach spaces in consideration  herein.  

We write the equations for $W_1$ and $\Tilde{\Theta}$ as follows 
 \begin{eqnarray*}
 W_1 (\tau)&=& e^{\tau  \mathcal{L}} \widetilde{W_0}-  \int_0^{\tau} e^{(\tau-s) \mathcal{L}} \cq_0 [U 
 	\cdot \nabla (\ga_1 (0)\ e^{(1- \f{2}{\al}) s} G+  e^{(2- \f{3}{\al}) s} G_1+ W_1(s))] \ ds +  \\
 	&+& \int_0^{\tau} e^{-\f{\tau-s}{\al}} \partial_1 e^{( \tau-s) \mathcal{L}}  
 	\widetilde{\Theta}(s) ds. \\
 	\widetilde{\Theta} (\tau)&=& e^{\tau  (\cl+1-\f{1}{\al})} \widetilde{\Theta_0}-\int_0^{\tau} e^{(\tau-s) (\cl+1-\f{1}{\al})} \cq_0 
 	[U \cdot \nabla (\ga_2 (0)\ e^{(2- \f{3}{\al})s} G+ \widetilde{\Theta}(s))] \ ds.
 \end{eqnarray*}
 	 Note that $U=e^{(1- \f{2}{\al}) s} U_G+ e^{(2- \f{3}{\al}) s} U_{G_1}+U_{W_1}$ and $U_G\cdot G=0$. 
 	 
 	 We start the estimates for $ \widetilde{\Theta}$ 
 	 \begin{eqnarray*}
 	\|\widetilde{\Theta}\|_{L^2(2)} &\leq &  C e^{(2-\f{4}{\al} + \de) \tau} \|\widetilde{\Theta}(0)\|_{L^2(2)} +  
 	|\ga_2(0)|  \int_0^{\tau}   e^{(2- \f{3}{\al}) s} 	 \|  e^{(\tau-s) (\cl+1-\f{1}{\al})} \cq_0 [U  \cdot  \nabla  G] \|_{L^2(2)}  \ ds+ \\
 	 &+&   \int_0^{\tau}   	 \|  e^{(\tau-s) (\cl+1-\f{1}{\al})} \cq_0 [U  \cdot  \nabla  \widetilde{\Theta}(s)]\|_{L^2(2)}  \ ds=:C e^{(2-\f{4}{\al} + \de) \tau}+J_1+J_2 
 	\end{eqnarray*}
For all $\de>0$ small enough, there is $C_\de$, 
 	  \begin{eqnarray*}
 J_1&=& 	  	 \int_0^{\tau}   e^{(2- \f{3}{\al}) s} 	 \|  e^{(\tau-s) (\cl+1-\f{1}{\al})}  \cq_0[U  \cdot  \nabla  G] \|_{L^2(2)}  \ ds \lesssim  
 	  	\|U_{G_1} G\|_{L^2(2)}    \int_0^{\tau}   \f{e^{(2-\f{5}{\al}+\de)(\tau-s)} e^{2(2- \f{3}{\al}) s} }{(a(\tau- s))^{\f{1}{\al}}}ds +  \\
 	  	&+&  \int_0^{\tau}   \f{e^{(2-\f{5}{\al}+\de)(\tau-s)} e^{(2- \f{3}{\al}) s} }{(\min(1, |\tau-s|)^{\f{1}{\al}}} 
 	  	\|U_{W_1}(s, \cdot) \cdot \nabla G\|_{L^2(2)} ds \lesssim  e^{2(2- \f{3}{\al}) \tau}+ \\
 	  	&+& 
 	  	\int_0^{\tau}   \f{e^{(2-\f{5}{\al}+\de)(\tau-s)} e^{(2- \f{3}{\al}) s} }{(\min(1, |\tau-s|)^{\f{1}{\al}}}  (e^{(2-\f{3}{\al})s})^{1-\eps}   ds\leq C_\de  e^{2(2-\f{3}{\al})\tau}. 
 	  \end{eqnarray*}
 	 where we have used    Lemma \ref{le:10}, Gagliardo-Nirenberg's, \eqref{dec:1}, $L^2(2)\hookrightarrow L^1$, \eqref{907}, to estimate 
 	  \begin{eqnarray*}
 	 \|U_{W_1} \nabla G\|_{L^2(2)}    &\leq &  \|U_{W_1}\|_{L^{\f{2}{\epsilon}}} \| (1+ |\xi|^2)|\nabla G|\|_{L^{\f{2}{1- \epsilon}}} \leq   C \|U_{W_1}\|_{L^{\f{2}{\epsilon}}} \leq  
 	 C\|W_1\|_{L^{\f{2}{1+ \epsilon}}} \\
 	 &\leq & C \|W_1\|_{L^2}^{1-\eps}  \|W_1\|_{L^1}^{\eps} \leq C (e^{(2-\f{3}{\al})s})^{1-\eps}.
 \end{eqnarray*}
 Similarly, 
  \begin{eqnarray*}
  J_2=  \int_0^{\tau}   	 \|  e^{(\tau-s) (\cl+1-\f{1}{\al})} \cq_0 [U  \cdot  \nabla  \widetilde{\Theta}(s)]\|_{L^2(2)}  \ ds \leq C 
    \int_0^\tau \f{e^{(2-\f{5}{\al}+\de)(\tau-s)}  }{(\min(1, |\tau-s|)^{\f{1}{\al}}} \|U(s)\|_{L^\infty} \|\widetilde{\Theta}(s)\|_{L^2(2)} ds
  \end{eqnarray*}
 	 Thus, we need a good estimate of $\|U(s)\|_{L^\infty}$.  We have by \eqref{31}   
 	  $$
 	  \|U(s, \cdot))\|_{L^\infty} \leq   C(\|W(s, \cdot)\|_{L^{2+\eps}}+
 	  \|W(s, \cdot)\|_{L^{2-\eps}}). 
 	  $$ 
 	  By the {\it a priori} estimate \eqref{540}, we have a good control  of $\|W(s, \cdot)\|_{L^{2+\eps}}$, namely 
 	  $
 	  \|W(s, \cdot)\|_{L^{2+\eps}}\leq C  e^{(2-\f{3}{\al})s}. 
 	  $
 	  For $\|W(s, \cdot)\|_{L^{2-\eps}}$, we can control it by \eqref{907}, but this is not efficient for our arguments - we need some, however small,   decay in $s$, which we can then input in the Gronwall's, \eqref{gron}. To achieve that, we proceed by Gagliardo-Nirenberg's estimate. Taking into account once again $L^2(2)\hookrightarrow L^1$, and the bounds 
 	  \eqref{dec:1}, 
 	  \begin{eqnarray*}
 	  	& & \|W(s, \cdot)\|_{L^{2-\eps}}\leq \|W(s, \cdot)\|_{L^{2}}^{\f{2-2\eps}{2-\eps}}  
 	  	\|W(s, \cdot)\|_{L^1}^{\f{\eps}{2-\eps}}\leq C ( e^{(2-\f{3}{\al})s})^{\f{2-2\eps}{2-\eps}}.
 	  \end{eqnarray*}
 	 All in all, for all $\de>0$, 
 	 \begin{equation}
 	 \label{u} \|U(s, \cdot)\|_{L^\infty} \leq   C_\de e^{-(\f{3}{\al}-2-\de)s}.
 	 \end{equation}
 	 This results in the following estimates for $J_2$ 
 	 $$
 	 J_2\leq  \int_0^\tau \f{e^{(2-\f{5}{\al}+\de)(\tau-s)} e^{-(\f{3}{\al}-2-\de)s}  }{(\min(1, |\tau-s|)^{\f{1}{\al}}}   \|\widetilde{\Theta}(s)\|_{L^2(2)} ds
 	 $$
 	 Combining all the estimates obtained about\footnote{note that with our restrictions on $\al$, $(\f{3}{\al}-2)<\f{4}{\al}-2$, so this is the slowest rate on the right hand sides of $\|\widetilde{\Theta}(\tau)\|_{L^2(2)}$.} $\|\widetilde{\Theta}(s)\|_{L^2(2)}$, 
 	 , we have 
 	 $$
 	 \|\widetilde{\Theta}(\tau)\|_{L^2(2)}\leq C e^{-2(\f{3}{\al}-2-\de)\tau} + \int_0^\tau \f{e^{(2-\f{5}{\al}+\de)(\tau-s)} e^{-(\f{3}{\al}-2-\de)s}  }{(\min(1, |\tau-s|)^{\f{1}{\al}}}   \|\widetilde{\Theta}(s)\|_{L^2(2)} ds
 	 $$
 	 Applying the Gronwal's, more precisely Lemma \ref{gro}, we conclude 
 	 $$
 	 \|\widetilde{\Theta}(\tau)\|_{L^2(2)}\leq C_\de  e^{-(\f{3}{\al}-2-\de)\tau},
 	 $$
 	 as stated.  
 	 For $W_1$, we get  
 	\begin{eqnarray*}
 		&&\|W_1\|_{L^2(2)} \leq C e^{-(\f{3}{\al}-1- \de) \tau} \|\widetilde{W_0}\|_{L^2(2)} + \\
 		&+& \int_0^{\tau} e^{- \frac{ (\tau- s)}{\al}} 
 		\| \nabla e^{(\tau- s) \mathcal{L}_0} [U  \cdot  (\ga_1 (0)e^{(1- \f{2}{\al}) s}  G+ e^{(2- \f{3}{\al}) s}G_1)+    
 		U \cdot  W_1 ]\|_{L^2(2)} ds \\
 		&+& \int_0^{\tau} e^{- \frac{ (\tau- s)}{\al}} \| \p_1 e^{(\tau-s) \mathcal{L}_0}  \widetilde{\Theta}(s)\|_{L^2(2)} ds 
 	\lesssim e^{(1- \f{3}{\al}+ \de) \tau}+ \\
		&+&   \int_0^{\tau} \f{ e^{(1- \f{3}{\al}+ \de) (\tau- s)} e^{(2- \f{3}{\al})s} \|U (|G|+|G_1|) \|_{L^2(2)} }{(a(\tau- s))^{\f{1}{\al}}} \ ds
			+ \int_0^{\tau}   \f{e^{(1- \f{3}{\al}+ \de) (\tau- s)} \|U\|_{L^\infty} \|W_1\|_{L^2(2)}}{(a(\tau- s))^{\f{1}{\al}}} \ ds\\
 		&+&\int_0^{\tau} \f{  e^{(1- \f{4}{\al}+ \de) (\tau- s)} \|\widetilde{\Theta}(s)    \|_{L^2(2)} }{(a(\tau- s))^{\f{1}{\al}}} \ ds=  e^{(1- \f{3}{\al}+ \de) \tau}  +  I_1+ I_2+ I_3 
 	\end{eqnarray*}
 	For $I_1$, we have 
 $$
 	 	 \|U (|G|+|G_1|) \|_{L^2(2)}  \leq  \|(e^{(1- \f{2}{\al}) s} U_G+ e^{(2- \f{3}{\al}) s} U_{G_1})(|G|+|G_1|) \|_{L^2(2)}  + 
 	 	 \| U_{W_1}(|G|+|G_1|) \|_{L^2(2)}.
$$
The first term is easily estimated, since $G, G_1  \in L^2(2)$ (whence $U_{G}, U_{G_1} \in L^\infty$ by Sobolev embedding and Lemma \ref{le:10})
$$
\|(e^{(1- \f{2}{\al}) s} U_G+ e^{(2- \f{3}{\al}) s} U_{G_1})(|G|+|G_1|) \|_{L^2(2)} \leq C e^{(2- \f{3}{\al}) s},
$$
whence the contribution of these terms is no more than 
$$
C \int_0^\tau \f{ e^{(1- \f{3}{\al}+ \de) (\tau- s)} e^{2(2- \f{3}{\al})s} }{\min(1, |\tau-s|)^{\f{1}{\al}}} ds \leq C e^{2\tau(2-\f{3}{\al})}.
$$
 For $U_{W_1}$ terms, we  can use Lemma \ref{le:10}, the Sobolev inequality and $L^2(2)\hookrightarrow L^{\f{2}{1+\eps}}$ to get
 	\begin{eqnarray*}
 	 	\|U_{W_1}(s) (|G|+|G_1|) \|_{L^2(2)} &=&   \|U_{W_1} \cdot (1+ |\xi|^2) (|G|+|G_1|)  \|_{L^2}  \leq \|U_{W_1}\|_{L^{\f{2}{\epsilon}}} \| (1+ |\xi|^2) (|G|+|G_1|) \|_{L^{\f{2}{1- \epsilon}}} \\
 		&\leq & C \|U_{W_1}\|_{L^{\f{2}{\epsilon}}} \leq C \| \nabla U_{W_1}\|_{L^{\f{2}{1+ \epsilon}}}  \leq 
 		C\|W_1\|_{L^{\f{2}{1+ \epsilon}}}\leq C \|W_1(s)\|_{L^2(2)}. 
 	\end{eqnarray*}
All together, the contribution of $I_1$ is estimated by 
$$
I_1 \leq C e^{-2(\f{3}{\al}-2)\tau}+  \int_0^{\tau} 
\f{ e^{-(\f{3}{\al}-1- \de) (\tau- s)} e^{-(\f{3}{\al}-2)s}}{\min(1, |\tau-s|)^{\f{1}{\al}}}  \|W_1(s)\|_{L^2(2)} ds
$$
	Regarding $I_2$, we first need an appropriate estimate on $\|U\|_{L^\infty}$, which is fortunately already given by \eqref{u}. 
	This then gives the bound for $I_2$, 
	$$
	I_2 \leq   \int_0^{\tau} 
	\f{ e^{-(\f{3}{\al}-1- \de) (\tau- s)}   e^{-(\f{3}{\al}-2-\de)s}  }{\min(1, |\tau-s|)^{\f{1}{\al}}}  \|W_1(s)\|_{L^2(2)} ds
	$$
Combining all estimates for $\|W_1(\tau)\|_{L^2(2)} $ yields 
$$
\|W_1(\tau, \cdot)\|_{L^2(2)} \leq C e^{-2(\f{3}{\al}-2)\tau}+ \int_0^{\tau} 
\f{ e^{-(\f{3}{\al}-1- \de) (\tau- s)}   e^{-(\f{3}{\al}-2-\de)s}  }{\min(1, |\tau-s|)^{\f{1}{\al}}}  \|W_1(s)\|_{L^2(2)} ds.
$$
	Applying Lemma \ref{gro}, with $\mu=2(\f{3}{\al}-2), \si=(\f{3}{\al}-1- \de) , \ka=(\f{3}{\al}-2-\de)$ yields 
	$$
	\|W_1(\tau, \cdot)\|_{L^2(2)} \leq C e^{-2(\f{3}{\al}-2)\tau}. 
	$$
	This is the statement of \eqref{940} and Theorem \ref{osn} is proved in full.

 	\end{proof}

 \appendix
  \section{Sobolev embedding at $L^\infty$: relation \eqref{31}}\label{2.2.2} 
Before we start the proof of \eqref{31}, we recall the following Bernstein inequality. Let $g$ satisfy 
$$
supp\ \hat{g} \subset \{\xi \in R^n: C_1 2^k \leq |\xi| \leq C_2 2^{k+1} \}
$$
for some $k$, and constants $C_1 \leq C_2$. Then for any $\al \geq 0$ and $ 1 \leq p \leq q \leq \infty$,
$$
C_1 2^{\al k} \|g\|_{L^q(\rn)} \leq \| |\nabla|^{\al} g\|_{L^q(\rn)} \leq C_2 2^{\al k+ n k (\f{1}{p}- \f{1}{q})} \|g\|_{L^p(\rn)}.
$$
Now let $\widehat{P_k f}(\xi)= \widehat{\psi}(2^{- k \xi}) \hat{f}(\xi)$, where $\hat{\psi} \in C^{\infty}$, 
$supp\ \hat{\psi}\subseteq  \{\xi \in R^n: \xi \in (\f{1}{2}, 2) \}$.   Then
$$
\|(\nabla^{\perp})^{- \be} f\|_{L^{\infty}} \leq \sum_{k=0}^{\infty} \|P_k ((\nabla^{\perp})^{- \be} f)\|_{L^{\infty}}+ \sum_{k=0}^{\infty} \|P_{-k} ((\nabla^{\perp})^{- \be} f)\|_{L^{\infty}}.
$$
We make use of the above Bernstein inequality several times to control each of these terms. Indeed,
\begin{eqnarray*}
 \sum_{k=0}^{\infty} \|P_k ((\nabla^{\perp})^{- \be} f)\|_{L^{\infty}} &\leq&  \sum_{k=0}^{\infty} 2^{-k \be} \|P_k f\|_{L^{\infty}} \leq \sum_{k=0}^{\infty} 2^{-k \be+ n k (\f{1}{\f{n}{\be}+ \de})} \|P_k f\|_{L^{\f{n}{\be}+ \de}}\\
 &\leq& \|f\|_{L^{\f{n}{\be}+ \de}} \sum_{k=0}^{\infty} 2^{-k \be(1- \f{n}{n+ \be \ga} ) } \leq C \|f\|_{L^{\f{n}{\be}+ \de}}.
\end{eqnarray*}
In the same way,
\begin{eqnarray*}
	\sum_{k=0}^{\infty} \|P_{-k} ((\nabla^{\perp})^{- \be} f)\|_{L^{\infty}} &\leq&  \sum_{k=0}^{\infty} 2^{k \be} \|P_{-k} f\|_{L^{\infty}} \leq \sum_{k=0}^{\infty} 2^{k \be- n k (\f{1}{\f{n}{\be}- \de})} \|P_k f\|_{L^{\f{n}{\be}- \de}}\\
	&\leq& \|f\|_{L^{\f{n}{\be}- \de}} \sum_{k=0}^{\infty} 2^{k \be(1- \f{n}{n- \be \ga} ) } \leq C \|f\|_{L^{\f{n}{\be}- \de}}.
\end{eqnarray*}
 \section{Generalized Gronwall's estimate: Lemma \ref{gro}} 
 The proof of Lemma \ref{gro} is straightforward, by a bootstrapping argument. We show that every Lyapunov exponent less than $-\mu$ can be bootstrapped lower. First,  relabeling    $I(\tau)\to (1+|A_1|+|A_2|+|A_3)^{-1} I(\tau)$, we may assume without loss of generality that $A_1=A_2=A_3=1$. Next, assume that $\ga<\mu$ is a Lyapunov  exponent, that is $I(\tau)\leq C e^{-\ga \tau}$. We know by the {\it a priori} assumed boundedness of $I(\tau)$ there is  such an exponent. 
 Applying this in \eqref{gron}, we obtain an improved estimate for $I(\tau)$. Indeed, 
 \begin{eqnarray*}
 I(\tau)\leq e^{-\mu \tau} + C e^{-\si \tau} \int_0^\tau \f{e^{s(\si-\ka-\ga)}}{(\min(1, |\tau-s|)^a} ds
 \end{eqnarray*}
 If $\si-\ka-\ga\neq 0$, we have for $\tau>1$, 
  \begin{eqnarray*}
 \int_0^\tau \f{e^{s(\si-\ka-\ga)}}{|(\min(1, |\tau-s|)^a} ds &\leq &  \int_0^{\tau-1}  e^{s(\si-\ka-\ga)}  ds+
 e^{\tau(\si-\ka-\ga)} e^{|\si-\ka-\ga|} \int_{\tau-1}^\tau \f{1}{|\tau-s|^a} ds \\
 &\leq & \f{e^{(\tau-1)(\si-\ka-\ga)}-1}{\si-\ka-\ga}+ C_{a, \si, \ka, \ga}  e^{\tau(\si-\ka-\ga)}. 
 \end{eqnarray*}
 whence the bound 
 $$
  I(\tau)\leq e^{-\mu \tau} + C_{a, \si, \ka, \ga}e^{-\tau(\ka+\ga)}.
 $$
 It follows that $\min(\mu, \ga+\ka)>\ga$ is a new, better Lyapunov exponent than $\ga$. 
 
 In general, we can keep $\si-\ka-\ga$ away from zero (and so the previous argument valid in all cases), if we readjust the $\ga$ if necessary. 
 
 In practice, starting with $\ga=0$, we jump immediately to $\ka$ by the previous argument, since $\si-\ka>0$, by assumption. Since $\ka<\mu$, we can apply the same argument again with $\ga=\ka$. At this point, either $2\ka>\mu$ and we finish off (by readjusting slightly $\ga$ by taking it smaller, like $\ga=\f{2\ka}{3}$, if it happens that, say  $|\si-2\ka|\leq \f{\ka}{2}$). If not, that is if $2\ka<\mu$, take  $\ga=2\ka$ to be our new Lyapunov exponent and repeat. Eventually, for some $n_0$, $n_0 \ka<\mu\leq (n_0+1) \ka$ and we will reach a Lyapunov exponent $\mu$. 
 
 \section{Commutator estimates with weights} 
 \label{app:10} 
 In this section, we prove \eqref{310} and \eqref{903}.  
 \subsection{Proof of \eqref{310}} 
 Recall, that for $s\in (0,2)$
 \begin{eqnarray*}
 	[|\nabla|^{s}, g] f (x)&=& |\nabla|^{s} (g f)- g \ |\nabla|^{s}f= c_s \int \f{f(x) g(x)- f(y) g(y)}{|x- y|^{2+ s}} dy- g(x) c_s\int \f{f(x)- f(y)}{|x- y|^{2+ s}} dy \\
 	&=& c_s \int \f{ f(y)( g(x)-  g(y) )}{|x- y|^{2+ s}} dy.
 \end{eqnarray*}
Introduce a smooth partition of unity, that is a function $\psi\in C^\infty_0(\rone)$, $supp\  \psi\subset (\f{1}{2},2)$, so that 
$$
\sum_{k=-\infty}^\infty \psi(2^{-k} |\xi|) = 1, \xi\in\rtwo, \xi\neq 0.
$$
Introduce another $C^\infty_0$ function $\Psi(z)=z^2 \psi(z)$, so that we can decompose 
$$
|\xi|^{2} = \sum_{k=-\infty}^\infty |\xi|^{2} \psi(2^{-k} |\xi|) = \sum_{k=-\infty}^\infty 2^{2k} \Psi(2^{-k} |\xi|). 
$$
We can then write 
\begin{eqnarray*}
F(\xi) &:=&  [|\nabla|^{\f{\al}{2}},|\xi|^2] f   \sum_k 2^{2 k}  [|\nabla|^{\f{\al}{2}}, \Psi (2^{- k} \cdot)]  f(\xi)  = 
 \sum_k 2^{2k}   \int  \f{f(y) (\Psi(2^{-k}\xi)- \Psi(2^{-k}y))}{|\xi- y|^{2+ \f{\al}{2}}} dy. 
\end{eqnarray*}
Introducing  
$$
F_k:= \int  \f{|f(y)| |\Psi(2^{-k}\xi)- \Psi(2^{-k}y)|}{|\xi- y|^{2+ \f{\al}{2}}} dy, 
$$
we need to control 
\begin{eqnarray*}
\|F\|_{L^2}^2 &=&  \sum_l \int_{|\xi|\sim 2^l} |F(\xi)|^2 d\xi=\sum_l \int_{|\xi|\sim 2^l} \left|\sum_k 2^{2k} F_k(\xi)\right|^2 d\xi=\\
&=& \sum_l \int_{|\xi|\sim 2^l} \left|\sum_{k>l+10}  2^{2k} F_k(\xi)\right|^2 d\xi+\sum_l \int_{|\xi|\sim 2^l} \left|\sum_{k=l-10}^{l+10}  2^{2k} F_k(\xi)\right|^2 d\xi+\\
&+& \sum_l \int_{|\xi|\sim 2^l} \left|\sum_{k<l-10}  2^{2k} F_k(\xi)\right|^2 d\xi=:K_1+K_2+K_3
\end{eqnarray*}

We first consider the cases $k>l+10$.  One can estimate easily $F_k$ point-wise.  More specifically, since in the denominator of the expression for $F_k$, we have $|\xi-y|\geq \f{1}{2} |\xi|\geq 2^{k-3}$, 
$$
|F_k(\xi)|\leq 2^{-k(2+\f{\al}{2})} \int  |f(y)||\Psi(2^{-k}y)| dy\leq C 2^{-k(1+\f{\al}{2})} \|f\|_{L^2(|y| \sim 2^k)},
$$
whence 
\begin{eqnarray*}
 & & K_1 \leq 
  \sum_l 2^{2l} \sum_{k_1>l+10} \sum_{k_2>l+10} 
  2^{k_1(1-\f{\al}{2})} \|f\|_{L^2(|y| \sim 2^{k_1})} 2^{k_2(1-\f{\al}{2})} \|f\|_{L^2(|y| \sim 2^{k_2})}\\
  &\leq &   \sum_{k_1} \sum_{k_2} 2^{2\min(k_1,k_2)}2^{k_1(1-\f{\al}{2})}  \|f\|_{L^2(|y| \sim 2^{k_1})} 2^{k_2(1-\f{\al}{2})} \|f\|_{L^2(|y| \sim 2^{k_2})} \\
  &\leq & C \sum_k 2^{k(4-\al)}  \|f\|_{L^2(|y| \sim 2^{k})}^2 \leq C  \| |\xi|^{2-\f{\al}{2}} f\|^2.
\end{eqnarray*}
where we have used $\sum_{l: l<\min(k_1,k_2)-10}  2^{2l} \leq C 2^{2\min(k_1,k_2)}$. 

For the case $k<l-10$, we perform similar argument, since 
$$
|F_k(\xi)|\leq C 2^{-l(2+\f{\al}{2})} 2^k \|f\|_{L^2(|y| \sim 2^k)}. 
$$
So, 
\begin{eqnarray*}
	& & K_3 \leq C \sum_l 2^{2l} 2^{-l(4+\al)} \sum_{k_1<l-10} \sum_{k_2<l-10} 2^{3k_1} \|f\|_{L^2(|y| \sim 2^{k_1})} 2^{3k_2} \|f\|_{L^2(|y| \sim 2^{k_2})} \\
	&\leq &C  \sum_{k_1} \sum_{k_2} 2^{3k_1} \|f\|_{L^2(|y| \sim 2^{k_1})} 2^{3k_2} \|f\|_{L^2(|y| \sim 2^{k_2})} 2^{-(2+\al)\max(k_1,k_2)} \\
	&\leq & C \sum_k 2^{k(4-\al)}  \|f\|_{L^2(|y| \sim 2^{k})}^2 \leq C  \| |\xi|^{2-\f{\al}{2}} f\|^2.
\end{eqnarray*}
Finally, for the case $|l-k|\leq 10$, we use 
$$
|\Psi (2^{- k} \xi)- \Psi (2^{- k} y)| \leq 2^{- k} |\xi- y| |\nabla \Psi (2^{- k} (\xi- y))| \leq C 2^{- k} |\xi- y|,
$$
so that 
$$
|F_k(\xi)|\leq C 2^{-k} \int_{|y| \sim 2^k} \f{|f(y)|}{|\xi- y|^{1+ \f{\al}{2}}} dy = C 2^{-k} |f| \chi_{|y| \sim 2^k} * \f{1}{|\cdot|^{1+ \f{\al}{2}}}.
$$
Thus, by H\"older's 
\begin{eqnarray*}
& & K_2  \leq C \sum_k \int_{|\xi|\sim 2^k} 2^{2k} \left||f| \chi_{|y| \sim 2^k} * \f{1}{|\cdot|^{1+ \f{\al}{2}}}\right|^2 d\xi\leq 
C \sum_k 2^{2k} \||f| \chi_{|y| \sim 2^k} * \f{1}{|\cdot|^{1+ \f{\al}{2}}}\|_{L^2(|\xi|\sim 2^k)}^2 \\
&\leq & C \sum_k 2^{k(4-\al)} \||f| \chi_{|y| \sim 2^k} * \f{1}{|\cdot|^{1+ \f{\al}{2}}}\|_{L^{\f{4}{\al}}(|\xi|\sim 2^k)}^2 \leq 
C \sum_k 2^{k(4-\al)}   \|f\|_{L^2(|\xi| \sim 2^k)}^2 \leq C  \| |\xi|^{2-\f{\al}{2}} f\|^2.
\end{eqnarray*}
where we have used the Hausdorf-Young's inequality 
$$
\|f \chi_{|y| \sim 2^k} * \f{1}{|\cdot|^{1+ \f{\al}{2}}}\|_{L^{\f{4}{\al}}}   \leq C \|\f{1}{|\cdot|^{1+ \f{\al}{2}}} \|_{L^{\f{4}{2- \al}, \infty}} \ \|f\|_{L^2(|\xi| \sim 2^k)}\leq C \|f\|_{L^2(|\xi| \sim 2^k)}.
$$
 \subsection{ Proof of \eqref{903}} 
For the proof of \eqref{903}, recall the representation formula \eqref{921}. We will reduce to the same expressions as above. With the partition of unity displayed above, write 
\begin{eqnarray*}
[\p_1 |\nabla|^{-a}, |\xi|^2] f (\xi) &=&  c_a \sum_{k=-\infty}^\infty  2^{2k} [\p_1 |\nabla|^{-a}, \psi(2^{-k}\cdot)] f = \\
&=& 
c_a \sum_{k=-\infty}^\infty  2^{2k} [\p_{\xi_1} \int_{\rtwo} \f{\psi(2^{-k} y) f(y)}{|\xi-y|^{2-a}} dy - 
\psi(2^{-k} \xi)\p_{\xi_1} \int_{\rtwo} \f{  f(y)}{|\xi-y|^{2-a}} dy]=\\
&=& c_a(a-2)  \sum_{k=-\infty}^\infty  2^{2k} \int_{\rtwo} \f{\xi_1-y_1}{|\xi-y|}\f{(\psi(2^{-k} y)-\psi(2^{-k}\xi))  f(y)}{|\xi-y|^{2-a}} dy
\end{eqnarray*}
Taking absolute values and estimating yields the bound 
$$
|[\p_1 |\nabla|^{-a}, |\xi|^2] f (\xi)|\leq C_a \sum_{k=-\infty}^\infty  2^{2k} \int_{\rtwo}  \f{|\psi(2^{-k} y)-\psi(2^{-k}\xi)| |f(y)|}{|\xi-y|^{3-a}} dy
$$
This is of course exactly the same expression as before for the $F_k$, with $a:=1-\f{\al}{2}$. Therefore, we can apply the same estimates to obtain 
$$
\|[\p_1 |\nabla|^{-a}, |\xi|^2] f\|_{L^2(\rtwo)}\leq C \||\xi|^{1+a} f\|_{L^2}.
$$
This establishes \eqref{903}. 

\section{Semigroup estimates:  Proposition \eqref{lem1}}
	\begin{proof}
	We have 
	\begin{eqnarray*}
		&&\|\partial^{\ga} (e^{\tau \cl}  f)\|^2_{L^2(2)} \leq \int_{\rtwo} |\partial^{\ga} (e^{\tau \mathcal{L}} f)|^2 d \xi+  \int_{\rtwo} | |\xi|^2 \partial^{\ga} (e^{\tau \mathcal{L}} f)|^2 d \xi\\
		&=&	e^{2(1-\f{3-\be}{\al})\tau}   \bigg( \int_{\rtwo} |p^{\ga} [e^{- a(\tau) |p|^{\al}}  \widehat{f}(p  e^{-\f{\tau}{\al}})]|^2 dp+  \int_{\rtwo} | \Delta_p[p^{\ga} e^{- a(\tau) |p|^{\al}}  \widehat{f}(p  e^{-\f{\tau}{\al}})]|^2 dp \bigg)\\
		&=&e^{2(1-\f{3-\be}{\al})\tau}   \bigg( \int_{\rtwo} |p^{\ga} [e^{- a(\tau) |p|^{\al}}  \widehat{f}(p  e^{-\f{\tau}{\al}})]|^2 dp+ \ga^2 \int_{\rtwo} |p^{ |\ga|- 1} \nabla_p[e^{- a(\tau) |p|^{\al}}  \widehat{f}(p  e^{-\f{\tau}{\al}})]|^2 dp \\
		&+&  \int_{\rtwo} |p^{\ga} \Delta_p[e^{- a(\tau) |p|^{\al}}  \widehat{f}(p  e^{-\f{\tau}{\al}})]|^2 dp \bigg). \\
	\end{eqnarray*}
	At this point it is clear that it is better,   to divide both sides by $e^{2(1-\f{3-\be}{\al})\tau}$. Then,  we want to control the right hand side of the following 
	\begin{eqnarray*}
		\f{\|\partial^{\ga} (e^{\tau \cl}  f)\|^2_{L^2(2)}}{e^{2(1-\f{3-\be}{\al})\tau}} &\leq&
		\int_{\rtwo} |p^{\ga} [e^{- a(\tau) |p|^{\al}}  \widehat{f}(p  e^{-\f{\tau}{\al}})]|^2 dp
		+ \ga^2 \int_{\rtwo} | |p|^{ |\ga|- 1} \nabla_p[e^{- a(\tau) |p|^{\al}}  \widehat{f}(p  e^{-\f{\tau}{\al}})]|^2 dp \\   
		&+&  \int_{\rtwo} |p^{\ga} \Delta_p[e^{- a(\tau) |p|^{\al}}  \widehat{f}(p  e^{-\f{\tau}{\al}})]|^2 dp:= J_1+ J_2+ J_3. \nonumber 
	\end{eqnarray*}
	\subsection{\textbf{Estimate for $J_1$}}   To control $J_1$ we divide the argument into two different cases, $\tau \leq 1$ and $\tau > 1$. In the case of $\tau \leq 1$, we have 
	\begin{eqnarray*}
		J_1&=& \int_{\rtwo} |p^{\ga} [e^{- a(\tau) |p|^{\al}}  \widehat{f}(p e^{-\f{\tau}{\al}})]|^2 dp = 
	 \int_{\rtwo} |q|^{2 |\ga|} e^{-2  a(\tau) |q \cdot e^{\f{\tau}{\al}}|^{\al} }  |\widehat{f}(q)|^2 dq\\
		&\leq&    \int\limits_{\{q: 0 \leq 2  a(\tau) |q \cdot e^{\f{\tau}{\al}}|^\al  \leq 1\}}\ \ |q|^{2 |\ga|} e^{-2  a(\tau) |q \cdot e^{\f{\tau}{\al}}|^{\al} }  |\widehat{f}(q)|^2 dq \\
		&+&   \sum_{j=1}^{\infty} \int\limits_{\{q: \ j  \leq 2  a(\tau) |q \cdot e^{\f{\tau}{\al}}|^\al \leq j+ 1\}} |q|^{2 |\ga|} e^{-2  a(\tau) |q \cdot e^{\f{\tau}{\al}}|^{\al} }  |\widehat{f}(q)|^2 dq \\ 
		&=&  J^1_1+ J^2_1.
	\end{eqnarray*}	 
	We can estimate, since $\tau\leq 1$, 
	\begin{eqnarray*}
		J^1_1  &\leq &     \int_{0 \leq |q| \leq \f{e^{-\f{ \tau}{\al}}}{(2 a(\tau))^{\f{1}{\al}}}} |q|^{2|\ga|} |\hat{f}(q)|^2 dq
		\leq   \f{1}{(a(\tau))^{\f{2 |\ga|}{\al}}} \|f\|^2_{L^2} \leq   C \ \f{e^{ \f{- 2 \tau}{\al}(1- \eps)}}{(a(\tau))^{\f{2 |\ga|}{\al}}} \|f\|^2_{L^2(2)}.
	\end{eqnarray*} 
	We treat $J_1^2$ in a similar manner. Indeed, again for $\tau\leq 1$, 
	\begin{eqnarray*}
		J^2_1 &\leq&    \sum_{j=1}^{\infty} e^{- j} \int\limits_{ j \leq 2 a(\tau)|q e^{\f{\tau}{\al}}|^{\al} \leq  (j+1)}  |q|^{2|\ga|} |\hat{f}(q)|^2dq \\
		&\leq&   \f{C}{(a(\tau))^{\f{2 |\ga|}{\al}}} \sum_{j=1}^{\infty}
		e^{- (j+1)} (j+1)^{\f{2 |\ga|}{\al}}\int\limits_{ j \leq 2 a(\tau)|q e^{\f{\tau}{\al}}|^{\al} \leq  (j+1)}  |\hat{f}|^2 dq
		\leq   \\
		&\leq & \f{C}{(a(\tau))^{\f{2 |\ga|}{\al}}} \|f\|^2_{L^2}\sum_{j=1}^{\infty} 	e^{- (j+1)} (j+1)^{\f{2 |\ga|}{\al}} \leq        C \f{e^{-\f{2 \tau}{\al}(1- \eps)}}{(a(\tau))^{\f{2 |\ga|}{\al}}} \|f\|^2_{L^2(2)}
	\end{eqnarray*}	
	After putting together the estimates for $J_1^1$ and $J_1^2$, we get 	$
	J_1 \leq C \f{e^{-\f{2 \tau}{\al} (1- \epsilon) } \|f\|^2_{L^2(2)}}{a(\tau)^{\f{2 |\ga|}{\al}}}	$. 
	
Regarding the case $\tau > 1$, first note that in this range of $\tau$, $a(\tau)\geq \f{1}{2}$. Moreover, 
\begin{eqnarray*}
|\hat{f}(q)- \hat{f}(0)| \leq  2 \|\hat{f}\|_{L^{\infty}}, \ \ 
	|\hat{f}(q)- \hat{f}(0)| \leq  |q| \|\nabla \hat{f}\|_{L^{\infty}},
\end{eqnarray*}
whence  by interpolation, we conclude that for every $\eps>0$, we have 
\begin{equation}
\label{80} 
|\hat{f}(q)- \hat{f}(0)| \leq C_\eps|q|^{1- \epsilon} \| |\nabla|^{1- \epsilon} \hat{f}\|_{L^{\infty}} \leq C_{\epsilon} |q|^{1- \epsilon} \|f\|_{L^2(2)},
\end{equation}
where in the last inequality we have used that by Hausdorf-Young's \\ 
$ 	\| |\nabla|^{1- \epsilon} \hat{f}\|_{L^{\infty}}\leq \int_{\rtwo} |\xi|^{1-\eps} |f(\xi)| d\xi\lesssim  \|f\|_{L^2(2)}$. Therefore,
	\begin{eqnarray*}
	J_1 &=&  \int_{\rtwo} |p^{\ga} [e^{- a(\tau) |p|^{\al}}  \widehat{f}(p e^{-\f{\tau}{\al}})]|^2 dp
	=  e^{ \f{2 \tau}{\al} (|\ga|+ 1)} \int_{\rtwo} e^{-2  a(\tau) |q \cdot e^{\f{\tau}{\al}}|^{\al} } |q|^{2 |\ga|} |\widehat{f}(q)|^2 dq\\
&\leq& e^{ \f{2 \tau}{\al} (|\ga|+ 1)} \|f\|^2_{L^2(2)} \int_{\rtwo} e^{-2  a(\tau) |q \cdot e^{\f{\tau}{\al}}|^{\al} } |q|^{2 (|\ga|+1- \eps)} dq\\	
	&\leq& e^{ \f{2 \tau}{\al} (|\ga|+ 1)} \|f\|^2_{L^2(2)} \int\limits_{\{q: 2  a(\tau) |q \cdot e^{\f{\tau}{\al}}|^\al  \leq 1\}}\ \ e^{-2  a(\tau) |q \cdot e^{\f{\tau}{\al}}|^{\al} } |q|^{2 (|\ga|+1- \eps)} dq \\
	&+&  e^{ \f{2 \tau}{\al} (|\ga|+ 1)}\|f\|^2_{L^2(2)} \sum_{j=1}^{\infty} \int\limits_{\{q: \ j  \leq 2  a(\tau) |q \cdot e^{\f{\tau}{\al}}|^\al \leq j+ 1\}} e^{-2  a(\tau) |q \cdot e^{\f{\tau}{\al}}|^{\al} } |q|^{2 (|\ga|+1 - \eps)}  dq =  J^1_1+ J^2_1.
\end{eqnarray*}		
	Now 
	\begin{eqnarray*}
J_1^1&=& e^{ \f{2 \tau}{\al} (|\ga|+ 1)} \|f\|^2_{L^2(2)} \int\limits_{\{q: 2  a(\tau) |q \cdot e^{\f{\tau}{\al}}|^\al  \leq 1\}}\ \ e^{-2  a(\tau) |q \cdot e^{\f{\tau}{\al}}|^{\al} } |q|^{2 (|\ga|+1- \eps)} dq \\
&\leq& e^{ \f{2 \tau}{\al} (|\ga|+ 1)} \|f\|^2_{L^2(2)} \int\limits_{\{q: 2  a(\tau) |q \cdot e^{\f{\tau}{\al}}|^\al  \leq 1\}}\ \  |q|^{2 (|\ga|+ 1- \eps)} dq \leq \\
&\leq & C e^{ \f{2 \tau}{\al} (|\ga|+ 1)} \|f\|^2_{L^2(2)} \int\limits_0^{\f{e^{-\f{\tau}{\al}}}{(2a(\tau))^{\f{1}{\al}}}}\ \  r^{2 (|\ga|+ 1- \eps)+ 1} dq \leq  C \f{e^{-\f{2 \tau}{\al} (1- \epsilon) } \|f\|^2_{L^2(2)}}{a(\tau)^{\f{2}{\al}(|\ga| + 2- \epsilon)}} \leq C \f{e^{-\f{2 \tau}{\al} (1- \epsilon) } \|f\|^2_{L^2(2)}}{a(\tau)^{\f{2 |\ga|}{\al}}}.
	\end{eqnarray*}
In a similar way,
	\begin{eqnarray*}
	J_1^2&=& e^{ \f{2 \tau}{\al} (|\ga|+ 1)}\|f\|^2_{L^2(2)} \sum_{j=1}^{\infty} \int\limits_{\{q: \ j  \leq 2  a(\tau) |q \cdot e^{\f{\tau}{\al}}|^\al \leq j+ 1\}} e^{-2  a(\tau) |q \cdot e^{\f{\tau}{\al}}|^{\al} } |q|^{2 (|\ga|+ 1 - \eps)}  dq  \\
	&\leq& e^{ \f{2 \tau}{\al} (|\ga|+ 1)} \|f\|^2_{L^2(2)} \sum_{j=1}^{\infty} e^{-j} \int\limits_{\{q: \ j  \leq 2  a(\tau) |q \cdot e^{\f{\tau}{\al}}|^\al \leq j+ 1\}} |q|^{2 (|\ga|+ 1 - \eps)}  dq\\ &\leq& e^{ \f{2 \tau}{\al} (|\ga|+ 1)} \|f\|^2_{L^2(2)} \sum_{j=1}^{\infty} e^{-j} \int^{{(\f{j+1}{a(\tau)})^{\f{1}{\al}} e^{-\f{\tau}{\al}}}}_{(\f{j}{a(\tau)})^{\f{1}{\al}} e^{-\f{\tau}{\al}}} r^{2 (|\ga|+ 1 - \eps)+ 1}  dr \leq \\
	&\leq & C \f{e^{-\f{2 \tau}{\al} (1- \epsilon) }}{a(\tau)^{\f{2}{\al}(|\ga| + 2- \epsilon)}} \|f\|^2_{L^2(2)} \sum_{j=1}^{\infty} e^{-j} (j+1)^{2 (|\ga|+ 2 - \eps)} \leq   C \f{e^{-\f{2 \tau}{\al} (1- \epsilon) } \|f\|^2_{L^2(2)}}{a(\tau)^{\f{2 |\ga|}{\al}}}.
	\end{eqnarray*}	
Therefore for $\tau > 1$ we have
$
J_1 \leq C \f{e^{-\f{2 \tau}{\al} (1- \epsilon) } \|f\|^2_{L^2(2)}}{a(\tau)^{\f{2 |\ga|}{\al}}}.
$
	\subsection{\textbf{Estimate for $J_2$}}  To control $J_2$ first note that
	\begin{eqnarray}\label{nabla1}
	\nabla e^{- a(\tau) |p|^{\al}}= - \al \ a(\tau) \ p |p|^{\al- 2} e^{- a(\tau) |p|^{\al}}.
	\end{eqnarray}
	Therefore,
	\begin{eqnarray*}
		J_2& \leq &  \al^2 |\ga|^2 a(\tau)^2 \int_{\rtwo} | |p|^{ |\ga|- 1} |p|^{\al- 1} \ e^{- a(\tau) |p|^{\al}}  \ \widehat{f}(p  e^{\f{-\tau}{\al}})|^2 dp\\
		&+& |\ga|^2 e^{\f{-2 \tau}{\al}} \int_{\rtwo} |p^{|\ga|} \  e^{- a(\tau) |p|^{\al}} \cdot (\nabla \widehat{f})(p  e^{\f{-\tau}{\al}} )|^2 dp:= I_1+ I_2.
	\end{eqnarray*}
	\subsubsection{Estimate for $I_1$} 
	To control the first term $I_1$ we proceed as follows
	\begin{eqnarray*}
		\f{I_1}{a(\tau)^2} &\leq&   \int_{\rtwo} e^{- 2 a(\tau) |p|^{\al}} |p|^{2(\al+ |\ga|- 2)} 
		|\widehat{f}(p \cdot e^{-\f{ \tau}{\al}})|^2 dp
		= \\
		&=& e^{ \f{2 \tau}{\al} (\al+ |\ga|- 1)} \int_{\rtwo} e^{-2  a(\tau) |q \cdot e^{\f{\tau}{\al}}|^{\al} } |q|^{2 (\al+ |\ga|- 2)} |\widehat{f}(q)|^2 dq\\
		&\leq& e^{ \f{2 \tau}{\al} (\al+ |\ga|- 1)}  \int\limits_{\{q: 2  a(\tau) |q \cdot e^{\f{\tau}{\al}}|^\al  \leq 1\}}\ \ e^{-2  a(\tau) |q \cdot e^{\f{\tau}{\al}}|^{\al} } |q|^{2 (\al+ |\ga|- 2)} |\widehat{f}(q)|^2 dq \\
		&+&  e^{ \f{2 \tau}{\al} (\al+ |\ga|- 1)}\sum_{j=1}^{\infty} \int\limits_{\{q: \ j  \leq 2  a(\tau) |q \cdot e^{\f{\tau}{\al}}|^\al \leq j+ 1\}} e^{-2  a(\tau) |q \cdot e^{\f{\tau}{\al}}|^{\al} } |q|^{2 (\al+ |\ga|- 2)} |\widehat{f}(q)|^2 dq \\ 
		&=&  I^1_1+ I^2_1.
	\end{eqnarray*}	 
	We can estimate 
	\begin{eqnarray*}
		I^1_1  &\leq &   e^{ \f{2 \tau}{\al} (\al+ |\ga|- 1)} \int_{|q| \leq \f{e^{-\f{ \tau}{\al}}}{(2 a(\tau))^{\f{1}{\al}}}} |q|^{2(\al+ |\ga|- 2)} |\hat{f}(q)|^2 dq
		=  \\
		&=&  e^{ \f{2 \tau}{\al} (\al+ |\ga|- 1)} \int_{|q| \leq \f{e^{-\f{ \tau}{\al}}}{(2 a(\tau))^{\f{1}{\al}}}}|q|^{2(\al+ |\ga|- 2)} |\hat{f}(q) - \hat{f}(0)|^2  dq\\
	\end{eqnarray*}
Using the relation \eqref{80}, we obtain 
	\begin{eqnarray*}
		I_1^1 &\leq & e^{ \f{2 \tau}{\al} (\al+ |\ga|- 1)}  \|f\|^2_{L^2(2)} \int_{|q| \leq \f{e^{-\f{ \tau}{\al}}}{(2 a(\tau))^{\f{1}{\al}}}} |q|^{2(\al+\ga-2)} |q|^{2(1- \epsilon)} dq= \\
		&=& C  e^{ \f{2 \tau}{\al} (\al+ |\ga|- 1)} \|f\|^2_{L^2(2)} \int_0^{\f{e^{-\f{ \tau}{\al}}}{(2 a(\tau))^{\f{1}{\al}}}} r^{2(\al+ |\ga|- \epsilon)-1} \ dr=    C\f{ e^{- \f{2 \tau}{\al}(1- \epsilon)} \|f\|^2_{L^2(2)} }{a(\tau)^{2 (1+ \f{|\ga|-\eps}{\al})}} .
	\end{eqnarray*}
	therefore, recalling that $a(\tau)\leq 1$, 
	$
	I^1_1 \leq \f{e^{  -\f{2 \tau}{\al} (1 - \epsilon)} \|f\|^2_{L^2(2)}}{a(\tau)^{2 (1+ \f{|\ga|}{\al})}}.
	$
	
	We treat $I_1^2$ in a similar manner. Again, using \eqref{80}, 
	\begin{eqnarray*}
		I^2_1 &\leq&  e^{\f{2 \tau}{\al} (\al+ |\ga|- 1)} \sum_{j=1}^{\infty} e^{- j} \int\limits_{ j \leq 2 a(\tau)|q e^{\f{\tau}{\al}}|^{\al} \leq  (j+1)}  |q|^{2(\al+ |\ga|- 2)} |\hat{f}(q)- \hat{f}(0)|^2dq \\
		&\leq&   e^{\f{2 \tau}{\al} (\al+ |\ga|- 1)} \sum_{j=1}^{\infty}
		e^{- j} \int\limits_{ j \leq 2 a(\tau)|q e^{\f{\tau}{\al}}|^{\al} \leq  (j+1)} |q|^{2(\al+ |\ga|- 2)} |q|^{2(1- \epsilon)} \|f\|^2_{L^2(2)} dq\\
		&\leq&  e^{\f{2 \tau}{\al} (\al+ |\ga|- 1)} \|f\|^2_{L^2(2)} \sum_{j=1}^{\infty} e^{- j }\int\limits_{ j \leq 2 a(\tau)|q e^{\f{\tau}{\al}}|^{\al} \leq  (j+1)} |q|^{2(\al+ |\ga| - 1- \epsilon)}  dq\\
		&\leq &  C \f{e^{\f{2 \tau}{\al} (\al+ |\ga|- 1)} \|f\|^2_{L^2(2)}}{a(\tau)^{2(1+ \f{|\ga|}{\al})}} \sum_{j=1}^{\infty} e^{- j} 
		j^{\f{2(\al+ |\ga|- \epsilon)}{\al}} e^{- \f{2 \tau}{\al}(\al+ |\ga|- \epsilon)} \leq C \f{e^{-\f{2 \tau}{\al}(1- \epsilon) } \|f\|^2_{L^2(2)} }{a(\tau)^{2(1+ \f{|\ga|}{\al})}}.
	\end{eqnarray*}	
	After putting together the estimates for $I_1^1$ and $I_1^2$ we get
	$
	I_1 \leq C \f{e^{-\f{2 \tau}{\al} (1- \epsilon) } \|f\|^2_{L^2(2)}}{a(\tau)^{\f{2 |\ga|}{\al}}}.
	$
	
	\subsubsection{Estimate for $I_2$} 
	\begin{eqnarray*}
		I_2 &\leq&  C e^{-\f{2 \tau}{\al}} \int_{\rtwo} | \ |p|^{  |\ga|- 1}e^{- a(\tau) |p|^{\al}}  (\nabla \widehat{f})(p  e^{\f{-\tau}{\al}})|^2 dp= \\
		&=& e^{ \f{2 \tau}{\al} (  |\ga|- 1)} \int_{\rtwo} |q|^{2 ( |\ga|- 1)} e^{-2  a(\tau) |q \cdot e^{\f{\tau}{\al}}|^{\al} }  |\nabla \widehat{f}(q)|^2 dq\\
		&\leq&  e^{ \f{2 \tau}{\al} ( |\ga|- 1)}   \sum_{j=0}^{\infty} 
		\int\limits_{\{q: \ j  \leq 2  a(\tau) |q \cdot e^{\f{\tau}{\al}}|^\al \leq j+ 1\}} \ \ \bigg( |q|^{2 (  |\ga|- 1)} e^{-2  a(\tau) |q \cdot e^{\f{\tau}{\al}}|^{\al} }  |\nabla \widehat{f}(q)|^2 dq\bigg)=\\
		&=&  I^1_2+ I^2_2.
	\end{eqnarray*}	
	For $I_2^1$, we have by H\"older's 	
	\begin{eqnarray*}
		I_2^1 &\leq&   e^{ \f{2 \tau}{\al} ( |\ga|- 1)} \int_{0 \leq |q| \leq \f{e^{\f{- \tau}{\al}}}{(2 a(\tau))^{\f{1}{\al}}}} |q|^{2( |\ga|- 1)} |\nabla \hat{f}(q)|^2 dq \leq \\
		&\leq&  C e^{ \f{2 \tau}{\al} ( |\ga|- 1)} \|\nabla \hat{f}\|^2_{L^{\f{2}{\epsilon}}} \bigg(\int_{0 \leq |r| \leq \f{e^{\f{- \tau}{\al}}}{(2 a(\tau))^{\f{1}{\al}}}}  \ \ r^{\f{2( |\ga|- 1)}{1- \epsilon}+ 1} dr \bigg)^{1- \epsilon} \\
		&=& C e^{ \f{2 \tau}{\al} ( |\ga|- 1)} \|\nabla \hat{f}\|^2_{L^{\f{2}{\epsilon}}}  \f{e^{- \f{2 \tau}{\al} ( |\ga|- 1)- \f{2 \tau}{\al} (1- \epsilon)}}{(2 a(\tau))^{ \f{2 |\ga|}{\al}- \f{2 \epsilon}{\al}}}
		\leq C \f{e^{-\f{2 \tau}{\al} (1- \epsilon)}}{(a(\tau))^{ \f{2 |\ga|}{\al}}} \|\nabla \hat{f}\|^2_{L^{\f{2}{\epsilon}}}. 
	\end{eqnarray*}
	By Sobolev embedding, we have $\|\nabla \hat{f}\|^2_{L^{\f{2}{\epsilon}}}\leq C
	\|\nabla \hat{f}\|^2_{H^{1-\eps}(\rtwo)}\leq C \|(1-\De) \hat{f}\|^2_{L^2}=C \|f\|^2_{L^2(2)}.$ Therefore 
	\begin{eqnarray*}
		I_2^1 
		\leq C \f{e^{-\f{2 \tau}{\al} (1- \epsilon)}}{(a(\tau))^{ \f{2 |\ga|}{\al}}} \|f\|^2_{L^2(2)}. 
	\end{eqnarray*}
	
	For $I_2^2$, we estimate 
	\begin{eqnarray*}
		I_2^2 &\leq&  e^{ \f{2 \tau}{\al} ( |\ga|- 1)} \sum_{j=1}^{\infty} e^{- j} \int_{ j \leq 2 a(\tau)|q e^{\f{\tau}{\al}}|^{\al} \leq  (j+1)} |q|^{2( |\ga|- 1)} |\nabla \hat{f}|^2 dq \\
		&\leq&   e^{ \f{2 \tau}{\al} ( |\ga|- 1)} \|\nabla \hat{f}\|^2_{L^{\f{2}{\epsilon}}} \sum_{j=1}^{\infty}
		e^{- j} \bigg[\int_{ j \leq 2 a(\tau)|q e^{\f{\tau}{\al}}|^{\al} \leq  (j+1)} |q|^{\f{2 ( |\ga|- 1)}{1- \epsilon}} dq \bigg]^{1- \epsilon}.
	\end{eqnarray*}
	But, 
	\begin{eqnarray*}
		&&\int_{ j \leq 2 a(\tau)|q e^{\f{\tau}{\al}}|^{\al} \leq  (j+1)} |q|^{\f{2 ( |\ga|- 1)}{1- \epsilon}} dq  \leq C \bigg(\f{j}{a(\tau)}\bigg)^{\f{2( |\ga|- 1)}{1- \epsilon}+ 2} \ \  e^{\f{- \f{2 \tau}{\al}( |\ga|- 1)}{1- \epsilon}- \f{2 \tau}{\al}},
	\end{eqnarray*}
	so using again the bound $\|\nabla \hat{f}\|_{L^{\f{2}{\epsilon}}}\leq C \|f\|_{L^2(2)}$, 
	\begin{eqnarray*}
		I_2^2 &\leq& C e^{ \f{2 \tau}{\al} ( |\ga|- 1)} \|f\|_{L^2(2)}^2 \sum_{j=1}^{\infty}
		e^{- j} \bigg[\left(\f{j}{a(\tau)}\right)^{\f{2( |\ga|- 1)}{1- \epsilon}+ 2} e^{\f{- \f{2 \tau}{\al}(\al+ |\ga|- 1)}{1- \epsilon}- \f{2 \tau}{\al}} \bigg]^{1- \epsilon}\\
		&\leq& C \f{e^{- \f{2\tau (1- \epsilon)}{\al}} \|f\|_{L^2(2)}^2}{a(\tau)^{2\f{|\ga|}{\al}-2\eps}} \sum_{j=1}^{\infty}
		e^{- j} j^{2 \al+ 2 |\ga|- 2 \epsilon} 
		\leq C \f{e^{-\f{2 \tau}{\al} (1- \epsilon)} \|f\|^2_{L^2(2)}}{a(\tau)^{ \f{2 |\ga|}{\al}}}.
	\end{eqnarray*}
	Hence after putting together the estimates for $I_2^1$ and $I_2^2$ we have
	$
	I_2 \leq C \f{e^{-\f{2 \tau}{\al} (1- \epsilon)} \|f\|^2_{L^2(2)}}{(a(\tau))^{\f{2 |\ga|}{\al}}}.
	$	
	\subsection{\textbf{Estimate for $J_3$}}				
	\begin{eqnarray*}
		J_3&=&  \int_{\rtwo} |p^{|\ga|} \Delta_p[e^{- a(\tau) |p|^{\al}}  \widehat{f}(p  e^{-\f{\tau}{\al}})]|^2 dp 
		\leq \int_{\rtwo} |p^{|\ga|} \ \Delta_p [e^{- a(\tau) |p|^{\al}}]  \ \widehat{f}(p  e^{\f{-\tau}{\al}})|^2 dp\\
		&+& 2 \int_{\rtwo} |p^{|\ga|} \ \nabla e^{- a(\tau) |p|^{\al}} \cdot \nabla ( \widehat{f}(p  e^{\f{-\tau}{\al}}) )|^2 dp
		+  \int_{\rtwo} |p^{|\ga|}\ e^{- a(\tau) |p|^{\al}}  \Delta_p(\widehat{f}(p  e^{\f{-\tau}{\al}}))|^2 dp.
	\end{eqnarray*}
	By \eqref{nabla1} we have, 
	\begin{eqnarray*}
		&&\Delta_p[ e^{- a(\tau) |p|^{\al}}]=   \sum_{j=1}^2  \partial_j \bigg ( - \al \ a(\tau) \ p_j |p|^{\al- 2} e^{- a(\tau) |p|^{\al}} \bigg )=\\
		&=& - \al \ a(\tau)\sum_{j=1}^2 \bigg ( |p|^{\al- 2} + (\al- 2) \ \f{p^2_j}{|p|} |p|^{\al- 3} + \ p_j |p|^{\al- 2} (- \al \ a(\tau)) \f{p_j}{|p|} |p|^{\al - 1} \bigg ) e^{- a(\tau) |p|^{\al}}\\
		&=& \bigg(- \al^2 \ a(\tau)  |p|^{\al- 2} + \   \al^2 \ a(\tau)^2 |p|^{2(\al - 1)} \bigg) e^{- a(\tau)  |p|^{\al}}.\\
	\end{eqnarray*}	
	Hence, by allowing for a slight abuse of notations by using $\ga$, which is a multi-index instead of $|\ga|$, its length, 
	\begin{eqnarray*}
	J_3	  &\lesssim &  a(\tau)^2  \int_{\rtwo} | \ |p|^{\al+ |\ga|- 2}e^{- a(\tau) |p|^{\al}}  \widehat{f}(p  e^{-\f{\tau}{\al}})|^2 dp +  \\
		&+& a(\tau)^4  \int_{\rtwo} | \ |p|^{2(\al- 1)+ |\ga|}e^{-  a(\tau) |p|^{\al}}  \widehat{f}(p  e^{-\f{\tau}{\al}})|^2 dp\\
		&+& a(\tau)^2   e^{- \f{2 \tau}{\al}}  \int_{\rtwo} | \ |p|^{\al+ |\ga|- 1}e^{- a(\tau) |p|^{\al}}  (\nabla \widehat{f})(p  e^{-\f{\tau}{\al}})|^2 dp +   \\ 
		&+&  e^{-\f{4 \tau}{\al}} \int_{\rtwo} |  |p|^{\ga} e^{- a(\tau) |p|^{\al}}  (\Delta \widehat{f})(p  e^{-\f{\tau}{\al}})|^2 dp\\
		&:=& I_3+ I_4+ I_5+ I_6,
	\end{eqnarray*}	
	\subsubsection{Estimate for $I_3$ and $I_4$}	 
	
	\begin{eqnarray*}
		\f{I_3}{a(\tau)^2} &\leq&   \int_{\rtwo} e^{- 2 a(\tau) |p|^{\al}} |p|^{2(\al+ |\ga|- 2)} 
		|\widehat{f}(p \cdot e^{-\f{ \tau}{\al}})|^2 dp
		= \\
		&=& e^{ \f{2 \tau}{\al} (\al+ |\ga|- 1)} \int_{\rtwo} e^{-2  a(\tau) |q \cdot e^{\f{\tau}{\al}}|^{\al} } |q|^{2 (\al+ |\ga|- 2)} |\widehat{f}(q)|^2 dq\\
		&\leq& e^{ \f{2 \tau}{\al} (\al+ |\ga|- 1)}  \int\limits_{\{q: 2  a(\tau) |q \cdot e^{\f{\tau}{\al}}|^\al  \leq 1\}}\ \ e^{-2  a(\tau) |q \cdot e^{\f{\tau}{\al}}|^{\al} } |q|^{2 (\al+ |\ga|- 2)} |\widehat{f}(q)|^2 dq \\
		&+&  e^{ \f{2 \tau}{\al} (\al+ |\ga|- 1)}\sum_{j=1}^{\infty} \int\limits_{\{q: \ j  \leq 2  a(\tau) |q \cdot e^{\f{\tau}{\al}}|^\al \leq j+ 1\}} e^{-2  a(\tau) |q \cdot e^{\f{\tau}{\al}}|^{\al} } |q|^{2 (\al+ |\ga|- 2)} |\widehat{f}(q)|^2 dq \\ 
		&=&  I^1_3+ I^2_3.
	\end{eqnarray*}	 
	By comparing $I_3$ with $I_1$ it is clear that $I^1_3= I_1^1$ and $I^2_3= I_1^2$, and we treat them in the same way. Hence 
	$
	I_3 \leq C \f{e^{-\f{2 \tau}{\al} (1- \epsilon) } \|f\|^2_{L^2(2)}}{a(\tau)^{\f{2 |\ga|}{\al}}}.
	$
	
	The estimate for $I_4$ proceeds in an   identical	manner, but we have a slightly different power of $p$, so we present it here briefly. 
	\begin{eqnarray*}
		\f{ I_4}{a(\tau)^4}  &\leq &  \int_{\rtwo} | \ |p|^{2(\al- 1)+ |\ga|} e^{-a(\tau) |p|^{\al}}  \widehat{f}(p  e^{\f{-\tau}{\al}})|^2 dp=\\
		&=&  e^{ \f{ \tau}{\al} (4\al+ 2|\ga|- 2)}   \sum_{j=0}^{\infty} \int\limits_{\{q: \ j  \leq 2  a(\tau) |q \cdot e^{\f{\tau}{\al}}|^{\al} \leq j+ 1\}} \ \ \bigg(e^{-2  a(\tau) |q \cdot e^{\f{\tau}{\al}}|^{\al} } |q|^{4 (\al- 1)+ 2 |\ga|} |\widehat{f}(q)|^2 dq\bigg):= I_4^2+ I_4^2.
	\end{eqnarray*}
	Denoting by  $I_4^1$ the integral corresponding to $2  a(\tau) |q \cdot e^{\f{\tau}{\al}}|^{\al} \leq 1$ and the rest with $I_2^2$, we have by \eqref{80}, $|\hat{f}(q)|=|\hat{f}(q)-\hat{f}(0)|\leq C|q|^{1-\eps} \|f\|_{L^2(2)}$,     
	\begin{eqnarray*}
		I_4^1 &\leq& e^{ \f{ \tau}{\al} (4\al+ 2|\ga|- 2)} \int\limits_{\{q: 2  a(\tau) |q \cdot e^{\f{\tau}{\al}}|^{\al} \leq 1\}} |q|^{4 (\al- 1)+ 2 |\ga|} |\widehat{f}(q)|^2 dq\leq \\
		&\leq& e^{ \f{ \tau}{\al} (4\al+ 2|\ga|- 2)} \|f\|^2_{L^2(2)}\int\limits_{\{q: 2  a(\tau) |q \cdot e^{\f{\tau}{\al}}|^{\al} \leq 1\}} |q|^{4 (\al- 1)+ 2 |\ga|+ 2 (1- \epsilon)} dq  \leq  C \f{e^{-\f{2 \tau}{\al} (1- \epsilon)} \|f\|^2_{L^2(2)} }{(a(\tau))^{4+ \f{2 |\ga|}{\al}}}.
	\end{eqnarray*}
	For $I_4^2$, we have 
	\begin{eqnarray*}
		I_4^2 &\leq& e^{ \f{ \tau}{\al} (4\al+ 2|\ga|- 2)}  \sum_{j=1}^{\infty} \int_{\{q: \ j  \leq 2  a(\tau) |q \cdot e^{\f{\tau}{\al}}|^{\al} \leq j+ 1\}} |q|^{4 (\al- 1)+ 2 |\ga|} e^{-2  a(\tau) |q \cdot e^{\f{\tau}{\al}}|^{\al} }  |\widehat{f}(q)|^2 dq\\
		&\leq& e^{ \f{ \tau}{\al} (4\al+ 2|\ga|- 2)}  \sum_{j=1}^{\infty} e^{- j}\int_{\{q: \ j  \leq 2  a(\tau) |q \cdot e^{\f{\tau}{\al}}|^{\al} \leq j+ 1\}} |q|^{4 (\al- 1)+ 2 |\ga|}   |\widehat{f}(q)- \widehat{f}(0)|^2 dq\\
		&\leq& e^{ \f{ \tau}{\al} (4\al+ 2|\ga|- 2)}  \|f\|^2_{L^2(2)}\sum_{j=1}^{\infty} e^{- j}\int_{\{q: \ j  \leq 2  a(\tau) |q \cdot e^{\f{\tau}{\al}}|^{\al} \leq j+ 1\}} |q|^{4 (\al- 1)+ 2 |\ga|}   |q|^{2(1- \epsilon)} dq\\
		&\leq& e^{ \f{ \tau}{\al} (4\al+ 2|\ga|- 2)}  \|f\|^2_{L^2(2)} \f{e^{- \f{\tau}{\al} (4 \al+ 2|\ga|- 2 \epsilon)}}{(a(\tau))^{4+ \f{2 |\ga|}{\al}- \f{ 2 \epsilon}{\al}}} \sum_{j=1}^{\infty} e^{- j} j^{2(2 \al+ |\ga|- \epsilon)} \leq  C \  \f{e^{- \f{2 \tau}{\al} (1- \epsilon)}}{(a(\tau))^{4+ \f{2 |\ga|}{\al}}}  \|f\|^2_{L^2(2)}.
	\end{eqnarray*}
	Therefore 
	$
	I_4 \leq C \ \f{ e^{- \f{2 \tau}{\al} (1- \epsilon)}  \|f\|^2_{L^2(2)}}{(a(\tau))^\f{2 |\ga|}{\al}}. 
	$
	\subsection{Estimate for $I_5$} 
	\begin{eqnarray*}
		\f{ I_5}{a(\tau)^2} &\leq&  C e^{-\f{2 \tau}{\al}} \int_{\rtwo} | \ |p|^{\al+ |\ga|- 1}e^{- a(\tau) |p|^{\al}}  (\nabla \widehat{f})(p  e^{\f{-\tau}{\al}})|^2 dp= \\
		&=& e^{ \f{2 \tau}{\al} (\al+ |\ga|- 1)} \int_{\rtwo} |q|^{2 (\al+ |\ga|- 1)} e^{-2  a(\tau) |q \cdot e^{\f{\tau}{\al}}|^{\al} }  |\nabla \widehat{f}(q)|^2 dq\\
		&\leq&  e^{ \f{2 \tau}{\al} (\al+ |\ga|- 1)}   \sum_{j=0}^{\infty} 
		\int\limits_{\{q: \ j  \leq 2  a(\tau) |q \cdot e^{\f{\tau}{\al}}|^\al \leq j+ 1\}} \ \ \bigg( |q|^{2 (\al+ |\ga|- 1)} e^{-2  a(\tau) |q \cdot e^{\f{\tau}{\al}}|^{\al} }  |\nabla \widehat{f}(q)|^2 dq\bigg)=\\
		&=&  I^1_5+ I^2_5.
	\end{eqnarray*}	
	For $I_5^1$, we have by H\"older's 	
	\begin{eqnarray*}
		I^1_5 &\leq&   e^{ \f{2 \tau}{\al} (\al+ |\ga|- 1)} \int_{|q| \leq \f{e^{\f{- \tau}{\al}}}{(2 a(\tau))^{\f{1}{\al}}}} |q|^{2(\al+ |\ga|- 1)} |\nabla \hat{f}(q)|^2 dq \leq \\
		&=& C e^{ \f{2 \tau}{\al} (\al+ |\ga|- 1)} \|\nabla \hat{f}\|^2_{L^{\f{2}{\epsilon}}}  \f{e^{- \f{2 \tau}{\al} (\al+ |\ga|- 1)- \f{2 \tau}{\al} (1- \epsilon)}}{(2 a(\tau))^{2+ \f{2 |\ga|}{\al}- \f{2 \epsilon}{\al}}}
		\leq C \f{e^{-\f{2 \tau}{\al} (1- \epsilon)}}{(a(\tau))^{2+  \f{2 |\ga|}{\al}}} \|\nabla \hat{f}\|^2_{L^{\f{2}{\epsilon}}}. 
	\end{eqnarray*}
	However, by Sobolev embedding, we have as before $\|\nabla \hat{f}\|^2_{L^{\f{2}{\epsilon}}}\leq  C \|f\|^2_{L^2(2)}.$ 
	
	For $I_5^2$, we estimate 
	\begin{eqnarray*}
		I^2_5 &\leq&  e^{ \f{2 \tau}{\al} (\al+ |\ga|- 1)} \sum_{j=1}^{\infty} e^{- j} \int_{ j \leq 2 a(\tau)|q e^{\f{\tau}{\al}}|^{\al} \leq  (j+1)} |q|^{2(\al+ |\ga|- 1)} |\nabla \hat{f}|^2 dq \\
		&\leq&   e^{ \f{2 \tau}{\al} (\al+ |\ga|- 1)} \|\nabla \hat{f}\|^2_{L^{\f{2}{\epsilon}}} \sum_{j=1}^{\infty}
		e^{- j} \bigg[\int_{ j \leq 2 a(\tau)|q e^{\f{\tau}{\al}}|^{\al} \leq  (j+1)} |q|^{\f{2 (\al+ |\ga|- 1)}{1- \epsilon}} dq \bigg]^{1- \epsilon}.
	\end{eqnarray*}
	But, 
	\begin{eqnarray*}
		&&\int_{ j \leq 2 a(\tau)|q e^{\f{\tau}{\al}}|^{\al} \leq  (j+1)} |q|^{\f{2 (\al+ |\ga|- 1)}{1- \epsilon}} dq  \leq C \bigg(\f{j}{a(\tau)}\bigg)^{\f{2(\al+ |\ga|- 1)}{1- \epsilon}+ 2} \ \  e^{\f{- \f{2 \tau}{\al}(\al+ |\ga|- 1)}{1- \epsilon}- \f{2 \tau}{\al}},
	\end{eqnarray*}
	so using again the bound $\|\nabla \hat{f}\|_{L^{\f{2}{\epsilon}}}\leq C \|f\|_{L^2(2)}$, 
	\begin{eqnarray*}
		I_5^2 &\leq& C e^{ \f{2 \tau}{\al} (\al+ |\ga|- 1)} \|f\|_{L^2(2)}^2 \sum_{j=1}^{\infty}
		e^{- j} \bigg[\left(\f{j}{a(\tau)}\right)^{\f{2(\al+ |\ga|- 1)}{1- \epsilon}+ 2} e^{\f{- \f{2 \tau}{\al}(\al+ |\ga|- 1)}{1- \epsilon}- \f{2 \tau}{\al}} \bigg]^{1- \epsilon}\\
		&\leq& C \f{e^{- \f{2\tau (1- \epsilon)}{\al}} \|f\|_{L^2(2)}^2}{a(\tau)^{2+2\f{|\ga|}{\al}-2\eps}} \sum_{j=1}^{\infty}
		e^{- j} j^{2 \al+ 2 |\ga|- 2 \epsilon} 
		\leq C \f{e^{-\f{2 \tau}{\al} (1- \epsilon)} \|f\|^2_{L^2(2)}}{a(\tau)^{2+ \f{2 |\ga|}{\al}}}.
	\end{eqnarray*}
	Hence after putting together the estimates for $I_5^1$ and $I_5^2$ we have
	$
	I_5 \leq C \f{e^{-\f{2 \tau}{\al} (1- \epsilon)} \|f\|^2_{L^2(2)}}{(a(\tau))^{\f{2 |\ga|}{\al}}}.
	$
	\subsection{Estimate for $I_6$}
	In the same way we can get the estimate for $I_6$. Indeed,
	\begin{eqnarray*}
		I_6 &\leq&  e^{-\f{4 \tau}{\al}} \int_{\rtwo} \bigg| \ |p|^{ |\ga|} e^{- a(\tau) |p|^{\al}}  (\De \widehat{f})(p  e^{\f{-\tau}{\al}})\bigg|^2 dp= e^{ \f{2 \tau}{\al} (  |\ga|- 1)} \int_{\rtwo} |q|^{2 |\ga|} e^{-2  a(\tau) |q \cdot e^{\f{\tau}{\al}}|^{\al} }  |\De \widehat{f}(q)|^2 dq\\
		&\leq&  e^{ \f{2 \tau}{\al} (  |\ga|- 1)} \bigg[\int_{\{q: 2  a(\tau) |q \cdot e^{\f{\tau}{\al}}| \leq 1\}}+ \sum_{j=1}^{\infty} \int_{\{q: \ j  \leq 2  a(\tau) |q \cdot e^{\f{\tau}{\al}}| \leq j+ 1\}} \bigg]\ \ \bigg( |q|^{2 |\ga|} e^{-2  a(\tau) |q \cdot e^{\f{\tau}{\al}}|^{\al} }  |\De \widehat{f}(q)|^2 dq\bigg)\\
		&=&  I^1_6+ I^2_6.
	\end{eqnarray*}	
	For $I_6^1$, 	
	\begin{eqnarray*}
		I^1_6 &\leq&   e^{ \f{2 \tau}{\al} (  |\ga|- 1)} \int_{|q| \leq \f{e^{-\f{ \tau}{\al}}}{(2 a(\tau))^{\f{1}{\al}}}} |q|^{2 |\ga|} |\De \hat{f}(q)|^2 dq
		=  \f{e^{ \f{2 \tau}{\al} (  |\ga|- 1)} e^{- \f{2 \tau |\ga|}{\al}}}{(a(\tau))^{\f{2 |\ga|}{\al}}}  \int_{|q| \leq \f{e^{-\f{ \tau}{\al}}}{(2 a(\tau))^{\f{1}{\al}}}}  \ |\De \hat{f}(q)|^2 dq \leq \\
		&\leq & 
		C \f{e^{- \f{2 \tau}{\al}}}{(a(\tau))^{\f{2 |\ga|}{\al}}}  \|f\|^2_{L^2(2)}. 
	\end{eqnarray*}
	For $I_6^2$, we have 
	\begin{eqnarray*}
		I^2_6 &\leq&  e^{ \f{2 \tau}{\al} (  |\ga|- 1)} \sum_{j=1}^{\infty} e^{- j} \int_{ j \leq 2 a(\tau)|q e^{\f{\tau}{\al}}|^{\al} \leq  (j+1)} |q|^{2 |\ga|} |\De \hat{f}(q)|^2 dq \\
		&\leq&   \f{e^{ \f{2 \tau}{\al} (  |\ga|- 1)} e^{- \f{2 \tau |\ga|}{\al}}}{a(\tau)^{\f{2 |\ga|}{\al}}} \sum_{j=1}^{\infty}
		e^{- j} (j+1)^{\f{2 |\ga|}{\al}} \int  |\De \hat{f}(q)|^2 dq 
		\leq C \f{e^{- \f{2 \tau}{\al}}}{(a(\tau))^{\f{2 |\ga|}{\al}}}  \|f\|^2_{L^2(2)}.
	\end{eqnarray*}
	Therefore,
	$$
	I_6 \leq C \f{e^{- \f{2 \tau}{\al}}}{(a(\tau))^{\f{2 |\ga|}{\al}}}  \|f\|^2_{L^2(2)} \leq C \f{e^{- \f{2 \tau}{\al}(1- \epsilon)}}{(a(\tau))^{\f{2 |\ga|}{\al}}}  \|f\|^2_{L^2(2)}.
	$$
	Putting it all together finishes off the proof.  	  		
\end{proof}


\end{document}